\documentclass[12pt]{article}

\title{$S^1$-equivariant contact homology for hypertight contact forms}
\author{Michael Hutchings\footnote{Partially supported by NSF grants DMS-1406312 and DMS-1708899.}\;  and Jo Nelson\footnote{Partially supported by NSF grants DMS-1303903 and DMS-1810692.}\;}
\date{}

\usepackage{amssymb}
\usepackage{latexsym}
\usepackage{amsmath}
\usepackage{amsthm}
\usepackage{relsize}
\usepackage[normalem]{ulem}
\usepackage[matha,  mathx]{mathabx}
\usepackage{marvosym}
\usepackage{wasysym}
\usepackage{scalerel}
\usepackage{mathrsfs}

\usepackage{amscd}
\usepackage{overpic}
\usepackage{enumerate}
\usepackage{color}
\usepackage[usenames,dvipsnames,svgnames,table]{xcolor}
\usepackage{graphicx}

\usepackage{hyperref}
\hypersetup{colorlinks}

\definecolor{indigo}{RGB}{51,0,102}
\definecolor{brightpurple}{RGB}{102,0,153}
\definecolor{fuchsia}{RGB}{180,51,180}
\definecolor{jolightpurple}{RGB}{188,171,240}

\hypersetup{colorlinks,
linkcolor=brightpurple,
filecolor=brightpurple,
urlcolor=indigo,
citecolor=fuchsia}

\usepackage[margin=1.25in]{geometry}
\usepackage{dsfont}

\normalbaselines
\newcommand{\mc}[1]{{\mathcal #1}}

\numberwithin{equation}{section}

\newtheorem{theorem}{Theorem}[section]
\newtheorem{proposition}[theorem]{Proposition}
\newtheorem{corollary}[theorem]{Corollary}
\newtheorem{lemma}[theorem]{Lemma}
\newtheorem{lemma-definition}[theorem]{Lemma-Definition}

\theoremstyle{definition}
\newtheorem{definition}[theorem]{Definition}
\newtheorem{remark}[theorem]{Remark}

\newtheorem{example}[theorem]{Example}

\newcommand{\eqdef}{\;{:=}\;}

\renewcommand{\frak}{\mathfrak}

\newcommand{\bbC}{{\mathbb C}}

\newcommand{\Q}{{\mathbb Q}}
\newcommand{\R}{{\mathbb R}}

\newcommand{\Z}{{\mathbb Z}}

\newcommand{\op}{\operatorname}
\newcommand{\dbar}{\overline{\partial}}

\newcommand{\M}{\mc{M}}

\newcommand{\Ker}{\op{Ker}}
\newcommand{\Coker}{\op{Coker}}

\newcommand{\tensor}{\otimes}

\newcommand{\CZ}{\op{CZ}}

\newcommand{\bpm}{\begin{pmatrix}}
\newcommand{\epm}{\end{pmatrix}}

\newcommand{\J}{\mathbb{J}}

\newcommand{\Mt}{\widetilde{\M}}

\newcommand{\ca}{{\mbox{\lightning}}}

\begin{document}

\maketitle

\begin{abstract}
In a previous paper, we showed that the original definition of cylindrical contact homology, with rational coefficients, is valid on a closed three-manifold with a dynamically convex contact form. However we did not show that this cylindrical contact homology is an invariant of the contact structure.

In the present paper, we define ``nonequivariant contact homology'' and ``$S^1$-equivariant contact homology'', both with integer coefficients, for a contact form on a closed manifold in any dimension with no contractible Reeb orbits. We prove that these contact homologies depend only on the contact structure. Our construction uses Morse-Bott theory and is related to the positive $S^1$-equivariant symplectic homology of Bourgeois-Oancea. However, instead of working with Hamiltonian Floer homology, we work directly in contact geometry, using families of almost complex structures. When  cylindrical contact homology can also be defined, it agrees with the tensor product of the $S^1$-equivariant contact homology with $\Q$. We also present examples showing that the $S^1$-equivariant contact homology contains interesting torsion information.

In a subsequent paper we will use obstruction bundle gluing to extend the above story to closed three-manifolds with dynamically convex contact forms, which in particular will prove that their cylindrical contact homology has a lift to integer coefficients which depends only on the contact structure.


\end{abstract}

\tableofcontents





\section{Introduction and statement of results}

Let $Y$ be a closed odd-dimensional manifold with a nondegenerate contact form $\lambda$. This paper is concerned with the foundations of three kinds of contact homology of $(Y,\lambda)$, each of which, when defined, depends only on the contact structure $\xi=\Ker(\lambda)$:

\medskip

(1) {\em Cylindrical contact homology\/} as defined by Eliashberg-Givental-Hofer \cite{egh}, which we denote by $CH^{EGH}_*(Y,\lambda;J)$. In the absence of certain contractible Reeb orbits, this is the homology of a chain complex over $\Q$ which is generated by ``good'' Reeb orbits. The differential, which we denote by $\partial^{EGH}$, counts $J$-holomorphic cylinders in $\R\times Y$, where $J$ is a generic ``$\lambda$-compatible'' almost complex structure on $\R\times Y$.

In general, it is not possible to obtain sufficient transversality for $J$-holomorphic cylinders to define this theory, even with generic $J$, so some abstract perturbations are needed. However in our previous paper \cite{dc}, we showed that in the three-dimensional case, for dynamically convex\footnote{A nondegenerate contact form $\lambda$ on a three-manifold $Y$ is called {\em dynamically convex\/} if there are no contractible Reeb orbits, or the following two conditions hold: (1) $c_1(\xi)$ vanishes on $\pi_2(Y)$, so that each contractible Reeb orbit $\gamma$ has a well-defined Conley-Zehnder index $\CZ(\gamma)\in\Z$; and (2) each contractible Reeb orbit $\gamma$ has $\CZ(\gamma)\ge 3$. In \cite{dc} we made the additional hypothesis that a contractible Reeb orbit $\gamma$ has $\CZ(\gamma)=3$ only if it is embedded; this assumption can be dropped by \cite{chz}.} contact forms, if $J$ is generic then the differential $\partial^{EGH}$ is in fact well-defined and satisfies $(\partial^{EGH})^2=0$. Thus for dynamically convex contact forms in three dimensions, for generic $J$ we have a well-defined homology $CH^{EGH}_*(Y,\lambda;J)$. 

Continuing the work in \cite{dc}, the next step is to show that this homology depends only on $\xi$ and not on $J$, and more generally to define maps on cylindrical contact homology induced by appropriate symplectic cobordisms. A natural approach would be to define a cobordism map by counting $J$-holomorphic cylinders in the cobordism for a generic ``cobordism-compatible'' almost complex structure $J$. However even for cobordisms between dynamically convex contact forms on $S^3$, sometimes there does not exist $J$ satisfying sufficient transversality; see \cite[Ex.\ 1.26]{jo1} for an example arising from an inclusion of four-dimensional ellipsoids.

Instead, we will bring in two new ingredients: Morse-Bott theory, and obstruction bundle gluing. The present paper explains the Morse-Bott part, which suffices to prove invariance in the case when $Y$ is three dimensional and $\lambda$ is {\em hypertight\/}, meaning that $\lambda$ has no contractible Reeb orbits. That is, if $\lambda'$ is another hypertight contact form on $Y$ with $\Ker(\lambda)=\Ker(\lambda')$, and if $J'$ is a generic $\lambda'$-compatible almost complex structure, then there is a canonical isomorphism\footnote{In this paper we denote canonical isomorphisms by an equals sign.}
\begin{equation}
\label{eqn:eghinv}
CH_*^{EGH}(Y,\lambda;J) = CH_*^{EGH}(Y,\lambda';J').
\end{equation}
In fact, the cylindrical contact homology $CH_*^{EGH}$ has an integral lift which is also an invariant of $\xi$; see equation \eqref{eqn:integrallift} below. In the sequel \cite{inv}, we will use obstruction bundle gluing to extend this result (the existence of an invariant integral lift of cylindrical homology) from the hypertight case to the dynamically convex case in three dimensions.

\medskip

(2) {\em Nonequivariant contact homology}, which we denote by $NCH_*(Y,\lambda;\J)$. This theory, which is defined over $\Z$, is a stepping stone to proving invariance of the cylindrical contact homology $CH_*^{EGH}$, and it has interest in its own right. In this paper we define nonequivariant contact homology $NCH_*$ for closed manifolds $Y$ of arbitrary odd dimension, assuming that $\lambda$ is hypertight.

The idea, combining ingredients from \cite{bce,bee,boduke}, is to count $\J$-holomorphic cylinders in $\R\times Y$ between Reeb orbits, where $\J$ is an almost complex structure on $\R\times Y$ which now depends on the $S^1$ coordinate on the domain. Breaking the $S^1$-symmetry this way eliminates the transversality problems in defining $\partial^{EGH}$, and gives us transverse moduli spaces of $\J$-holomorphic cylinders for generic $\lambda$-compatible $\J$. However the gluing theory to prove that $(\partial^{EGH})^2=0$ does not carry over to this situation to give a chain complex with one generator for each (good) Reeb orbit; see Remark~\ref{rem:jowantsthis}. To define a chain complex in this situation, we need two generators for each (good or bad) Reeb orbit $\alpha$, which we denote by $\widecheck{\alpha}$ and $\widehat{\alpha}$. The differential counts ``Morse-Bott cascades'' built out of $\J$-holomorphic cylinders, using the algebraic formalism in \cite{td}. We then obtain a well-defined homology $NCH_*(Y,\lambda;\J)$, which we call ``nonequivariant contact homology''. We also prove that if $\lambda'$ is another nondegenerate hypertight contact form on $Y$ with $\Ker(\lambda)=\Ker(\lambda')$, and if $\J'$ is a generic $S^1$-family of $\lambda'$-compatible almost complex structures, then there is a canonical isomorphism
\[
NCH_*(Y,\lambda;\J) = NCH_*(Y,\lambda';\J').
\]

\medskip

(3) {\em $S^1$-equivariant contact homology\/}, which we denote by $CH_*^{S^1}(Y,\lambda;\frak{J})$. This homology is also defined over $\Z$, and to define it we again assume that $Y$ is a closed manifold of arbitrary odd dimension and $\lambda$ is hypertight. Equivariant contact homology is a ``family'' version of nonequivariant contact homology, which is defined using a larger family $\mathfrak{J}$ of $\lambda$-compatible almost complex structures on $\R\times Y$, following ideas of \cite{bo12,sesm}, adapted to the contact setting. Roughly speaking, $\mathfrak{J}$ is a $BS^1$-family of $S^1$-families of almost complex structures $\J$. More precisely, $\mathfrak{J}$ is a generic $S^1$-equivariant family of almost complex structures on $\R\times Y$ parametrized by $S^1\times ES^1$.

The $S^1$-equivariant contact homology is the homology of a chain complex whose generators have the form $\widecheck{\alpha}\tensor U^k$ and $\widehat{\alpha}\tensor U^k$, where $k$ is a nonnegative integer, $U$ is a formal variable, and $\alpha$ is a Reeb orbit. Here $U^k$ corresponds to the index $2k$ critical point of a perfect Morse function on $BS^1$. The differential counts holomorphic cylinders in $\R\times Y$ which are ``coupled'' to Morse flow lines on $BS^1$. We denote the resulting homology by $CH^{S^1}_*(Y,\lambda;\mathfrak{J})$. We prove that if $\lambda'$ is another hypertight contact form on $Y$ with $\Ker(\lambda)=\Ker(\lambda')$, and if ${\mathfrak J}'$ is a generic family of $\lambda'$-compatible almost complex structures, then there is a canonical isomorphism
\begin{equation}
\label{eqn:chs1inv}
CH^{S^1}_*(Y,\lambda;\mathfrak{J}) = CH_*^{S^1}(Y,\lambda';\mathfrak{J}').
\end{equation}

Returning to the original goal, we show that if $\lambda$ is hypertight, if $J$ is a $\lambda$-compatible almost complex structure on $\R\times Y$ satisfying sufficient transversality for $J$-holomorphic cylinders to define cylindrical contact homology (which can always be achieved in the three-dimensional case), and if we set $\mathfrak{J}$ to be the constant family given by $J$, then there is a canonical isomorphism
\begin{equation}
\label{eqn:integrallift}
CH^{S^1}_*(Y,\lambda;\mathfrak{J})\tensor\Q = CH_*^{EGH}(Y,\lambda;J).
\end{equation}
Combining this with the topological invariance of equivariant contact homology \eqref{eqn:chs1inv}, we obtain the desired topological invariance of cylindrical contact homology \eqref{eqn:eghinv} in the hypertight case.

Although hypertight contact forms are somewhat special, one application of the present paper is to give a rigorous definition of the ``local contact homology'' from \cite{hm} of a degenerate Reeb orbit; see \S\ref{sec:addstr} and \S\ref{sec:localcontact}. One can also obtain many examples of hypertight contact forms from taut foliations on three-manifolds \cite{ch,zung}.

The nonequivariant and $S^1$-equivariant contact homology described above will be extended to the dynamically convex case in three dimensions in \cite{inv}.

\subsection{Contact preliminaries}
\label{sec:contactprelim}

To explain the above story in more detail, we first recall some basic definitions. Let $Y$ be a closed odd-dimensional manifold with a nondegenerate contact form $\lambda$. Let $\xi=\Ker(\lambda)$ denote the associated contact structure, and let $R$ denote the Reeb vector field determined by $\lambda$.

A {\em Reeb orbit\/} is a map $\gamma:\R/T\Z\to Y$, for some $T>0$, such that $\gamma'(t)=R(\gamma(t))$. We consider two Reeb orbits to be equivalent if they differ by a translation of the domain. We do not assume that $\gamma$ is an embedding; every Reeb orbit is a $d$-fold cover of an embedded Reeb orbit for some positive integer $d$.  For a Reeb orbit as above, the linearized Reeb flow for time $T$ defines a symplectic linear map
\begin{equation}
\label{slm}
P_\gamma : \left( \xi_{\gamma(0)}, d\lambda \right) \to  \left( \xi_{\gamma(0)}, d\lambda \right).
\end{equation}
The Reeb orbit $\gamma$ is \emph{nondegenerate} if $P_\gamma$ does not have 1 as an eigenvalue.  The contact form $\lambda$ is called \emph{nondegenerate} if all Reeb orbits are nondegenerate; generic contact forms have this property.

\begin{definition}
\label{def:lambdacompatible}
An almost complex structure $J$ on ${\mathbb R}\times Y$ is called {\em $\lambda$-compatible\/} if $J(\partial_r)=R$, where $r$ denotes the $\R$ coordinate; $J$ sends $\xi=\Ker(\lambda)$ to itself, compatibly with the linear symplectic form $d\lambda$ on $\xi$; and $J$ is invariant under translation of the $\R$ factor on $\R\times Y$.
\end{definition}

Fix a $\lambda$-compatible almost complex structure $J$, and let $\gamma_+$ and $\gamma_-$ be Reeb orbits. We consider maps $u:\R\times S^1\to\R\times Y$ such that
\begin{equation}
\label{eqn:Floerautonomous}
\partial_s u + J \partial _t u = 0,
\end{equation}
$\lim_{s\to\pm\infty}\pi_\R(u(s,t))=\pm\infty$, and $\lim_{s\to\pm\infty}\pi_Y(u(s,\cdot))$ is a parametrization of $\gamma_\pm$. Here $\pi_\R$ and $\pi_Y$ denote the projections from $\R\times Y$ to $\R$ and $Y$ respectively. We declare two such maps to be equivalent if they differ by translation of the $\R$ and $S^1$ coordinates on the domain $\R\times S^1$, and we denote the set of equivalence classes by $\widetilde{\M}^J(\gamma_+,\gamma_-)$.

Given $u$ as above, we define its {\em Fredholm index\/} by
\begin{equation}
\label{eqn:Fredholmindex}
\op{ind}(u) = \CZ_\tau(\gamma_+) - \CZ_\tau(\gamma_-) + 2c_1(u^*\xi,\tau).
\end{equation}
Here $\tau$ is a symplectic trivialization of $\gamma_+^*\xi$ and $\gamma_-^*\xi$, while  $\CZ_\tau(\gamma_\pm)\in\Z$ is the Conley-Zehnder index of $\gamma_\pm$ with respect to $\tau$, and $c_1(u^*\xi,\tau)$ denotes the relative first Chern class of $u^*\xi$ with respect to $\tau$, which vanishes if and only if $\tau$ extends to a trivialization of $u^*\xi$. If $J$ is generic and $u$ is somewhere injective, then $\widetilde{\M}^J(\gamma_+,\gamma_-)$ is a manifold near $u$ of dimension $\op{ind}(u)$. This is a special case of a more general result for holomorphic curves that are not necessarily cylinders (where the index formula includes an additional Euler characteristic term) which is proved in \cite{dragnev} and explained in more detail in \cite[Thms. 5.4 and 8.1]{wendl}.

Note that $\R$ acts on $\widetilde{\M}^J(\gamma_+,\gamma_-)$ by translation of the $\R$ factor in the target $\R\times Y$. We define
\begin{equation}
\label{eqn:modholcyl}
\M^J(\gamma_+,\gamma_-) = \widetilde{\M}^J(\gamma_+,\gamma_-)/\R.
\end{equation}
Let $\M^J_d(\gamma_+,\gamma_-)$ denote the set of $u\in\M^J(\gamma_+,\gamma_-)$ with Fredholm index $\op{ind}(u)=d$.

Recall that if $\gamma$ is a Reeb orbit and $\tau$ is a trivialization of $\gamma^*\xi$, then the parity of the Conley-Zehnder index $\CZ_\tau(\gamma)$ does not depend on $\tau$. Thus every Reeb orbit $\gamma$ has a well-defined mod $2$ Conley-Zehnder index $\CZ(\gamma)\in\Z/2$. A Reeb orbit $\gamma$ is called {\em bad\/} if it is a (necessarily even degree) multiple cover of a Reeb orbit $\gamma'$ such that
\[
\CZ(\gamma) \neq \CZ(\gamma') \in \Z/2.
\]
Otherwise, $\gamma$ is called {\em good\/}.

\subsection{Cylindrical contact homology}
\label{sec:cchintro}

We now review what we will need to know about the cylindrical contact homology $CH_*^{EGH}(Y,\lambda;J)$. The original definition is due to Eliashberg-Givental-Hofer \cite{egh}; we are using notation\footnote{A notational difference is that in \cite{dc}, we denoted cylindrical contact homology by $CH^\Q(Y,\lambda,J)$.} from \cite{dc}.

Assume that $\lambda$ is nondegenerate and hypertight. Let $J$ be a $\lambda$-compatible almost complex structure on $\R\times Y$. Assuming that $J$ satisfies certain transversality conditions (to be specified below), we define a chain complex $CC_*^{EGH}(Y,\lambda;J)$ over $\Q$ as follows.

As a module, $CC_*^{EGH}(Y,\lambda;J)$ is noncanonically isomorphic to the vector space over $\Q$ generated by good Reeb orbits; an isomorphism is fixed by making certain orientation choices. More precisely, for each good Reeb orbit $\gamma$, the theory of coherent orientations as in \cite{bm,fh} can be used to define a $\Z$-module $\mc{O}_\gamma$ which is noncanonically isomorphic to $\Z$; see Proposition~\ref{prop:orientations} and \S\ref{sec:oms}. We then define
\[
CC_*^{EGH}(Y,\lambda;J) = \bigoplus_{\mbox{\scriptsize $\gamma$ good}}\mc{O}_\gamma\tensor_\Z\Q.
\]
Choosing a generator of $\mc{O}_\gamma$ for each good Reeb orbit $\gamma$ specifies an isomorphism
\[
CC_*^{EGH}(Y,\lambda;J) \simeq \Q\{\mbox{good Reeb orbits}\}.
\]

This chain complex has a canonical $\Z/2$-grading determined by the mod $2$ Conley-Zehnder index\footnote{It is common in the literature to instead define the grading on cylindrical contact homology to be the Conley-Zehnder index plus $1-n$, where $\dim(Y)=2n-1$.}. In some cases the grading can be refined; see \S\ref{sec:addstr} below.

To define the differential, we first define an operator
\[
\delta: CC_*^{EGH}(Y,\lambda;J) \longrightarrow CC_{*-1}^{EGH}(Y,\lambda;J)
\]
as follows: If $\alpha$ is a good Reeb orbit, then
\begin{equation}
\label{eqn:defdelta}
\delta\alpha = \sum_\beta \sum_{u\in\M_1^J(\alpha,\beta)}\frac{\epsilon(u)}{d(u)}\beta,
\end{equation}
where the sum is over good Reeb orbits $\beta$. Here $\epsilon(u)\in\{\pm1\}$ is a sign\footnote{More precisely, $\epsilon(u)$ is an element of $\{\pm1\}$ after generators of $\mc{O}_\alpha$ and $\mc{O}_\beta$ have been chosen. Without making such choices, $\epsilon(u)$ is an isomorphism $\mc{O}_\alpha\simeq \mc{O}_\beta$.} associated to $u$; our sign convention is spelled out in Definition~\ref{def:EGHsigns}. Also, $d(u)\in\Z^{>0}$ is the covering multiplicity of $u$, which is $1$ if and only if $u$ is somewhere injective. The definition \eqref{eqn:defdelta} makes sense provided that all moduli spaces $\M^J_d(\alpha,\beta)$ with Fredholm index $d\le 1$ are cut out transversely\footnote{In particular, then all moduli spaces $\M^J_d(\alpha,\beta)$ with $\alpha\neq\beta$ and $d\le 0$ are empty, which under our hypertightness assumption guarantees that the moduli spaces $\M^J_1(\alpha,\beta)$ are compact so that we obtain finite counts.}.

Next we define an operator
\[
\kappa: CC_*^{EGH}(Y,\lambda;J) \longrightarrow CC_*^{EGH}(Y,\lambda;J)
\]
by
\[
\kappa(\alpha) = d(\alpha)\alpha.
\]
If we further assume suitable transversality for the moduli spaces $\M^J_2(\alpha,\beta)$, then counting their ends leads to the equation
\begin{equation}
\label{eqn:dkd0}
\delta\kappa\delta = 0.
\end{equation}
This was proved in the three-dimensional case in \cite{dc}, and we will recover it in arbitrary odd dimensions from the Morse-Bott theory below; see Corollary~\ref{cor:degh20}. Equation \eqref{eqn:dkd0} implies that
\begin{equation}
\label{cyldiff}
\partial^{EGH}:=\delta\kappa
\end{equation}
is a differential on $CC_*^{EGH}(Y,\lambda;J)$. 

\begin{definition}
\label{def:cch}
If $\lambda$ is hypertight and $J$ is admissible (see Definition~\ref{def:admissible}; this is a certain transversality hypothesis on the moduli spaces $\M^J_d(\alpha,\beta)$ for $d\le 2$), we define the {\em cylindrical contact homology\/} $CH_*^{EGH}(Y,\lambda;J)$ to be the homology of the chain complex $(CC_*^{EGH}(Y,\lambda;J),\partial^{EGH})$.
\end{definition}

\begin{remark}
It is also possible to take the differential to be $\kappa\delta$ instead of $\delta\kappa$. In fact, both of these differentials arise naturally in the Morse-Bott story; see equation \eqref{eqn:tdan} below. The operator $\kappa$ defines an isomorphism between these two chain complexes over $\Q$, because $(\kappa\delta)\kappa = \kappa(\delta\kappa)$. While both of these differentials are actually defined over $\Z$, we do not expect the homologies over $\Z$ to be isomorphic to each other or invariant in the sense of \eqref{eqn:eghinv}.
\end{remark}

In \cite{dc} we showed that in the three-dimensional case, the transversality for $J$-holomorphic cylinders needed to define $CH_*^{EGH}(Y,\lambda;J)$ can be achieved by choosing $J$ generically; see also \S\ref{sec:admissible3} below. However this is impossible in most higher dimensional cases. The difficulty is that there may exist multiply covered $J$-holomorphic cylinders with negative Fredholm index, even when $J$ is generic.

\subsection{Nonequivariant contact homology}
\label{sec:nchintro}

As suggested in \cite{bce}, one can fix the transversality problems for holomorphic cylinders using a domain-dependent almost complex structure.  Breaking the $S^1$ symmetry naturally leads one to a ``Morse-Bott" version of the chain complex.  The homology of this chain complex is not the cylindrical contact homology described in the previous section, but rather a ``non-equivariant" version of it, which we define in \S\ref{section:plainmoduli} and \S\ref{cascades}.

To introduce this, let $Y$ be a closed odd-dimensional manifold, and let $\lambda$ be a nondegenerate hypertight contact form on $Y$. Let $\J=\{J_t\}$ be a family of $\lambda$-compatible almost complex structures on $\R\times Y$ parametrized by $t\in S^1$. If $\gamma_+$ and $\gamma_-$ are Reeb orbits, we consider maps $u:\R\times S^1\to \R\times Y$ such that
\begin{equation}
\label{eqn:Floerdd}
\partial_s u + J_t\partial_t u = 0,
\end{equation}
$\lim_{s\to\pm\infty}\pi_\R(u(s,t))=\pm\infty$, and $\lim_{s\to\pm\infty}\pi_Y(u(s,\cdot))$ is a parametrization of $\gamma_\pm$. We declare two such maps to be equivalent if they differ by translation of the $\R$ coordinate on the domain $\R\times S^1$, and we denote the set of equivalence classes by $\widetilde{\M}^\J(\gamma_+,\gamma_-)$. Note that for solutions to \eqref{eqn:Floerdd}, unlike \eqref{eqn:Floerautonomous}, we can no longer mod out by rotation of the $S^1$ coordinate on the domain.

Given $u$ as above, let $\widetilde{\M}^\J_u(\gamma_+,\gamma_-)$ denote the component of $\widetilde{\M}^\J(\gamma_+,\gamma_-)$ containing $u$.
If $\J$ is generic, then this is a smooth manifold of dimension 
\[
\dim\left(\widetilde{\M}^\J_u(\gamma_+,\gamma_-)\right) = \CZ_\tau(\gamma_+) - \CZ_\tau(\gamma_-) + 2c_1(u^*\xi,\tau) + 1.
\]
The right hand side here is one greater than the right hand side of \eqref{eqn:Fredholmindex}, because we are no longer modding out by an $S^1$ symmetry.

As before, $\R$ acts on $\widetilde{\M}^\J(\gamma_+,\gamma_-)$ by translation of the $\R$ factor in the target $\R\times Y$, and we let $\M^\J(\gamma_+,\gamma_-)$ denote the quotient. Below, if $\gamma_+\neq\gamma_-$, and if $d$ is a nonnegative integer, let $\M^\J_d(\gamma_+,\gamma_-)$ denote the union of the $d$-dimensional components of $\M^\J(\gamma_+,\gamma_-)$. 

We now also have well-defined smooth evaluation maps
\[
\begin{split}
e_\pm: \M^\J(\gamma_+,\gamma_-) & \longrightarrow \overline{\gamma_\pm},\\
u & \longmapsto \lim_{s\to\pm\infty}\pi_Y(u(s,0)).
\end{split}
\]
Here $\overline{\gamma}$ denotes the image of the Reeb orbit $\gamma$ in $Y$.

The moduli spaces $\M^\J_d(\gamma_+,\gamma_-)$, together with the evaluation maps $e_\pm$ (and some orientations and compactifications), constitute what we call\footnote{To be more precise one could call this an ``$S^1$-Morse-Bott system'', as here the analogues of ``critical submanifolds'' are circles.} a ``Morse-Bott system" in \cite{td}. As explained in \cite{td}, out of this data we can naturally construct a ``cascade'' chain complex $(NCC_*(Y,\lambda),\partial_\ca^\J)$ as follows.

We define $NCC_*(Y,\lambda)$ to be the free $\Z$-module with two generators $\widecheck{\alpha}$ and $\widehat{\alpha}$ for each Reeb orbit $\alpha$. This module has a canonical $\Z/2$-grading, where the grading of $\widecheck{\alpha}$ is $\CZ(\alpha)$, and the grading of $\widehat{\alpha}$ is $\CZ(\alpha)+1$.

To define the differential, we generically choose a point $p_\alpha\in\overline{\alpha}$ for each Reeb orbit $\alpha$. If $\alpha\neq\beta$, then the differential coefficient $\langle\partial_\ca^\J\widehat{\alpha},\widecheck{\beta}\rangle$ is a signed count of tuples $(u_1,\ldots,u_k)$, where there are distinct Reeb orbits $\alpha=\gamma_0,\gamma_1,\ldots,\gamma_k=\beta$ such that $u_i\in\M_0^\J(\gamma_{i-1},\gamma_i)$, and for $1<i<k$, the points $p_{\gamma_i},e_-(u_i)$, and $e_+(u_{i+1})$ are cyclically ordered on $\overline{\gamma_i}$ with respect to the orientation given by the Reeb vector field. If we replace $\widehat{\alpha}$ by $\widecheck{\alpha}$, then we add the constraint that $e_+(u_1)=p_\alpha$, and we increase the dimension of $u_1$'s moduli space by $1$. Likewise, if we replace $\widecheck{\beta}$ by $\widehat{\beta}$, then we add the constraint that $e_-(u_k)=p_\beta$, and we increase the dimension of $u_k$'s moduli space by $1$. When $\alpha=\beta$, all differential coefficients are defined to be zero, except that
\begin{equation}
\label{eqn:minus2}
\langle\partial_\ca^\J\widehat{\alpha},\widecheck{\alpha}\rangle = -2
\end{equation}
when $\alpha$ is a bad Reeb orbit.

Some motivation for the above definition comes from finite-dimensional Morse-Bott theory. A Morse-Bott function on a finite-dimensional manifold can be perturbed using a Morse function $f_S$ on each critical submanifold $S$. Gradient flow lines after perturbation correspond to ``cascades'', which start and end at critical points of the perturbing Morse functions $f_S$, and which are alternating sequences of downward gradient flow lines of the Morse-Bott function and downward gradient trajectories of the perturbing Morse functions $f_S$; see \cite{banyaga-hurtubise,bourgeois,frauenfelder}. In the situation of nonequivariant contact homology, there is no direct analogue of perturbing to a Morse function. However the above differential still counts an analogue of cascades, in which the simple Reeb orbits $\overline{\gamma}$ play the role of critical submanifolds. One can imagine choosing for each $\gamma$ a perturbing Morse function $f_\gamma$ on $\overline{\gamma}$ which has two critical points which are very close to $p_\gamma$, such that the downward gradient flow away from $p_\gamma$ moves in the direction of the Reeb vector field. One can then think of the generators $\widehat{\gamma}$ and $\widecheck{\gamma}$ as representing the maximum and minimum, respectively, of the perturbing Morse function $f_\gamma$. The cyclic ordering condition in the previous paragraph corresponds to the fact that the downward gradient trajectories of $f_\gamma$ move in the direction of the Reeb vector field.

Formal arguments in \cite{td} show that $\left(\partial_\ca^\J\right)^2=0$, and that the homology does not depend on the choice of base points $p_\alpha$. This homology is the nonequivariant contact homology, which we denote by $NCH_*(Y,\lambda;\J)$. It is invariant in the following sense:

\begin{theorem}
\label{thm:NCHinvariant}
Let $Y$ be a closed manifold, and let $\lambda$ and $\lambda'$ be nondegenerate hypertight contact forms on $Y$ with $\Ker(\lambda)=\Ker(\lambda')$. Let $\J$ be a generic $S^1$-family of $\lambda$-compatible almost complex structures, and let $\J'$ be a generic $S^1$-family of $\lambda'$-compatible almost complex structures. Then there is a canonical isomorphism
\[
NCH_*(Y,\lambda;\J) = NCH_*(Y,\lambda';\J').
\]
\end{theorem}

In particular, if $\xi$ is a contact structure on $Y$ admitting a nondegenerate\footnote{In fact one can remove the nondegeneracy assumption here; see \S\ref{sec:addstr} below.} hypertight contact form, then we have a well-defined nonequivariant contact homology $NCH_*(Y,\xi)$.

\begin{remark}
Our Morse-Bott chain complex $(NCC_*(Y,\lambda),\partial_\ca^\J)$ is a contact analogue of the Floer theory for autonomous Hamiltonians studied in \cite{boduke}. In that paper, the idea was to perturb the autonomous Hamiltonian to a nondegenerate one, and to understand the Floer chain complex of the nondegenerate perturbation in Morse-Bott terms. In our situation, by contrast, we need to define the homology and prove its invariance entirely in the Morse-Bott setting.
\end{remark}

\subsection{$S^1$-equivariant contact homology}
\label{sec:chs1intro}

In \S\ref{equicurrents}, we carry out a variant of the above construction defined using a larger family of almost complex structures on $\R\times Y$, namely an $S^1$-equivariant $S^1\times ES^1$ family of almost complex structures ${\mathfrak J}$.

To define the chain complex, we fix a ``perfect Morse function'' $f$ on $BS^1$; see \S\ref{sec:familymoduli} for details. Let $\widetilde{f}$ denote its pullback to $ES^1$, and if $x$ is a critical point of $f$, let $\pi^{-1}(x)$ denote its inverse image in $ES^1$. Given critical points $x_\pm$ of $f$, and given Reeb orbits $\gamma_\pm$, we consider pairs $(\eta,u)$, where $\eta:\R\to ES^1$ is an upward gradient flow line of $\widetilde{f}$ asymptotic to points in $\pi^{-1}(x_\pm)$, and $u:\R\times S^1\to\R\times Y$ satisfies the equation
\[
\partial_s u + {\mathfrak J}_{t,\eta(s)}\partial_t u = 0,
\]
with the asymptotic conditions that $\lim_{s\to\pm\infty}\pi_\R u(s,\cdot)=\pm\infty$, and $\lim_{s\to\pm\infty}\pi_Yu(s,\cdot)$ is a parametrization of a Reeb orbit $\gamma_\pm$. As before, there is an $\R$ action on the set of solutions by translating the $\R$ coordinate in the domains of $\eta$ and $u$ simultaneously, and another $\R$ action by translating the $\R$ coordinate on the target $\R\times Y$ of $u$. There is also an $S^1$ action which simultaneously translates the $S^1$ factor on the domain of $u$ and acts on $ES^1$ in the target of $\eta$. We denote the quotient of the solution set by these actions by $\M^{\mathfrak J}((x_+,\gamma_+),(x_-,\gamma_-))$. As before, there are evaluation maps, which now have the form
\[
e_\pm: \M^{\mathfrak J}((x_+,\gamma_+),(x_-,\gamma_-)) \longrightarrow \left(\pi^{-1}(x_\pm)\times\overline{\gamma_\pm}\right)/S^1.
\]

These moduli spaces and evaluation maps satisfy the axioms of a Morse-Bott system, so we can again invoke the formalism of \cite{td} to obtain a chain complex $\left(CC_*^{S^1}(Y,\lambda),\partial^{S^1,\frak{J}}\right)$. This chain complex then has two generators for each pair $(x,\alpha)$, where $x$ is a critical point of $f$ and $\alpha$ is a Reeb orbit. We denote these two generators by $\widecheck{\alpha}\tensor U^k$ and $\widehat{\alpha}\tensor U^k$, where $2k$ is the Morse index of $x$. 
More concisely, we have a canonical identification of $\Z$-modules
\[
CC_*^{S^1}(Y,\lambda) = NCC_*(Y,\lambda)\tensor \Z[U].
\]
One can think of the formal variable $U$ as having degree $2$, although for now this chain complex is only $\Z/2$-graded, where $\widecheck{\alpha}\tensor U^k$ has grading $\op{CZ}(\alpha)$ and $\widehat{\alpha}\tensor U^k$ has grading $\op{CZ}(\alpha)+1$.

The homology of this chain complex is the $S^1$-equivariant contact homology, which we denote by $CH_*^{S^1}(Y,\lambda;\mathfrak{J})$. It is invariant in the following sense, analogously to Theorem~\ref{thm:NCHinvariant}:

\begin{theorem}
\label{thm:CHS1invariant}
Let $Y$ be a closed manifold, and $\lambda$ and $\lambda'$ be nondegenerate hypertight contact forms on $Y$ with $\Ker(\lambda)=\Ker(\lambda')$. Let $\mathfrak{J}$ be a generic $S^1$-equivariant $S^1\times ES^1$-family of $\lambda$-compatible almost complex structures, and let $\mathfrak{J}'$ be a generic $S^1$-equivariant $S^1\times ES^1$-family of $\lambda'$-compatible almost complex structures. Then there is a canonical isomorphism
\[
CH_*^{S^1}(Y,\lambda;\mathfrak{J}) = CH_*^{S^1}(Y,\lambda';\mathfrak{J}').
\]
\end{theorem}

In particular, if $\xi$ is a contact structure admitting a hypertight contact form, then we have a well-defined $S^1$-equivariant contact homology $CH_*^{S^1}(Y,\xi)$.

\begin{remark}
The $S^1$-equivariant contact homology defined above is analogous to the $S^1$-equivariant symplectic homology introduced in \cite{bo12}. The difference is that we are using contact forms instead of Hamiltonians, and we are working in a Morse-Bott setting.
\end{remark}

\begin{remark}
The nonequivariant contact homology $NCH_*(Y,\lambda;\J)$ is the homology of the subcomplex of $CC_*^{S^1}(Y,\lambda;\mathfrak{J})$ in which the exponent of $U$ is zero. Here $\J$ is obtained by restricting $\mathfrak{J}$ to the part of $S^1\times ES^1$ corresponding to the index $0$ critical point of $f$. If $\mathfrak{J}$ is chosen appropriately, then the differential on $CC_*^{S^1}$ will commute with ``multiplication by $U^{-1}$'', namely the map sending $U^k\mapsto U^{k-1}$ for $k>0$, and sending $1\mapsto 0$; see \cite[Rmk.\ 5.15]{guhu} for explanation in the similar situation of $S^1$-equivariant symplectic homology. It follows that, analogously to \cite{bo12}, there is a long exact sequence
\[
\cdots\to NCH_* \to CH_*^{S^1} \to CH_*^{S^1} \to NCH_{*-1}\to\cdots
\]
where the middle map is induced by multiplication by $U^{-1}$ on the chain complex.
\end{remark}

\subsection{The autonomous case}

We now explain how to recover the cylindrical contact homology in \S\ref{sec:cchintro} from the $S^1$-equivariant contact homology in \S\ref{sec:chs1intro}.

Suppose that $J$ is a $\lambda$-compatible almost complex structure on $\R\times Y$ which satisfies the transversality conditions needed to define cylindrical contact homology, see Definition~\ref{def:cch}. We can then compute the $S^1$-equivariant contact homology using the ``autonomous'' family of almost complex structures $\mathfrak{J} = \{J\}$. (In general, a slight perturbation of the autonomous family might be needed to obtain the transversality necessary to define the $S^1$-equivariant differential. See \S\ref{sec:chaut} for details.)

In this case, we find that the equivariant differential is given by
\[
\partial^{S^1} = \partial_\ca^J\tensor 1 + \partial_1\tensor U^{-1}.
\]
Here $\partial_\ca^J$ denotes the nonequivariant cascade differential for the autonomous family $\J=\{J\}$. In addition, the ``BV operator'' $\partial_1$ is given by
\begin{equation}
\label{eqn:defbv}
\begin{split}
\partial_1\widehat{\alpha} &= 0,\\
\partial_1\widecheck{\alpha} &= \left\{\begin{array}{cl} d(\alpha)\widehat{\alpha}, & \mbox{$\alpha$ good},\\
0, & \mbox{$\alpha$ bad}.
\end{array}
\right.
\end{split}
\end{equation}

We will see in \S\ref{sec:chaut} that the above differential is related to the cylindrical contact homology differential as follows: If $\alpha$ and $\beta$ are good Reeb orbits, then
\begin{equation}
\label{eqn:tdan}
\begin{split}
\left\langle\partial_\ca^J\widecheck{\alpha},\widecheck{\beta}\right\rangle &= \left\langle\delta\kappa\alpha,\beta\right\rangle,\\
\left\langle\partial_\ca^J\widehat{\alpha},\widehat{\beta}\right\rangle &= \left\langle -\kappa\delta\alpha,\beta\right\rangle.
\end{split}
\end{equation}
In addition, if $\alpha$ is a bad Reeb orbit, then $\left\langle\partial_\ca^J\widecheck{\alpha}, \widecheck{\beta}\right\rangle=0$ for any Reeb orbit $\beta$; and if $\beta$ is a bad Reeb orbit, then $\left\langle\partial_\ca^J\widehat{\alpha}, \widehat{\beta}\right\rangle=0$ for any Reeb orbit $\alpha$. Finally, $\left\langle\partial_\ca^J\widehat{\alpha}, \widecheck{\beta}\right\rangle=0$, except when $\alpha$ and $\beta$ are equal and bad, in which case the differential coefficient is -2; c.f.\ \eqref{eqn:minus2}.

Given the above observations, a calculation in \S\ref{sec:comparison} proves the following:

\begin{theorem}
\label{thm:cchinv}
Let $Y$ be a closed manifold, let $\lambda$ be a nondegenerate hypertight contact form on $Y$, and write $\xi=\op{Ker}(\lambda)$. Let $J$ be an almost complex structure on $\R\times Y$ which is admissible (see Definition~\ref{def:admissible}). Then there is a canonical isomorphism
\[
CH_*^{S^1}(Y,\xi)\tensor \Q = CH_*^{EGH}(Y,\lambda;J).
\]
\end{theorem}

\begin{corollary}
\label{cor:eghinv}
$CH_*^{EGH}$ is an invariant of closed contact manifolds $(Y,\xi)$ for which there exists a pair $(\lambda,J)$ where $\lambda$ is a nondegenerate hypertight contact form with $\Ker(\lambda)=\xi$, and $J$ is an admissible $\lambda$-compatible almost complex structure.
\end{corollary}

\subsection{Additional structure}
\label{sec:addstr}

The three kinds of contact homology discussed above have some additional structure on them. These are standard constructions given the material in the rest of the paper, so we will just briefly describe them here. We will mostly ignore cylindrical contact homology below, since $S^1$-equivariant contact homology determines it by Theorem~\ref{thm:cchinv} but is defined more generally.

\paragraph{Splitting by free homotopy classes.}
The differentials on the chain complexes defining cylindrical, nonequivariant, and $S^1$-equivariant contact homology all preserve the free homotopy class of Reeb orbits (since they count cylinders which project to homotopies in $Y$ between Reeb orbits). Furthermore, the chain maps proving topological invariance of the nonequivariant and $S^1$-equivariant contact homologies also preserve the free homotopy class of Reeb orbits. Consequently, if $\xi$ is a contact structure on $Y$ admitting a hypertight contact form, and if $\Gamma$ is a free homotopy class of loops in $Y$, then we have well-defined contact homologies $NCH_*(Y,\xi,\Gamma)$ and $CH_*^{S^1}(Y,\xi,\Gamma)$, which are the homologies of the subcomplexes involving Reeb orbits in the class $\Gamma$.

\paragraph{Refined grading.}
Let $N$ denote twice the minimum positive pairing of $c_1(\xi)$ with a toroidal class in $H_2(Y)$, or infinity if $c_1(\xi)$ annihilates all toroidal classes in $H_2(Y)$. 
Each of the above contact homologies has a noncanonical $\Z/N$-grading, which refines the canonical $\Z/2$-grading. To define this relative grading on cylindrical contact homology, for each free homotopy class $\Gamma$ that contains good Reeb orbits, choose a good Reeb orbit $\gamma$ in the class $\Gamma$, and choose an arbitrary value of the grading $|\gamma|\in\Z/N$ which has the same parity as $\op{CZ}(\gamma)$. There is then a unique way to extend the $\Z/N$-grading over all good Reeb orbits in the class $\Gamma$ such that if $u$ is any homotopy class of cylinder with boundary $\gamma_+ - \gamma_-$, then
\[
|\gamma_+| - |\gamma_-| = \op{ind}(u).
\]
Here $\op{ind}(u)$ is defined as in \eqref{eqn:Fredholmindex}, which makes sense even if $u$ does not come from a holomorphic cylinder.

To define the grading on nonequivariant or equivariant contact homology, one likewise chooses the grading $|\widecheck{\gamma}|$ for all Reeb orbits $\gamma$ in the homotopy class $\Gamma$. We then adopt the conventions
\[
\begin{split}
\left|\widehat{\gamma}\right| &= \left|\widecheck{\gamma}\right| + 1,\\
\left|\widecheck{\gamma}\tensor U^k\right| &= \left|\widecheck{\gamma}\right| + 2k,\\
\left|\widehat{\gamma}\tensor U^k\right| &= \left|\widecheck{\gamma}\right| + 2k+1.\\
\end{split}
\]
The topological invariance in Theorems~\ref{thm:NCHinvariant} and \ref{thm:CHS1invariant}, and the isomorphism in Theorem~\ref{thm:cchinv}, respect the relative gradings.

\paragraph{Cobordism maps.} Let $(Y_+,\lambda_+)$ and $(Y_-,\lambda_-)$ be closed manifolds with nondegenerate hypertight contact forms. Let $(X,\lambda)$ be an exact\footnote{One can also obtain cobordism maps from a strong symplectic cobordism if one uses a suitable Novikov completion of contact homology; see \cite{field} for the analogous story for embedded contact homology.} symplectic cobordism (see Definition~\ref{def:exactcob}) from $(Y_+,\lambda_+)$ to $(Y_-,\lambda_-)$, and assume further that no Reeb orbit in $Y_+$ is contractible in $X$.
Proposition~\ref{prop:cobordismmap} and the subsequent discussion show that the cobordism $(X,\lambda)$ induces a map
\[
\Phi(X,\lambda): NCH_*(Y_+,\lambda_+) \longrightarrow NCH_*(Y_-,\lambda_-)
\]
which is functorial with respect to composition of cobordisms. Likewise, Proposition~\ref{prop:famcobmap} and the subsequent discussion give a functorial map
\[
\Phi(X,\lambda): CH_*^{S^1}(Y_+,\lambda_+) \longrightarrow CH_*^{S^1}(Y_-,\lambda_-).
\]

\paragraph{Filtered versions.} Let $Y$ be a closed manifold, let $\lambda$ be a contact form on $Y$, let $L$ be a positive real number, and assume that $\lambda$ is ``$L$-nondegenerate'' and ``$L$-hypertight'', meaning that all Reeb orbits of action less than $L$ are nondegenerate and noncontractible. (In particular, $\lambda$ does not need to be hypertight.) We can then repeat the constructions of nonequivariant and equivariant contact homology above, considering only Reeb orbits with symplectic action less than $L$, to obtain well-defined ``filtered contact homologies''$NCH_*^{< L}(Y,\lambda)$ and $CH_*^{S^1,< L}(Y,\lambda)$. These do not depend on the choice of almost complex structure, although they do depend on the contact form $\lambda$; cf.\ \cite[Thm.\ 1.3]{cc2}. When $\lambda$ is actually nondegenerate and hypertight, the usual contact homologies are recovered from the filtered contact homologies by taking the direct limit over $L$, e.g.
\begin{equation}
\label{eqn:dirlimfil}
CH_*^{S^1}(Y,\xi) = \lim_{L\to\infty}CH_*^{S^1,< L}(Y,\lambda).
\end{equation}

\paragraph{The degenerate case.} 
If $\lambda$ is $L$-hypertight but possibly degenerate, and if $\lambda$ does not have any Reeb orbit of action equal to $L$, then one can still define the filtered nonequivariant or $S^1$-equivariant contact homology by letting $\lambda'$ be a small $L$-nondegenerate and $L$-hypertight perturbation of $\lambda$ and defining
\[
CH_*^{S^1,<L}(Y,\lambda) = CH_*^{S^1,<L}(Y,\lambda'),
\]
and likewise for nonequivariant contact homology. This does not depend on the choice of $\lambda'$ if the perturbation is sufficiently small. With this definition, if $\lambda$ is hypertight but possibly degenerate, then we still have the direct limit \eqref{eqn:dirlimfil}. (If $\xi$ has hypertight representatives but they are all degenerate, then the right hand side of \eqref{eqn:dirlimfil} is still an invariant of $(Y,\xi)$ and can be taken as a definition of the left hand side.)

\paragraph{Local contact homology.}

In \cite{hm}, Hryniewicz and Macarini introduced the {\em local contact homology\/} of the $d^{th}$ iterate of a simple Reeb orbit $\gamma_0$ in a (not necessarily compact) contact manifold $(Y,\lambda_0)$. We assume that the Reeb orbits $\gamma_0^k$ for $1\le k\le d$ are isolated in the loop space of $Y$, but we do not assume that these are nondegenerate. Local contact homology is defined analogously to the cylindrical contact homology $CH_*^{\op{EGH}}$, but only working in a small tubular neighborhood $N$of $\gamma_0$, for a nondegenerate perturbation $\lambda$ of $\lambda_0$, and only considering Reeb orbits of $\lambda$ that wind $d$ times around $N$.  This local contact homology is defined in \cite{hm}, assuming that one can find almost complex structures satisfying suitable transversality, and it is a key ingredient in various dynamical applications, see e.g.\ \cite{gg,ggm,ghhm}.

Using our methods, without any transversality difficulties, we can define local versions of nonequivariant and $S^1$-equivariant contact homology, which we denote by $NCH_*(Y,\lambda_0,\gamma_0,d)$ and $CH_*^{S^1}(Y,\lambda_0,\gamma_0,d)$, and prove that these are invariants which depend only on the contact form $\lambda_0$ in a neighborhood of the Reeb orbit $\gamma_0$. As in Theorem~\ref{thm:cchinv}, if there exists a perturbation $\lambda$ of $\lambda_0$ in $N$ and a $\lambda$-compatible almost complex structure $J$ satisfying sufficient transversality to define the cylindrical contact homology $CH_*^{\op{EGH}}$, which is always true in the three-dimensional case, then this cylindrical contact homology does not depend on the perturbation $\lambda$ or on $J$ and agrees with $CH_*^{S^1}(Y,\lambda_0,\gamma_0,d)\tensor\Q$. See \S\ref{sec:localcontact} for details.

\subsection{Relation with other approaches}

Bao-Honda \cite{bh} give another construction of cylindrical contact homology for hypertight contact forms in dimension 3, by modifying the contact form so that all Reeb orbits of action less than $L$ are hyperbolic, using obstruction bundle gluing to prove that the cylindrical contact homology in action less than $L$ for the modified contact form is independent of the choice of modification, and then taking the direct limit over $L$. Action-filtered versions of Theorems~\ref{thm:CHS1invariant} and \ref{thm:cchinv} show that this definition of cylindrical contact homology is also isomorphic to $CH^{S^1}\tensor\Q$.

Bourgeois-Oancea \cite[\S4.1.2(2)]{bo12} define a version of positive $S^1$-equivariant symplectic homology (over $\Z$) for a nondegenerate contact form $\lambda$ on a closed manifold $Y^{2n-1}$, assuming that $\lambda$ is hypertight, or that $c_1(\xi)|_{\pi_2(Y)}=0$ and every contractible Reeb orbit $\gamma$ satisfies $CZ(\gamma)>4-n$. In particular this includes the dynamically convex case in three dimensions, and also local contact homology. The theory defined by Bourgeois-Oancea can be used a substitute for cylindrical contact homology in some applications. We expect that it is canonically isomorphic to the $S^1$-equivariant contact homology $CH^{S^1}_*(Y,\xi)$ defined here (in the hypertight case) and in the sequel \cite{inv} (in the dynamically convex case in three dimensions).

Bao-Honda \cite{bh2} and Pardon \cite{p} use virtual techniques (variations on the idea of ``Kuranishi structure'') to define the contact homology algebra (over $\Q$) of any closed manifold with a nondegenerate contact form. In the hypertight case, one can obtain cylindrical contact homology and its invariance from the contact homology algebra. 

More generally, work in progress of Fish-Hofer will use the polyfold theory of Hofer-Wysocki-Zehnder \cite{HWZbook} to define symplectic field theory (SFT), which in particular will yield cylindrical contact homology (over $\Q$) for a dynamically convex contact form.  An alternate foundation for SFT is proposed by Ishikawa \cite{is}.

One reason why we are pursuing the more geometric approach in the present paper and \cite{dc,inv}, even though it is less general than the more abstract approaches above, is that for computations and applications, it is desirable when possible to understand cylindrical contact homology directly in terms of Reeb orbits and holomorphic cylinders between them. Also, in applications to symplectic embedding problems, it is important to understand the holomorphic curves in symplectic cobordisms that arise from contact homology, see e.g.\ \cite{mcd}.

\subsection{The plan}

In \S\ref{section:plainmoduli} and \S\ref{cascades} we explain the definition of nonequivariant contact homology and prove its invariance (Theorem~\ref{thm:NCHinvariant}). In \S\ref{equicurrents} we modify this construction to define equivariant contact homology and prove its invariance (Theorem~\ref{thm:CHS1invariant}). In \S\ref{equicasc} we describe the nonequivariant and equivariant contact homology for autonomous $J$ (assuming suitable transversality) and prove the relation with cylindrical contact homology (Theorem~\ref{thm:cchinv}). In \S\ref{examples} we work out some examples, including a definition of local contact homology.

Our constructions use various analytical results on transversality, compactness, and gluing, and we omit the proofs of these where they follow from standard arguments. However we do include a long appendix giving details of the orientations of the moduli spaces that we consider. The gluing theory is sketched in \S\ref{sec:gbs}, and more details about gluing will be provided in the sequel \cite{inv}, where we need to consider a more general situation.

In \cite{inv}, we will extend the machinery in the present paper to construct an invariant integral lift of cylindrical contact homology for dynamically convex context forms in three dimensions. In this case, for a generic $\lambda$-compatible almost complex structure $J$ on $\R\times Y$, there may exist certain nontransverse index $2$ holomorphic buildings with one positive end and one negative end. These do not interfere with the proof that $(\partial^{EGH})^2=0$, as shown in \cite[Prop.\ 3.1]{dc}. However these buildings do make nontrivial contributions to the cascade differentials computing nonequivariant and $S^1$-equivariant contact homology. To understand these contributions, we will need to use a bit of obstruction bundle gluing as in \cite{ht2}.

\paragraph{Acknowledgments.} We thank Mohammed Abouzaid, Helmut Hofer, Umberto Hryniewicz, Dusa McDuff, Paul Seidel, and Chris Wendl for helpful conversations. We thank the anonymous referee for careful reading and many detailed comments which helped us improve the paper.

\paragraph{Notation conventions.} Many moduli spaces below are defined by first defining a larger moduli space and then modding out by some group action(s). Generally, if $\M$ is a moduli space of interest, then $\widetilde{\M}$ denotes a larger version of this moduli space before modding out by an $\R$ action, so that $\M=\widetilde{\M}/\R$. Likewise, where applicable, $\widetilde{\widetilde{\M}}$ denotes a larger version of $\M$ before modding out by an $\R^2$ action, and $\widehat{\M}$ denotes a larger version of $\M$ before modding out by an $\R^2\times S^1$ action.

In addition, $\overline{\M}$ denotes a compactification of $\M$. Moduli spaces of the form $\M^\ca$ are cascade moduli spaces defined in \S\ref{cascades}.


\section{Nonequivariant moduli spaces}
\label{section:plainmoduli}

In this section we give the definitions and state the key properties of moduli spaces of holomorphic cylinders for $S^1$-dependent almost complex structures.  These moduli spaces will be used in \S\ref{cascades} to define nonequivariant contact homology.

\subsection{Definitions}

Let $(Y^{2n-1}, \lambda)$ be a closed nondegenerate contact manifold with contact structure $\xi=\ker \lambda$ and Reeb vector field $R$. We assume throughout that $\lambda$ is hypertight, i.e.\ all Reeb orbits are noncontractible.  

Let $\J=\{J_t\}_{t\in S^1}$ be an $S^1$-family of $\lambda$-compatible almost complex structures on $\R\times Y$; see Definition~\ref{def:lambdacompatible}. 

\begin{definition}
\label{def:Mtilde}
If $\gamma_+$ and $\gamma_-$ are Reeb orbits, let $\Mt^\J(\gamma_+,\gamma_-)$ denote the moduli space of maps $u:\R\times S^1\to \R\times Y$ satisfying the equations
\begin{gather}
\label{eqn:cr}
\partial_s u + J_t\partial_t u = 0,\\
\label{eqn:pirlim}
\lim_{s\to\pm\infty}\pi_\R(u(s,\cdot))=\pm\infty,\\
\label{eqn:param}
\lim_{s\to\pm\infty}\pi_Y(u(s,\cdot)) \;\mbox{is a parametrization of $\gamma_\pm$},
\end{gather}
modulo $\R$ translation in the domain. If $\gamma_+$ and $\gamma_-$ are distinct, then $\R$ acts freely on $\Mt^\J(\gamma_+,\gamma_-)$ by translation of the $\R$ coordinate on the target $\R\times Y$, and we define
\[
\M^\J(\gamma_+,\gamma_-)=\Mt^\J(\gamma_+,\gamma_-)/\R.
\]
\end{definition}

If $\gamma$ is a Reeb orbit, let $\overline{\gamma}$ denote the underlying simple Reeb orbit, so that $\gamma$ is a $d$-fold cover of $\overline{\gamma}$ for some integer $d>0$. There are then well-defined evaluation maps
\[
e_\pm: \M^\J(\gamma_+,\gamma_-) \longrightarrow \overline{\gamma_\pm}
\]
defined by
\begin{equation}
\label{eqn:defepm}
e_\pm(u) \eqdef \lim_{s\to\pm\infty} \pi_Y(u(s,0)).
\end{equation}

\subsection{Transversality}
\label{sec:trans}

If $d$ is an integer, let $\M^\J_d(\gamma_+,\gamma_-)$ denote the set of $u\in\M^\J(\gamma_+,\gamma_-)$ with
\begin{equation}
\label{eqn:mdconstraint}
\CZ_\tau(\gamma_+)-\CZ_\tau(\gamma_-)+2c_1(u^*\xi,\tau)=d.
\end{equation}
Here the notation is as in equation \eqref{eqn:Fredholmindex}. A standard transversality argument based on \cite[\S8]{wendl}, going back to \cite{dragnev,fhs}, shows the following.

\begin{proposition}
\label{prop:transversality}
If $\J$ is generic, then:
\begin{description}
\item{\emph{(a)}}
For any distinct Reeb orbits $\gamma_+$ and $\gamma_-$, and any integer $d$, the moduli space $\M_d^\J(\gamma_+,\gamma_-)$ is cut out transversely and is a smooth manifold of dimension $d$,
and the evaluation maps $e_+$ and $e_-$ on it are smooth.
\item{\emph{(b)}}
For any distinct Reeb orbits $\gamma_0,\ldots,\gamma_k$ and any integers $d_1,\ldots,d_k$ the $k$-fold fiber product
\[
\M_{d_1}^\J(\gamma_0,\gamma_1)\times_{\overline{\gamma_1}}\M_{d_2}^\J(\gamma_1,\gamma_2) \times_{\overline{\gamma_2}} \cdots \times_{\overline{\gamma_{k-1}}}\M_{d_k}^\J(\gamma_{k-1},\gamma_k)
\]
is cut out transversely, and in particular is a smooth manifold of dimension $1-k+\sum_{i=1}^kd_i$.
\end{description}
\end{proposition}

The precise meaning of transversality in part (a) is that each $u$ in the moduli space is ``regular'' in the sense of Definition~\ref{def:regular}; see \S\ref{sec:tms} for explanation. The transversality in (b) means that if $(u_1,\ldots,u_k)$ is an element of this fiber product, then the map
\[
\begin{split}
T_{(u_1,\ldots,u_k)}\prod_{i=1}^k\M_{d_i}^\J(\gamma_{i-1},\gamma_i) &\longrightarrow \bigoplus_{i=1}^{k-1} T_{e_-(u_i)}\overline{\gamma_i},\\
(v_1,\ldots,v_k) & \longmapsto (de_-(v_1)-de_+(v_2), \ldots, de_-(v_{k-1}) - de_+(v_k))
\end{split}
\]
is surjective.

Assume below that $\J$ is generic in the sense of Proposition~\ref{prop:transversality}.

\subsection{Orientations}

Recall that any manifold $M$ has an ``orientation sheaf'' $\mc{O}_M$, which is a local system locally isomorphic to $\Z$, defined by $\mc{O}_M(p)=H_{\dim(M)}(M,M\setminus\{p\})$ for $p\in M$; an orientation of $M$ is equivalent to a section of $\mc{O}_M$ which restricts to a generator of each fiber. If $\mc{O}$ is another local system on $M$ which is locally isomorphic to $\Z$, then we define an ``orientation of $M$ with values in $\mc{O}$'' to be a section
\[
\mathfrak{o}\in H^0(M;\mc{O}_M\tensor\mc{O})
\]
which restricts to a generator of each fiber.

As we review in the appendix (see Definition~\ref{def:Ogamma}), one can use the theory of coherent orientations to assign to each Reeb orbit $\gamma$ a canonical local system $\mc{O}_\gamma$ over $\overline{\gamma}$, locally isomorphic to $\Z$, such that:

\begin{proposition}
\label{prop:orientations}
(proved in \S\ref{sec:oms})
\begin{description}
\item{\emph{(a)}}
The local system $\mc{O}_{\gamma}$ is trivial, i.e.\ (noncanonically) isomorphic to $\overline{\gamma}\times\Z$, if and only if $\gamma$ is a good Reeb orbit.
\item{\emph{(b)}}
The moduli space $\M^\J(\gamma_+,\gamma_-)$ has a canonical orientation with values in $e_+^*\mc{O}_{\gamma_+} \tensor e_-^*\mc{O}_{\gamma_-}$.
\end{description}
\end{proposition}

\subsection{Compactness}

\begin{definition}
Let $\gamma_+$ and $\gamma_-$ be distinct Reeb orbits. A ($k$-level) {\em broken $\J$-holomorphic cylinder\/} from $\gamma_+$ to $\gamma_-$ is a $k$-tuple $(u_1,\ldots,u_k)$ where there exist distinct Reeb orbits $\gamma_+=\gamma_0,\gamma_1,\ldots,\gamma_k=\gamma_-$ such that $u_i\in\M^\J(\gamma_{i-1},\gamma_i)$ for $i=1,\ldots,k$ and $e_-(u_i)=e_+(u_{i+1})$ for $i=1,\ldots,k-1$.
\end{definition}

\begin{definition}
\label{def:Mdbar}
\begin{itemize}
\item
If $\gamma_+$ and $\gamma_-$ are distinct Reeb orbits, let $\overline{\M^\J_d}(\gamma_+,\gamma_-)$ denote the set of broken $\J$-holomorphic cylinders $(u_1,\ldots,u_k)$ as above, where $u_i\in\M^\J_{d_i}(\gamma_{i-1},\gamma_i)$ with $\sum_{i=1}^kd_i=d$.
\item
Define evaluation maps
\[
e_\pm:\overline{\M^\J_d}(\gamma_+,\gamma_-) \longrightarrow \overline{\gamma_\pm}
\]
by $e_+(u_1,\ldots,u_k) = e_+(u_1)$ and $e_-(u_1,\ldots,u_k) = e_-(u_k)$.
\item
We give $\overline{\M^\J_d}(\gamma_+,\gamma_-)$ the usual topology. In particular, a sequence $\{u(\nu)\}_{\nu = 1,2,\ldots}$ in $\M^\J_d(\gamma_+,\gamma_-)$ converges to $(u_1,\ldots,u_k)\in\overline{\M^\J_d}(\gamma_+,\gamma_-)$ if and only if one can assign to each $\nu$ a choice of $k$ representatives $u(\nu)_1,\ldots,u(\nu)_k\in\Mt^\J_d(\gamma_+,\gamma_-)$ of $u(\nu)$ such that for each $i=1,\ldots,k$, the sequence $\{u(\nu)_i\}_{\nu = 1,2,\ldots}$ of maps $\R\times S^1\to \R\times Y$ converges in $C^\infty$ on compact sets to $u_i$.
\end{itemize}
\end{definition}

\begin{proposition}
\label{prop:compactness}
For any $\J$ (not necessarily generic), if $\gamma_+$ and $\gamma_-$ are distinct Reeb orbits, then $\overline{\M^\J_d}(\gamma_+,\gamma_-)$ is compact.
\end{proposition}

\begin{proof}
This follows from standard compactness arguments as in \cite[Thm.\ 10.4]{behwz}. (This reference does not consider domain-dependent almost complex structures, but that does not affect the argument here.) The hypertightness assumption is needed to avoid bubbling of holomorphic planes.
\end{proof}

\subsection{Constrained moduli spaces and gluing}
\label{sec:constrained}

Let $\gamma_+,\gamma_-$ be distinct Reeb orbits, and let $p_\pm\in\overline{\gamma_\pm}$. We then define moduli spaces with point constraints
\[
\begin{split}
\M^\J_d(\gamma_+,p_+,\gamma_-) &= e_+^{-1}(p_+)\subset \M^\J_{d}(\gamma_+,\gamma_-),\\
\M^\J_d(\gamma_+,\gamma_-,p_-) &= e_-^{-1}(p_-)\subset \M^\J_{d}(\gamma_+,\gamma_-),\\
\M^\J_d(\gamma_+,p_+,\gamma_-,p_-) &= (e_+\times e_-)^{-1}(p_+,p_-) \subset \M^\J_{d}(\gamma_+,\gamma_-).
\end{split}
\]
If the pair $(p_+,p_-)$ is generic, then each set on the left hand side is a smooth manifold of dimension $d-1$ or $d-2$, with a canonical orientation with values in $e_+^*\mc{O}_+\tensor e_-^*\mc{O}_-$. Here we orient the spaces with point constraints using the conventions in \cite[\S2.2]{td}.
We also define ``compactified'' constrained moduli spaces by
\[
\overline{\M^\J_d}(\gamma_+,p_+,\gamma_-) = e_+^{-1}(p_+)\subset \overline{\M^\J_{d}}(\gamma_+,\gamma_-)
\]
and so forth.

In the proposition below, we orient fiber products using the convention in \cite[\S2.1]{td}.

\begin{proposition}
\label{prop:currentcompact}
Let $\gamma_+$ and $\gamma_-$ be distinct Reeb orbits. Assume that the pair $(p_+,p_-)$ is generic so that:
\begin{itemize}
\item
$p_+$ is a regular value of all evaluation maps $e_+: \M_d^\J(\gamma_+,\gamma_0)\to\overline{\gamma_+}$ for $d\le 2$.
\item
$p_-$ is a regular value of all evaluation maps $e_-:\M_d^\J(\gamma_0,\gamma_-)\to\overline{\gamma_-}$ for $d\le 2$.
\item
$(p_+,p_-)$ is a regular value of all products of evaluation maps
\[
e_+\times e_-: \M_d^\J(\gamma_+,\gamma_-) \longrightarrow \overline{\gamma_+}\times\overline{\gamma_-}
\]
for $d\le 3$.
\end{itemize}
Then:
\begin{description}
\item{\emph{(a)}} $\M^\J_0(\gamma_+,\gamma_-)$ is finite.

\item{\emph{(b)}} $\overline{\M^\J_1}(\gamma_+,\gamma_-)$ is a compact oriented topological one-manifold with oriented boundary

\begin{equation}
\label{eqn:boundaryhatcheck}
\partial\overline{\M^{\J}_1}(\gamma_+,\gamma_-) = \coprod_{\substack{\gamma_0\neq\gamma_+,\gamma_-\\d_++d_-=1}} (-1)^{d_+}\M^\J_{d_+}(\gamma_+,\gamma_0)\times_{\overline{\gamma_0}}\M^\J_{d_-}(\gamma_0,\gamma_-).
\end{equation}

\item{\emph{(c)}}
$\overline{\M^\J_2}(\gamma_+,\gamma_-,p_-)$ is a compact oriented topological one-manifold with oriented boundary
\begin{equation}
\label{eqn:boundarycheckcheck}
\partial \overline{\M^\J_2}(\gamma_+,\gamma_-,p_-) =
 \coprod_{\substack{\gamma_0\neq\gamma_+,\gamma_-\\d_++d_-=2}} (-1)^{d_+} \M^\J_{d_+}(\gamma_+,\gamma_0) \times_{\overline{\gamma_0}} \M^\J_{d_-}(\gamma_0,\gamma_-,p_-)\\
\end{equation}
Similarly, $\overline{\M^\J_2}(\gamma_+,p_+,\gamma_-)$ is a compact oriented topological one-manifold with oriented boundary
\begin{equation}
\label{eqn:boundaryhathat}
\partial\overline{\M^\J_2}(\gamma_+,p_+,\gamma_-) =  \coprod_{\substack{\gamma_0\neq \gamma_+,\gamma_-\\d_++d_-=2}} (-1)^{d_+-1}\M^\J_{d_+}(\gamma_+,p_+,\gamma_0)\times_{\overline{\gamma_0}}\M^\J_{d_-}(\gamma_0,\gamma_-).
\end{equation}
\item{\emph{(d)}}
$\overline{\M^\J_3}(\gamma_+,p_+,\gamma_-,p_-)$ is a compact oriented topological one-manifold with oriented boundary
\begin{equation}
\label{eqn:boundarycheckhat}
\partial \overline{\M^\J_3}(\gamma_+,p_+,\gamma_-,p_-) =  \coprod_{\substack{\gamma_0\neq\gamma_+,\gamma_-\\d_++d_-=3}} (-1)^{d_+-1} \M^\J_{d_+}(\gamma_+,p_+,\gamma_0)\times_{\overline{\gamma_0}} \M^\J_{d_-}(\gamma_0,\gamma_-,p_-).
\end{equation}
\end{description}
\end{proposition}

\begin{proof}
(a) By Proposition~\ref{prop:transversality}, the moduli space $\M^\J_0(\gamma_+,\gamma_-)$ is a $0$-dimensional manifold, and $\overline{\M^\J_0(\gamma_+,\gamma_-)}\setminus \M^\J_0(\gamma_+,\gamma_-)$ is empty. Thus $\overline{\M^\J_0(\gamma_+,\gamma_0)}$ is discrete. It then follows from Proposition~\ref{prop:compactness} that $\overline{\M^\J_0(\gamma_+,\gamma_-)}$, and in particular $\M^\J_0(\gamma_+,\gamma_-)$, is finite.

(b)--(d). The compactness follows from Proposition~\ref{prop:compactness}. The fact that the compactified moduli spaces are manifolds with boundary as described follows from Proposition~\ref{prop:gbs}.
\end{proof}

\subsection{Morse-Bott systems}
\label{sec:mbs}

It follows from the above results that for generic $\J$, the moduli spaces $\M^\J_d(\gamma_+,\gamma_-)$ and the evaluation maps on them constitute a ``Morse-Bott system'' in the sense of \cite[Def.\ 2.1]{td}.

More precisely, a Morse-Bott system is a tuple $(X,|\cdot|,S,\mc{O},M_*,e_\pm)$ where:
\begin{itemize}
\item
$X$ is a set.
\item
$|\cdot|$ is a function $X\to\Z/2$ (the mod 2 grading).
\item
$S$ is a function which assigns to each $x\in X$ a closed connected oriented $1$-manifold $S(x)$.
\item
$\mc{O}$ assigns to each $x\in X$ a local system $\mc{O}_x$ over $S(x)$ which is locally isomorphic to $\Z$.
\item
If $x_+,x_-\in X$ are distinct and $d\in\{0,1,2,3\}$, then $M_d(x_+,x_-)$ is a smooth manifold of dimension $d$.
\item
$e_\pm:M_d(x_+,x_-)\to S(x_\pm)$ are smooth maps.
\item
$M_d(x_+,x_-)$ is equipped with an orientation with values in $e_+^*\mc{O}_{x_+}\tensor e_-^*\mc{O}_{x_-}$.
\end{itemize}
These are required to satisfy the ``Grading'', ``Finiteness'', ``Fiber Product Transversality'', and ``Compactification'' axioms in \cite[\S2.2]{td}.

In the present case we can take $X$ to be the set of Reeb orbits. For a Reeb orbit $\gamma$, we define $|\gamma|$ to be the mod $2$ Conley-Zehnder index $\op{CZ}(\gamma)$, and $S(\gamma)=\overline{\gamma}$, oriented by the Reeb vector field. Then $\mc{O}_\gamma$ is the local system in Proposition~\ref{prop:orientations}, and $M_d(\gamma_+,\gamma_-)$ is the moduli space $\M_d^\J(\gamma_+,\gamma_-)$, with the evaluation maps defined by \eqref{eqn:defepm}. Here we are discarding the moduli spaces $\M_d^\J$ with $d>3$.

\begin{proposition}
\label{prop:mbs}
If $\J$ is generic, then the above data constitute a Morse-Bott system in the sense of \cite[Def.\ 2.1]{td}.
\end{proposition}

\begin{proof}
The Grading axiom in \cite[\S2.2]{td} requires that if $\M_d^\J(\gamma_+,\gamma_-)$ is nonempty then
\begin{equation}
\label{eqn:gradingaxiom}
\op{CZ}(\gamma_+) - \op{CZ}(\gamma_-) \equiv d \mod 2.
\end{equation}
This follows from equation \eqref{eqn:mdconstraint} and Proposition~\ref{prop:transversality}(a).

The Finiteness axiom in \cite[\S2.2]{td} requires that for each Reeb orbit $\gamma_0$, there are only finitely many tuples $(k,\gamma_1,\ldots,\gamma_k)$ where $k$ is a positive integer, $\gamma_1,\ldots,\gamma_k$ are distinct Reeb orbits, and there exist $d_1,\ldots,d_k\in\{0,1,2,3\}$ with $\M_{d_i}^\J(\gamma_{i-1},\gamma_i)\neq\emptyset$ for each $i=1,\ldots,k$. This holds in the present case because if $\gamma_+\neq\gamma_-$ then $\M_d^\J(\gamma_+,\gamma_-)\neq\emptyset$ only if the symplectic action of $\gamma_+$ is strictly greater than the symplectic action of $\gamma_-$; and for each $L\in\R$, there are only finitely many Reeb orbits with action less than $L$ (because $Y$ is compact and the contact form $\lambda$ is nondegenerate).

The Fiber Product Transversality axiom in \cite[\S2.2]{td} follows from Proposition~\ref{prop:transversality}(b). (The latter is a much stronger statement.)

Parts (a)--(d) of the Compactness axiom in \cite[\S2.2]{td} follow from the corresponding parts of Proposition~\ref{prop:currentcompact}. The rest of the Compactness axiom holds automatically as explained in \cite[Rmk.\ 2.6]{td}.
\end{proof}

\section{Nonequivariant contact homology}
\label{cascades}

As in \S\ref{section:plainmoduli}, let $Y$ be a closed manifold, let $\lambda$ be a nondegenerate hypertight contact form on $Y$, and let $\J$ be a generic $S^1$-family of $\lambda$-compatible almost complex structures on $\R\times Y$. In this section we define the nonequivariant contact homology $NCH_*(Y,\lambda;\J)$ and prove that it is an invariant of $Y$ and $\xi=\Ker(\lambda)$.

\subsection{Abstract Morse-Bott theory}

To define nonequivariant contact homology and prove its invariance, we will invoke the following result from \cite{td}. The statement of this result includes some terminology defined in \cite{td} which will be reviewed below.

\begin{theorem}\cite[Thm.\ 1.1]{td}
\label{thm:td}
\begin{description}
\item{(a)} Let $A$ be a Morse-Bott system. Then the cascade homology $H_*^\ca(A)$ is well-defined, independently of the choice of base points.
\item{(b)} Let $\Phi$ be a morphism of Morse-Bott systems from $A_1$ to $A_2$. Then:
\begin{description}
\item{(i)}
There is a well-defined induced map on cascade homology
\[
\Phi_*: H_*^\ca(A_1) \longrightarrow H_*^\ca(A_2).
\]
\item{(ii)} If $A_1=A_2$ and $\Phi$ is the identity morphism, then $\Phi_*$ is the identity map.
\item{(iii)} If $\Psi$ is a morphism from $A_2$ to $A_3$, and if $\Phi$ and $\Psi$ are composable, then the composition $\Psi\circ\Phi$ satisfies
\[
(\Psi\circ\Phi)_* = \Psi_* \circ \Phi_*: H_*^\ca(A_1) \longrightarrow H_*^\ca(A_3).
\]
\item{(iv)}
If $\Phi'$ is another morphism from $A_1$ to $A_2$ which is homotopic to $\Phi$, then
\[
\Phi_* = (\Phi')_*: H_*^\ca(A_1) \longrightarrow H_*^\ca(A_2).
\]
\end{description}
\end{description}
\end{theorem}

We define the nonequivariant contact homology $NCH_*(Y,\lambda;\J)$ to be the cascade homology of the Morse-Bott system in Proposition~\ref{prop:mbs}. We now spell out explicitly what this means.

\subsection{Cascade moduli spaces}
\label{sec:cms}

To start, we need to generically choose, for each Reeb orbit $\gamma$, a base point $p_\gamma\in\overline{\gamma}$. Denote this set of choices by $\mc{P}$. 

We need to study ``cascade'' moduli spaces $\M^\ca_d\left(\widehat{\alpha},\widehat{\beta}\right)$, $\M^\ca_d\left(\widehat{\alpha},\widecheck{\beta}\right)$, $\M^\ca_d\left(\widecheck{\alpha},\widehat{\beta}\right)$, and $\M^\ca_d\left(\widecheck{\alpha},\widecheck{\beta}\right)$ for each pair of (possibly equal) Reeb orbits $\alpha,\beta$ and each nonnegative integer $d$. These will be $d$-dimensional manifolds with orientations with values in $\mc{O}_\alpha(p_\alpha)\tensor \mc{O}_\beta(p_\beta)$.

When $\alpha=\beta$, the definition is simple:

\begin{definition}
\label{def:cascaa}
If $\alpha$ is a Reeb orbit, define
\begin{gather*}
\M^\ca_d\left(\widehat{\alpha},\widehat{\alpha}\right) = \M^\ca_d\left(\widecheck{\alpha},\widehat{\alpha}\right) =  M^\ca_d\left(\widecheck{\alpha},\widecheck{\alpha}\right) = \emptyset,\\
\M^\ca_d\left(\widehat{\alpha},\widecheck{\alpha}\right) = \left\{\begin{array}{cl} \mbox{$2$ points} & \mbox{if $d=0$},\\
\emptyset & \mbox{if $d>0$.}
\end{array}\right.
\end{gather*}
The above two points have opposite orientations when $\alpha$ is good; and they both have negative orientation\footnote{It makes sense to speak of ``negative orientation'' here because the orientation of $\M^\ca_0\left(\widehat{\alpha},\widecheck{\alpha}\right)$ has values in $\mc{O}_\alpha(p_\alpha)\tensor\mc{O}_\alpha(p_\alpha)=\Z$.} when $\alpha$ is bad.
\end{definition}

We now define the cascade moduli spaces for $\alpha\neq\beta$.

\paragraph{Notation guide.} Below, the notation $\widecheck{\gamma}$ means that there is a point constraint when $\gamma$ is at the top, but not when it is at the bottom; and $\widehat{\gamma}$ means that there is a point constraint when $\gamma$ is at the bottom, but not when it is at the top.

\begin{definition}
\label{def:cms}
If $\alpha$ and $\beta$ are distinct Reeb orbits, let  $\widetilde{\alpha}$ denote either $\widehat{\alpha}$ or $\widecheck{\alpha}$, and let $\widetilde{\beta}$ denote either $\widehat{\beta}$ or $\widecheck{\beta}$. We define the cascade moduli space $\M^\ca_d\left(\widetilde{\alpha},\widetilde{\beta}\right)$ as follows. An element of $\M^\ca_d\left(\widetilde{\alpha},\widetilde{\beta}\right)$ is a tuple $(u_1,\ldots,u_k)$ for some positive integer $k$, such that there are distinct Reeb orbits $\alpha=\gamma_0,\gamma_1,\ldots,\gamma_k=\beta$ and nonnegative integers $d_1,\ldots,d_k$, such that:
\begin{itemize}
\item
\[
\sum_{i=1}^kd_i=
\left\{\begin{array}{cl}
d, & (\widetilde{\alpha},\widetilde{\beta}) = (\widehat{\alpha},\widecheck{\beta}),\\
d+1, & (\widetilde{\alpha},\widetilde{\beta}) = (\widehat{\alpha},\widehat{\beta}),\; (\widecheck{\alpha},\widecheck{\beta}),\\
d+2, & (\widetilde{\alpha},\widetilde{\beta}) = (\widecheck{\alpha},\widehat{\beta}).
\end{array}
\right.
\]
\item
If $k=1$, then
\[
u_1\in\left\{\begin{array}{cl} \M^\J_d(\alpha,\beta), & (\widetilde{\alpha},\widetilde{\beta}) = (\widehat{\alpha},\widecheck{\beta}),\\
\M^\J_{d+1}(\alpha,p_{\alpha},\beta), & (\widetilde{\alpha},\widetilde{\beta}) = (\widecheck{\alpha},\widecheck{\beta}),\\
\M^\J_{d+1}(\alpha,\beta,p_{\beta}), & (\widetilde{\alpha},\widetilde{\beta}) = (\widehat{\alpha},\widehat{\beta}),\\
\M^\J_{d+2}(\alpha,p_\alpha,\beta,p_\beta) & (\widetilde{\alpha},\widetilde{\beta}) = (\widecheck{\alpha},\widehat{\beta}).
\end{array}
\right.
\]
\item
If $k>1$, then:
\begin{itemize}
\item
\[
u_1\in\left\{\begin{array}{cl} \M^\J_{d_1}(\alpha,\gamma_1), & \widetilde{\alpha}=\widehat{\alpha},\\
\M^\J_{d_1}(\alpha,p_\alpha,\gamma_1), & \widetilde{\alpha}=\widecheck{\alpha}.
\end{array}
\right.
\]
\item
If $1<i<k$ then $u_i\in\M^\J_{d_i}(\gamma_{i-1},\gamma_i)$.
\item
\[
u_k\in\left\{\begin{array}{cl} \M^\J_{d_k}(\gamma_{k-1},\beta), & \widetilde{\beta}=\widecheck{\beta},\\
\M^\J_{d_k}(\gamma_{k-1},\beta,p_\beta), & \widetilde{\beta}=\widehat{\beta}.
\end{array}\right.
\]
\item For $i=1,\ldots,k-1$, the points $p_{\gamma_i}$, $e_-(u_i)$, and $e_+(u_{i+1})$, are distinct and positively cyclically ordered with respect to the orientation of $\overline{\gamma_i}$.
\end{itemize}
\item When $u_1$ is not constrained to have $e_+(u_1)=p_\alpha$, we assume that $e_+(u_1)\neq p_\alpha$; likewise when $u_k$ is not constrained to have $e_-(u_k)=p_\beta$, we assume that $e_-(u_k)\neq p_\beta$.
\end{itemize}
We topologize $\M^\ca_d\left(\widetilde{\alpha},\widetilde{\beta}\right)$ as a subset of the disjoint union of Cartesian products
\[
\M^\J_{d_1}(\gamma_0,\gamma_1)\times\cdots\times\M^\J_{d_k}(\gamma_{k-1},\gamma_k).
\]
\end{definition}

It follows from Proposition~\ref{prop:transversality}(a) that if the set of base points $\mc{P}=\{p_\gamma\}$ is generic, then $\M^\ca_d\left(\widetilde{\alpha},\widetilde{\beta}\right)$ is a smooth manifold of dimension $d$. Furthermore, $\M^\ca_d\left(\widetilde{\alpha},\widetilde{\beta}\right)$ has a canonical orientation with values in $\mc{O}_\alpha(p_\alpha)\tensor\mc{O}_\beta(p_\beta)$, according to the convention in \cite[\S3.2]{td}.

We now have the following compactness result for the cascade moduli spaces. Below, following \cite[\S2.1]{td}, a ``compactification'' of a smooth oriented $1$-manifold $M$ means a compact oriented topological manifold with boundary $\overline{M}$ such that $M$ is an open oriented submanifold of $\overline{M}$ and $\overline{M}\setminus M$ is finite (but possibly larger than $\partial\overline{M}$).

\begin{proposition}
\label{prop:cascadekey}
Suppose the set of base points $\mc{P}=\{p_\gamma\}$ is generic. Let $\alpha$ and $\beta$ be Reeb orbits, let $\widetilde{\alpha}$ denote either $\widehat{\alpha}$ or $\widecheck{\alpha}$, and let $\widetilde{\beta}$ denote either $\widehat{\beta}$ or $\widecheck{\beta}$. Then:
\begin{description}
\item{(a)}
$\M^\ca_0\left(\widetilde{\alpha},\widetilde{\beta}\right)$ is finite.
\item{(b)}
$\M^\ca_1\left(\widetilde{\alpha},\widetilde{\beta}\right)$ has a compactification $\overline{\M}^\ca_1\left(\widetilde{\alpha},\widetilde{\beta}\right)$ with oriented boundary
\[
\begin{split}
\partial \overline{\M}^\ca_1\left(\widetilde{\alpha},\widetilde{\beta}\right) &= \coprod_{\gamma} \M^\ca_0\left(\widetilde{\alpha},\widehat{\gamma}\right)\times \M^\ca_0\left(\widehat{\gamma},\widetilde{\beta}\right) \\
&\;\sqcup\coprod_{\gamma} \M^\ca_0\left(\widetilde{\alpha},\widecheck{\gamma}\right)\times \M^\ca_0\left(\widecheck{\gamma},\widetilde{\beta}\right).
\end{split}
\]
\end{description}
\end{proposition}

\begin{proof}
This follows by applying \cite[Prop.\ 3.2]{td} to the Morse-Bott system in Proposition~\ref{prop:mbs}.
\end{proof}

\begin{remark}
To briefly review what goes into the proof of Proposition~\ref{prop:cascadekey}: Each tuple $(u_1,\ldots,u_k)\in \M^\ca_1\left(\widetilde{\alpha},\widetilde{\beta}\right)$ has exactly one element $u_i$ which lives in a one-dimensional moduli space, while all $u_j$ with $j\neq i$ are rigid. This moduli space has ends where $u_i$ breaks, and also where the cyclic ordering condition fails. Some of the latter ends glue to the former ends, while others give rise to the desired boundary points. When $i=1$ and $\widetilde{\alpha}=\widehat{\alpha}$, there can be additional ends where the last condition in Definition~\ref{def:cms} fails; when $\alpha$ is good these are glued together, while when $\alpha$ is bad they cannot be glued together (due to incompatible orientations) and give rise to boundary points involving $\M^\ca_d\left(\widehat{\alpha},\widecheck{\alpha}\right)$. Boundary points involving $\M^\ca_d\left(\widehat{\beta},\widecheck{\beta}\right)$ arise similarly when $i=k$, $\widetilde{\beta}=\widecheck{\beta}$, and $\beta$ is bad.
\end{remark}

\subsection{Definition of nonequivariant contact homology}
\label{sec:defnch}

Let $\mc{P}=\{p_\gamma\}$ be a generic choice of base points as in \S\ref{sec:cms}. We now define a $\Z/2$-graded chain complex $\left(NCC_*(Y,\lambda),\partial^\J_\ca\right)$ over $\Z$ as follows. (The differential also depends on $\mc{P}$, but we omit this from the notation.)

The $\Z$-module $NCC_*(Y,\lambda)$ is the direct sum of two copies of $\mc{O}_\gamma(p_\gamma)$ for each Reeb orbit $\gamma$. For notational convenience, we fix a generator of $\mc{O}_\gamma(p_\gamma)$ for each $\gamma$. We can then regard $NCC_*(Y,\lambda)$ as the free $\Z$-module with two generators $\widehat{\gamma}$ and $\widecheck{\gamma}$ for each Reeb orbit $\gamma$. The mod $2$ gradings of these generators are given by
\[
\begin{split}
|\widecheck{\gamma}| &= CZ(\gamma),\\
|\widehat{\gamma}| &= CZ(\gamma)+1.
\end{split}
\]

\begin{definition}
\label{def:cascadediff}
Define the differential
\[
\partial^\J_\ca: NCC_*(Y,\lambda) \longrightarrow NCC_{*-1}(Y,\lambda)
\]
as follows. Let $\alpha$ and $\beta$ be Reeb orbits, let $\widetilde{\alpha}$ denote either $\widehat{\alpha}$ or $\widecheck{\alpha}$, and let $\widetilde{\beta}$ denote either $\widehat{\beta}$ or $\widecheck{\beta}$. Then the differential coefficient $\left\langle\partial^\J_\ca\widetilde{\alpha},\widetilde{\beta}\right\rangle\in\Z$ is the signed count of points in the cascade moduli space $\mc{M}^\ca_0\left(\widetilde{\alpha},\widetilde{\beta}\right)$. Here the signs are determined by the fixed orientations of $\mc{O}_\alpha(p_\alpha)$ and $\mc{O}_\beta(p_\beta)$, together with the orientation of  $\mc{M}^\ca_0\left(\widetilde{\alpha},\widetilde{\beta}\right)$ with values in $\mc{O}_\alpha(p_\alpha)\tensor \mc{O}_\beta(p_\beta)$.
\end{definition}

\begin{lemma}
\label{lem:dwd}
The differential $\partial^\J_\ca$ is well-defined, decreases the mod $2$ grading by $1$, and satisfies $(\partial^\J_\ca)^2=0$.
\end{lemma}

\begin{proof}
The coefficient $\left\langle\partial^\J_\ca\widetilde{\alpha},\widetilde{\beta}\right\rangle$ is well defined by Proposition~\ref{prop:cascadekey}(a). Since for a given Reeb orbit $\alpha$ there are only finitely many Reeb orbits $\beta$ with symplectic action less than that of $\alpha$, we conclude that $\partial^\J_\ca\widetilde{\alpha}$ is well-defined.

Equation \eqref{eqn:gradingaxiom} implies that the differential $\partial^\J_\ca$ decreases the mod $2$ grading by $1$.

It follows from Proposition~\ref{prop:cascadekey}(b) that $\left(\partial^\J_\ca\right)^2=0$.
\end{proof}

In the terminology of \cite{td}, the homology of the chain complex $\left(NCC_*(Y,\lambda),\partial^\J_\ca\right)$ is the cascade homology of the Morse-Bott system in Proposition~\ref{prop:mbs}. Consequently, Theorem~\ref{thm:td}(a) implies that this homology does not depend on the choice of base points $\mc{P}$, so we can denote it by $NCH_*(Y,\lambda;\J)$, and we call this nonequivariant contact homology. Our next task is to use the rest of Theorem~\ref{thm:td} to show that that NCH in fact depends only on $(Y,\xi)$. (For the computation of NCH in our main example of interest, see Proposition~\ref{prop:nchaut} below.)

\subsection{Cobordism maps}
\label{sec:nchcobmaps}

To prove invariance of nonequivariant contact homology, we need to consider maps induced by certain symplectic cobordisms.

\begin{definition}
\label{def:exactcob}
Let $(Y_+,\lambda_+)$ and $(Y_-,\lambda_-)$ be closed contact manifolds of dimension $2n-1$. An {\em exact symplectic cobordism\/} from\footnote{Instead of using the words ``from'' and ``to'', one could say that $(Y_+,\lambda_+)$ is the convex boundary of $(X,\lambda)$, and $(Y_-,\lambda_-)$ is the concave boundary.} $(Y_+,\lambda_+)$ to $(Y_-,\lambda_-)$ is a pair $(X,\lambda)$ where $X$ is a compact $2n$-dimensional oriented manifold with $\partial X = Y_+ - Y_-$, and $d\lambda$ is a symplectic form on $X$ with $\lambda|_{Y_\pm}=\lambda_\pm$.
\end{definition}

The goal of this subsection is to prove the following:

\begin{proposition}
\label{prop:cobordismmap}
Let $(X,\lambda)$ be an exact symplectic cobordism from $(Y_+,\lambda_+)$ to $(Y_-,\lambda_-)$, where the contact forms $\lambda_\pm$ are nondegenerate and hypertight. Assume further that every Reeb orbit for $\lambda_+$ is noncontractible in $X$. Let $\J_\pm$ be $S^1$-families of almost complex structures as needed to define the nonequivariant contact homology of $(Y_\pm,\lambda_\pm)$. Then $(X,\lambda)$ induces a well-defined map
\[
\Phi(X,\lambda;\J_+,\J_-): NCH_*(Y_+,\lambda_+;\J_+) \longrightarrow NCH_*(Y_-,\lambda_-;\J_-).
\]
\end{proposition}

\begin{proof}
The strategy is to construct a ``morphism of Morse-Bott systems'' and invoke Theorem~\ref{thm:td}(b).

To set up the construction, recall that an exact symplectic cobordism $(X,\lambda)$ has a canonical Liouville vector field $V$ characterized by $\imath_Vd\lambda = \lambda$. The vector field $V$ points out of $X$ on $Y_+$ and into $X$ on $Y_-$. For $\epsilon>0$ small, the flow of $V$ then determines neighborhoods $N_\pm$ of $Y_\pm$ in $X$ with identifications
\begin{equation}
\label{eqn:LN}
\begin{split}
(N_+,\lambda) & \simeq \left((-\epsilon,0]\times Y_+,e^r\lambda_+\right),\\
(N_-,\lambda) & \simeq \left([0,\epsilon)\times Y_-,e^r\lambda_-\right)
\end{split}
\end{equation}
where $r$ denotes the $(-\epsilon,0]$ or $[0,\epsilon)$ coordinate. Here $V$ increases the first coordinate at unit speed, and $Y_\pm\subset N_\pm$ is identified with $\{0\}\times Y_\pm$. We now define the {\em completion\/}
\[
\overline{X} = \big((-\infty,0]\times Y_-\big) \cup_{Y_-} X \cup_{Y_+} \big([0,\infty)\times Y_+\big),
\]
glued using the neighborhood identifications \eqref{eqn:LN}.

\begin{definition}
Let $(X,\lambda)$ be an exact symplectic cobordism from $(Y_+,\lambda_+)$ to $(Y_-,\lambda_-)$. An almost complex structure $J$ on $\overline{X}$ is {\em cobordism-compatible\/} if :
\begin{itemize}
\item
$J$ agrees on $[0,\infty)\times Y_+$ with the restriction of a $\lambda_+$-compatible almost complex structure $J_+$ on $\R\times Y_+$.
\item
$J$ agrees on $(-\infty,0]\times Y_-$ with the restriction of a $\lambda_-$-compatible almost complex structure $J_-$ on $\R\times Y_-$.
\item
$J$ is compatible with the symplectic form $d\lambda$ on $X$.
\end{itemize}
\end{definition}

Now let $\J_+=\{J_{+,t}\}_{t\in S^1}$ be an $S^1$-family of $\lambda_+$-compatible almost complex structures on $\R\times Y_+$, and let $\J_-=\{J_{-,t}\}_{t\in S^1}$ be an $S^1$-family of $\lambda_-$-compatible almost complex structures on $\R\times Y_-$. Let $\J=\{J_t\}_{t\in S^1}$ be an $S^1$-family of cobordism-compatible almost complex structures on $\overline{X}$, such that $J_t$ agrees with $J_{+,t}$ on $[0,\infty)\times Y_+$ and with $J_{-,t}$ on $(-\infty,0]\times Y_-$. Note that given $\J_+$ and $\J_-$, the space of $\J$ is contractible.

If $\gamma_\pm$ are Reeb orbits for $\lambda_\pm$, let $\Phi^\J(\gamma_+,\gamma_-)$ denote the moduli spaces of maps $u:\R\times S^1\to\overline{X}$ satisfying the conditions
\begin{gather}
\label{eqn:cobmoduli}
\partial_s u + J_t\partial_t u = 0,\\
\nonumber
\mbox{$u(s,t)\in[0,\infty)\times Y_+$ for $s>>0$,}\\
\nonumber
\mbox{$u(s,t)\in(-\infty,0]\times Y_-$ for $s<<0$,}\\
\nonumber
\lim_{s\to\pm\infty}\pi_\R(u(s,\cdot)) = \pm\infty,\\
\label{eqn:cobmoduliend}
\mbox{$\lim_{s\to\pm\infty}\pi_{Y_\pm}(u(s,\cdot))$ is a parametrization of $\gamma_\pm$,}
\end{gather}
modulo $\R$ translation in the domain. (Note that unlike with the moduli spaces $\widetilde{\M}$ in Definition~\ref{def:Mtilde}, there is now no $\R$ action on the target to mod out by.)

As before, we have evaluation maps
\[
e_\pm:\Phi^\J(\gamma_+,\gamma_-) \longrightarrow \overline{\gamma_\pm}
\]
defined by
\[
e_\pm(u) = \lim_{s\to\pm\infty}\pi_{Y_\pm}(u(s,0)).
\]

Suppose now that the contact forms $\lambda_\pm$ are nondegenerate. If $d$ is an integer, let $\Phi^\J_d(\gamma_+,\gamma_-)$ denote the set of $u\in\Phi^\J(\gamma_+,\gamma_-)$ with
\[
\CZ_\tau(\gamma_+) - \CZ_\tau(\gamma_-) + 2c_1(u^*T\overline{X},\tau) = d-1.
\]

\begin{lemma}
\label{lem:morphism1}
Suppose that $\lambda_\pm$ are nondegenerate and that $\J_\pm$ and $\J$ are generic. Then:
\begin{description}
\item{(a)} For any Reeb orbits $\gamma_\pm$ and any integer $d$, the moduli space $\Phi_d^\J(\gamma_+,\gamma_-)$ is a smooth manifold of dimension $d$, and the evaluation maps $e_\pm$ on it are smooth.
\item{(b)} The moduli space $\Phi_d^\J(\gamma_+,\gamma_-)$ has a canonical orientation with values in $e_+^*\mc{O}_{\gamma_+}\tensor e_-^*\mc{O}_{\gamma_-}$.
\item{(c)} Let $k,l\ge 0$, let $\gamma_0^+,\ldots,\gamma_k^+$ be distinct Reeb orbits for $\lambda_+$, let $\gamma_0^-,\ldots,\gamma_l^-$ be distinct Reeb orbits for $\lambda_-$, and let $d_0,d_1^+,\ldots,d_k^+,d_1^-,\ldots,d_l^-$ be nonnegative integers. Then the fiber product
\begin{gather}
\nonumber
\M_{d_k^+}^{\J_+}(\gamma_k^+,\gamma_{k-1}^+) \times_{\overline{\gamma_{k-1}^+}} \cdots \times_{\overline{\gamma_1^+}} \M_{d_1^+}^{\J_+}(\gamma_1^+,\gamma_0^+)\\
\label{eqn:bigfiberproduct}
\times_{\overline{\gamma_0^+}} \Phi^\J_{d_0}(\gamma_0^+,\gamma_0^-) \times_{\overline{\gamma_0^-}}\\
\nonumber
\M_{d_1^-}^{\J_-}(\gamma_0^-,\gamma_1^-) \times_{\overline{\gamma_1^-}} \cdots \times_{\overline{\gamma_{l-1}^-}} \M_{d_l^-}^{\J_-}(\gamma_{l-1}^-,\gamma_l^-)
\end{gather}
is cut out transversely.
\end{description}
\end{lemma}

\begin{proof}
Parts (a) and (c) are standard transversality arguments, similar to Proposition~\ref{prop:transversality}. The orientation convention for part (b) is explained at the beginning of \S\ref{sec:cobsigns}.
\end{proof}

\begin{definition}
Analogously to Definition~\ref{def:Mdbar}, let $\overline{\Phi_d^\J}(\gamma_+,\gamma_-)$ denote the union of all fiber products \eqref{eqn:bigfiberproduct} with $\gamma_k^+=\gamma_+$, $\gamma_l^-=\gamma_-$, and $\sum_{i=1}^kd_i^+ + d_0 + \sum_{j=1}^ld_j^-=d$.
Define the evaluation maps
\[
e_\pm:\overline{\Phi_d^\J}(\gamma_+,\gamma_-) \longrightarrow \overline{\gamma_\pm}
\]
and the topology on $\overline{\Phi_d^\J}(\gamma_+,\gamma_-)$ as before.
\end{definition}

Under the assumptions of Proposition~\ref{prop:cobordismmap}, $\overline{\Phi_d^\J}(\gamma_+,\gamma_-)$ is compact, similarly to Proposition~\ref{prop:compactness}. Note that the extra hypothesis that every Reeb orbit of $\lambda_+$ is noncontractible in $X$ is needed to avoid bubbling of holomorphic planes.

We now have the following compactness and gluing result, which is analogous to Proposition~\ref{prop:currentcompact}, except that the signs are slightly different.

\begin{lemma}
\label{lem:morphism2}
Under the assumptions of Proposition~\ref{prop:cobordismmap}, let $\gamma_\pm$ be Reeb orbits for $\lambda_\pm$. Then:
\begin{description}
\item{\emph{(a)}} $\Phi^\J_0(\gamma_+,\gamma_-)$ is finite.

\item{\emph{(b)}} $\overline{\Phi^\J_1}(\gamma_+,\gamma_-)$ is a compact oriented topological one-manifold with oriented boundary
\begin{equation}
\label{eqn:phiboundary}
\begin{split}
\partial\overline{\Phi^{\J}_1}(\gamma_+,\gamma_-) = & \coprod_{\substack{\gamma_+'\neq\gamma_+\\d_++d=1}} \M^{\J_+}_{d_+}(\gamma_+,\gamma_+')\times_{\overline{\gamma_+'}}\Phi^\J_{d}(\gamma_+',\gamma_-)\\
&\bigsqcup \coprod_{\substack{\gamma_-'\neq\gamma_-\\d+d_-=1}} (-1)^d \Phi^\J_d(\gamma_+,\gamma_-') \times_{\overline{\gamma_-'}} \M^{\J_-}_{d_-}(\gamma_-',\gamma_-).
\end{split}
\end{equation}

\item{\emph{(c)}}
We also have analogues of \eqref{eqn:phiboundary} with point constraints as in \cite[Eqs.\ (2.13)-(2.15)]{td}.

\end{description}
\end{lemma}

\begin{proof}
Part (a) is similar to Proposition~\ref{prop:currentcompact}(a). Parts (b) and (c) follow from Proposition~\ref{prop:cobsigns}.
\end{proof}

Lemmas~\ref{lem:morphism1} and \ref{lem:morphism2} imply, as in Proposition~\ref{prop:mbs}, that the moduli spaces $\Phi_d^\J(\gamma_+,\gamma_-)$ constitute a ``morphism'', in the sense of \cite[Def.\ 2.7]{td}, from the Morse-Bott system for $(Y_+,\lambda_+;\J_+)$ to the Morse-Bott system for $(Y_-,\lambda_-;\J_-)$. It then follows from Theorem~\ref{thm:td}(b)(i) that we have an induced map
\[
\Phi(X,\lambda;\J): NCH_*(Y_+,\lambda_+;\J_+)\longrightarrow NCH_*(Y_-,\lambda_-;\J_-).
\] 

To complete the proof of Proposition~\ref{prop:cobordismmap}, we need to show that this map does not depend on the choice of generic $S^1$-family of cobordism-compatible complex structures $\J$ restricting to $\J_+$ and $\J_-$. For this purpose, let $\J^0$ and $\J^1$ be two such $S^1$-families of cobordism-compatible almost complex structures, and let $\{\J^\tau\}_{\tau\in[0,1]}$ be a generic homotopy between them.
Given Reeb orbits $\gamma_\pm$ for $\lambda_\pm$, let $K(\gamma_+,\gamma_-)$ denote the set of pairs $(\tau,u)$ where $\tau\in(0,1)$ and $u:\R\times S^1\to \overline{X}$ satisfies the conditions \eqref{eqn:cobmoduli} with $J_t$ replaced by $J^\tau_t$. Here, as usual, we mod out by $\R$ translation in the domain. And again, we have evaluation maps
\[
e_\pm: K(\gamma_+,\gamma_-) \longrightarrow \overline{\gamma_\pm}
\]
defined by
\[
e_\pm(\tau,u) = \lim_{s\to\pm\infty}\pi_{Y_\pm}(u(s,0)).
\]
Continue to assume that the contact forms $\lambda_\pm$ are nondegenerate. If $d$ is an integer, let $K_d(\gamma_+,\gamma_-)$ denote the set of $(\tau,u)\in K(\gamma_+,\gamma_-)$ with
\[
\CZ_\tau(\gamma_+) - \CZ_\tau(\gamma_-) + 2c_1(u^*T\overline{X},\tau)  = d - 2.
\]

We now have the following lemma which is similar to Lemmas~\ref{lem:morphism1} and \ref{lem:morphism2}; we omit the proof.

\begin{lemma}
\label{lem:homotopy}
Under the assumptions of Proposition~\ref{prop:cobordismmap}, suppose that $\J^0$, $\J^1$, and $\{\J^\tau\}$ are generic, and let $\gamma_\pm$ be Reeb orbits for $\lambda_\pm$. Then:
\begin{description}
\item{(a)}
The moduli space $K_d(\gamma_+,\gamma_-)$ is a smooth manifold of dimension $d$ with a canonical orientation taking values in $e_+^*\mc{O}_{\gamma_+}\tensor e_-^*\mc{O}_{\gamma_-}$, and the evaluation maps $e_\pm$ on it are smooth.
\item{(b)} Transversality as in \eqref{eqn:bigfiberproduct} holds, with $\Phi^\J_{d_0}(\gamma_0^+,\gamma_0^-)$ replaced by $K_{d_0}(\gamma_0^+,\gamma_0^-)$.
\item{(c)} $K_0(\gamma_+,\gamma_-)$ is finite.
\item{(d)} $K_1(\gamma_+,\gamma_-)$ has a compactification to a compact oriented topological one-manifold $\overline{K_1}(\gamma_+,\gamma_-)$ with oriented boundary
\begin{equation}
\label{eqn:Kboundary}
\begin{split}
\partial \overline{K_1}(\gamma_+,\gamma_-) =&
-\Phi^{\J^0}_0(\gamma_+,\gamma_-) \bigsqcup \Phi^{\J^1}_0(\gamma_+,\gamma_-)\\
& \bigsqcup \coprod_{\substack{\gamma_+'\neq\gamma_+\\d_++d=1}} (-1)^{d_+}\M^{\J_+}_{d_+}(\gamma_+,\gamma_+')\times_{\overline{\gamma_+'}}K_{d}(\gamma_+',\gamma_-)\\
&\bigsqcup \coprod_{\substack{\gamma_-'\neq\gamma_-\\d+d_-=1}} (-1)^d K_d(\gamma_+,\gamma_-') \times_{\overline{\gamma_-'}} \M^{\J_-}_{d_-}(\gamma_-',\gamma_-).
\end{split}
\end{equation}
\item{(e)}
We also have analogues of \eqref{eqn:Kboundary} with point constraints as in \cite[Eqs.\ (2.35)-(2.37)]{td}.
\end{description}
\end{lemma}

Lemma~\ref{lem:homotopy} implies that the moduli spaces $K_d(\gamma_+,\gamma_-)$ constitute a ``homotopy'', in the sense of \cite[Def.\ 2.15]{td}, between the morphisms of Morse-Bott systems induced by $\J^0$ and $\J^1$. It then follows from Theorem~\ref{thm:td}(b)(iv) that
\[
\Phi(X,\lambda;\J^0) = \Phi(X,\lambda;\J^1) : NCH_*(Y_+,\lambda_+;\J_+)\longrightarrow NCH_*(Y_-,\lambda_-;\J_-).
\] 
This completes the proof of Proposition~\ref{prop:cobordismmap}.
\end{proof}

\subsection{Invariance of NCH}
\label{sec:nchinv}

The cobordism maps in Proposition~\ref{prop:cobordismmap} have two important properties which we will need to prove invariance of nonequivariant contact homology.

We first consider scaling the contact form. Let $\lambda$ be a nondegenerate contact form on $Y$ with Reeb vector field $R$. If $c>0$, then $c\lambda$ is also a nondegenerate contact form on $Y$, with Reeb vector field $c^{-1}R$. Thus there is a canonical bijection between Reeb orbits of $\lambda$ and Reeb orbits of $c\lambda$; if $\gamma$ is a Reeb orbit of $\lambda$, we denote the corresponding Reeb orbit of $c\lambda$ by ${^c}\gamma$.

Let $\J=\{J_t\}$ be a generic $S^1$-family of $\lambda$-compatible almost complex structures as needed to define the nonequivariant contact homology $NCH_*(Y,\lambda;\J)$. There is then a unique $S^1$-family ${^c}\J=\{{^c}J_t\}$ of $c\lambda$-compatible almost complex structures which agrees with $\J$ on the contact distribution $\xi=\Ker(\lambda)=\Ker(c\lambda)$. The diffeomorphism $\phi$ of $\R\times Y$ sending $(r,y)\mapsto (cr,y)$ satisfies $d\phi\circ J_t = {{^c}J_t}\circ d\phi$. Thus for each pair $\gamma_+,\gamma_-$ of distinct Reeb orbits of $\gamma$, we obtain a canonical diffeomorphism of moduli spaces
\begin{equation}
\label{eqn:scop}
\M^\J_d(\gamma_+,\gamma_-)\simeq \M^{{^c}\J}({^c}\gamma_+,{^c}\gamma_-)
\end{equation}
sending $[u]\mapsto [\phi\circ u]$. This diffeomorphism preserves the orientations (see Lemma~\ref{lem:scop}) and evaluation maps. As a result, we have a canonical isomorphism of chain complexes
\[
\left( NCC_*(Y,\lambda),\partial^\J_\ca\right) = \left( NCC_*(Y,c\lambda),\partial^{{^c}\J}_\ca \right).
\]
We denote the induced map on homology by
\begin{equation}
\label{eqn:scaling}
s_c: NCH_*(Y,\lambda;\J) \stackrel{\simeq}{\longrightarrow} NCH_*(Y,c\lambda;{{^c}\J}).
\end{equation}

We also need to consider composition of cobordisms. If $(X_+,\lambda_+)$ is an exact symplectic cobordism from $(Y_1,\lambda_1)$ to $(Y_2,\lambda_2)$, and if $(X_-,\lambda_-)$ is an exact symplectic cobordism from $(Y_2,\lambda_2)$ to $(Y_3,\lambda_3)$, then we can form the composite cobordism
\[
X_-\circ X_+ = X_-\sqcup_{Y_2} X_+,
\]
glued using the neighborhood identifications \eqref{eqn:LN}. We define a $1$-form $\lambda$ on $X_-\circ X_+$ to agree with $\lambda_\pm$ on $X_\pm$, and this makes $(X_-\circ X_+,\lambda)$ into an exact symplectic cobordism from $(Y_1,\lambda_1)$ to $(Y_3,\lambda_3)$.

We can now state:

\begin{proposition}
\label{prop:cobordismmapproperties}
The cobordism maps in Proposition~\ref{prop:cobordismmap} have the following properties:
\begin{description}
\item{(a) (Scaling)}
Suppose $(Y,\lambda_0)$ is nondegenerate and hypertight. Let $\J$ be an $S^1$-family of $\lambda_0$-compatible almost complex structures as needed to define $NCH_*(Y,\lambda_0;\J)$. Consider the trivial cobordism
\begin{equation}
\label{eqn:trivcob}
(X,\lambda) = \left([a,b]\times Y,e^r\lambda_0\right).
\end{equation}
Then the cobordism map
\[
\Phi(X,\lambda;{^{e^a}\J},{^{e^b}\J}): NCH_*(Y,e^b\lambda_0;{^{e^b}\J}) \longrightarrow NCH_*(Y,e^a\lambda_0;{^{e^a}\J})
\]
agrees with the scaling isomorphism $s_{e^{a-b}}$ in \eqref{eqn:scaling}.
\item{(b) (Composition)}
Let $(Y_i,\lambda_i)$ be nondegenerate and hypertight, and let $\J_i$ be an $S^1$-family of $\lambda_i$-compatible almost complex structures as needed to define $NCH_*(Y,\lambda_i;\J_i)$, for $i=1,2,3$. Let $(X_+,\lambda_+)$ be an exact symplectic cobordism from $(Y_1,\lambda_1)$ to $(Y_2,\lambda_2)$, and let $(X_-,\lambda_-)$ be an exact symplectic cobordism from $(Y_2,\lambda_2)$ to $(Y_3,\lambda_3)$. Assume further that every Reeb orbit for $\lambda_1$ is noncontractible in $X_-\circ X_+$, and every Reeb orbit for $\lambda_2$ is noncontractible in $X_-$. Then
\[
\Phi(X_-\circ X_+,\lambda;\J_3,\J_1) = \Phi(X_-,\lambda_-;\J_3,\J_2) \circ \Phi(X_+,\lambda_+;\J_2,\J_1).
\]
\end{description}
\end{proposition}

Proposition~\ref{prop:cobordismmapproperties} will be proved in \S\ref{sec:scaling} and \S\ref{sec:composition} below. Meanwhile, we now use Proposition~\ref{prop:cobordismmapproperties} to deduce the invariance of nonequivariant contact homology by a simple formal argument.

\begin{proof}[Proof of Theorem~\ref{thm:NCHinvariant}.] 
Let $\lambda_1$ and $\lambda_2$ be nondegenerate hypertight contact forms on $Y$ with $\Ker(\lambda_1)=\Ker(\lambda_2)$. Let $\J_1$ and $\J_2$ be generic $S^1$-families of almost complex structures as needed to define the nonequivariant contact homology of $\lambda_1$ and $\lambda_2$. We define a map
\begin{equation}
\label{eqn:phils}
\phi_{(\lambda_2;\J_2),(\lambda_1;\J_1)}: NCH_*(Y,\lambda_1;\J_1) \longrightarrow NCH_*(Y,\lambda_2;\J_2)
\end{equation}
as follows. We know that $\lambda_1=e^f\lambda_2$ for some smooth function $f:Y\to\R$. Pick a sufficiently large constant $c$ so that $c+f>0$ on all of $Y$. We then have an exact symplectic cobordism $(X,\lambda)$ from $(Y,e^c\lambda_1)$ to $(Y,\lambda_2)$ given by
\begin{equation}
\label{eqn:cobforinv}
\begin{split}
X &= \{(r,y)\in\R\times Y \mid 0\le r \le f(y)+c\},\\
\lambda &= e^r\lambda_2.
\end{split}
\end{equation}
We define the map \eqref{eqn:phils} to be the composition of the scaling isomorphism
\[
s_{e^c}: NCH_*(Y,\lambda_1;\J_1) \stackrel{\simeq}{\longrightarrow} NCH_*(Y,e^c\lambda_1;{^{e^c}\J_1})
\]
with the cobordism map
\[
\Phi(X,\lambda;\J_2,{^{e^c}\J_1}): NCH_*(Y,e^c\lambda_1;{^{e^c}\J_1}) \longrightarrow NCH_*(Y,\lambda_2;\J_2).
\]

We now prove that:

\begin{description}
\item{(i)} The map \eqref{eqn:phils} does not depend on the choice of constant $c$ used to define it.
\item{(ii)} When $\lambda_1=\lambda_2$ and $\J_1=\J_2$, we have
\[
\phi_{(\lambda;\J),(\lambda;\J)} = \op{id}_{NCH_*(Y,\lambda;\J)}.
\]
\item{(iii)} If $\lambda_3$ is another nondegenerate hypertight contact form and $\J_3$ is a generic $S^1$-family of $\lambda_3$-compatible almost complex structures, then 
\begin{equation}
\label{eqn:phicomp}
\phi_{(\lambda_3;\J_3),(\lambda_2;\J_2)}\circ \phi_{(\lambda_2;\J_2),(\lambda_1;\J_1)} = \phi_{(\lambda_3;\J_3),(\lambda_1;\J_1)}.
\end{equation}
\end{description}
The above three properties imply that the maps \eqref{eqn:phils} canonically identify the nonequivariant contact homologies for different choices of $\lambda$ and $\J$ with each other.

(i) To prove that the map \eqref{eqn:phils} does not depend on the choice of $c$, suppose that $c'>c$, and let $(X',\lambda)$ denote the cobordism \eqref{eqn:cobforinv} defined using $c'$ instead of $c$. We then have a commutative diagram
\[
\begin{CD}
NCH_*(Y,\lambda_1;\J_1) @>{s_{e^{c'}}}>> NCH_*(Y,e^{c'}\lambda_1;{^{e^{c'}}\J_1}) @>{\Phi(X',\lambda;\J_2,{^{e^{c'}}\J_1})}>> NCH_*(Y,\lambda_2;\J_2)\\
@| @V{s_{e^{c-c'}}}VV  @|\\
NCH_*(Y,\lambda_1;\J_1) @>{s_{e^c}}>> NCH_*(Y,e^{c}\lambda_1;{^{e^c}\J_1}) @>{\Phi(X,\lambda;\J_2,{^{e^c}\J_1})}>> NCH_*(Y,\lambda_2;\J_2).
\end{CD}
\]
Here the top row is the map \eqref{eqn:phils} defined using $c'$, and the bottom row is the map \eqref{eqn:phils} defined using $c$. The left square commutes because the composition of two scaling isomorphisms is, by definition, a scaling isomorphism. Commutativity of the right square follows from both parts of Proposition~\ref{prop:cobordismmapproperties} and the fact that the cobordism $X'$ is the composition of $X$ with the trivial cobordism
\[
\big(\{(r,y)\in\R\times Y \mid f(y)+c \le r \le f(y)+c'\}, e^r\lambda_2\big) \simeq \big([c,c']\times Y, e^r(e^c\lambda_1)\big).
\]

(ii) This follows from the Scaling property in Proposition~\ref{prop:cobordismmapproperties}.

(iii) Write $\lambda_1 = e^{f_1}\lambda_2$ and $\lambda_2=e^{f_2}\lambda_3$. By arguments as in the proof of part (i), we can assume without loss of generality that the contact forms have been scaled so that $f_1>f_2>0$ everywhere. We can then define all of the maps in \eqref{eqn:phicomp} using $c=0$. Equation \eqref{eqn:phicomp} now follows from the Composition property in Proposition~\ref{prop:cobordismmapproperties}.
\end{proof}

\subsection{Proof of the scaling property}
\label{sec:scaling}

\begin{proof}[Proof of Proposition~\ref{prop:cobordismmapproperties}(a).] 
We can identify the completion $\overline{X}$ of the trivial cobordism \eqref{eqn:trivcob} with $\R\times Y$, so that $(-\infty,0]\times Y$ is identified with $(-\infty,a]\times Y$ by shifting the $\R$ coordinate by $a$, and $[0,\infty)\times Y$ is identified with $[b,\infty)\times Y$ by shifting the $\R$ coordinate by $b$. We now define an $S^1$-family $\J^X=\{J^X_t\}_{t\in S^1}$ of cobordism-compatible almost complex structures on $\overline{X}$ as follows. Choose a positive function $f:\R\to\R$ with $f(r)=e^{-a}$ for $r\le a$ and $f(r)=e^{-b}$ for $r\ge b$. There is then a unique cobordism-compatible almost complex structure $J^X_t$ on $\overline{X}$ such that
\[
J^X_t(v) = J_t(v)
\]
for $v\in\xi=\Ker(\lambda)$, and
\[
J^X_t(\partial_r) = f(r)R.
\]

Now let $g:\R\to\R$ be an antiderivative of $f$. Then the diffeomorphism $\phi$ of $\R\times Y$ sending $(r,y)\mapsto (g(r),y))$ satisfies $d\phi\circ J^X_t = J_t\circ d\phi$. Thus if $\gamma_+$ and $\gamma_-$ are Reeb orbits of $\gamma$, we obtain a diffeomorphism of moduli spaces
\[
\Phi^{\J^X}\left({^{e^b}}\gamma_+,{^{e^a}}\gamma_-\right) \simeq \widetilde{\M}^\J(\gamma_+,\gamma_-).
\]
If $\gamma_+\neq\gamma_-$, then after choosing a smooth slice of the $\R$ action on the right hand side, we obtain a diffeomorphism
\begin{equation}
\label{eqn:scacob1}
\Phi^{\J^X}_d \left({^{e^b}}\gamma_+,{^{e^a}}\gamma_-\right) \simeq \R\times \M^\J_{d-1}(\gamma_+,\gamma_-).
\end{equation}
And if $\gamma_+=\gamma_-$, then we have a canonical diffeomorphism
\begin{equation}
\label{eqn:scacob2}
\Phi^{\J^X}_d \left({^{e^b}}\gamma,{^{e^a}}\gamma\right) = \left\{\begin{array}{cl} \overline{\gamma}, & d=1,\\
\emptyset, & d\neq 1.
\end{array}\right.
\end{equation}
Moreover, the diffeomorphisms \eqref{eqn:scacob1} and \eqref{eqn:scacob2} are orientation preserving, as shown in Lemma~\ref{lem:scacob}. These orientation preserving diffeomorphisms imply that in the terminology of \cite[Ex.\ 2.8]{td}, the pushforwards under $\phi$ of the moduli spaces $\Phi^{\J^X}_d({^{e^b}}\gamma_+,{^{e^a}}\gamma_-)$ constitute the identity morphism on the Morse-Bott system determined by $(Y,\lambda;\J)$. It then follows from Theorem~\ref{thm:td}(b)(ii) that $\Phi(X,\lambda;\J^X)$ agrees with the scaling isomorphism $s_{e^{a-b}}$.
\end{proof}

\subsection{Proof of the composition property}
\label{sec:composition}

\begin{proof}[Proof of Proposition~\ref{prop:cobordismmapproperties}(b).]
If $R\ge 0$ is a nonnegative real number (here $R$ does not denote a Reeb vector field), define a ``stretched composition''
\[
X_-\circ_R X_+ = X_-\sqcup_{Y_2} \big([-R,R]\times Y_2\big) \sqcup_{Y_2} X_+.
\]
Define a $1$-form $\lambda_R$ on $X_-\circ_R X_+$ by
\[
\lambda_R = \left\{\begin{array}{cl} e^{-R}\lambda_- & \mbox{on $X_-$},\\
e^r\lambda_2 & \mbox{on $[-R,R]\times Y_2$},\\ e^R\lambda_+ & \mbox{on $X_+$}.
\end{array}\right.
\]
This makes $(X_-\circ_R X_+,\lambda_R)$ into an exact symplectic cobordism from $(Y_1,e^R\lambda_1)$ to $(Y_3,e^{-R}\lambda_3)$. 

Generically choose $S^1$-families $\J^\pm=\{J^\pm_t\}_{t\in S^1}$ of cobordism-compatible almost complex structures on $\overline{X_\pm}$ that agree with $\J_2$ near $Y_2$. Define an $S^1$-family $\J^R=\{J^R_t\}_{t\in S^1}$ of almost complex structures on $\overline{X_-\circ_R X_+}$ by
\[
J^R_t = \left\{\begin{array}{cl} J^-_t & \mbox{on $((-\infty,0]\times Y_3)\sqcup_{Y_3}X_-$},\\
J_{2,t} & \mbox{on $[-R,R]\times Y_2$},\\
J^+_t & \mbox{on $X_+\sqcup_{Y_1}([0,\infty)\times Y_1)$}.
\end{array}\right.
\]
Note that $J^R_t$ is not quite cobordism-compatible for $R\neq 0$, because on $(-\infty,0]\times Y_3$ and $[0,\infty)\times Y_1$, we have that $J^R_t(\partial_r)$ is $e^{\pm R}$ times what it should be. However this does not affect our arguments.

If $\gamma_+$ is a Reeb orbit for $\lambda_1$ and $\gamma_-$ is a Reeb orbit for $\gamma_3$, define $K(\gamma_+,\gamma_-)$ to be the set of pairs $(R,u)$ where $R>0$ and $u:\R\times S^1\to \overline{X_-\circ_R X_+}$ satisfies
\begin{gather*}
\partial_s u + J^R_t\partial_t u = 0,\\
\mbox{$u(s,t)\in[0,\infty)\times Y_1$ for $s>>0$,}\\
\mbox{$u(s,t)\in(-\infty,0]\times Y_3$ for $s<<0$,}\\
\mbox{$\lim_{s\to+\infty}\pi_{Y_1}(u(s,\cdot))$ is a parametrization of $\gamma_+$,}\\
\mbox{$\lim_{s\to-\infty}\pi_{Y_3}(u(s,\cdot))$ is a parametrization of $\gamma_-$,}
\end{gather*}
modulo $\R$ translation in the domain. As usual we have evaluation maps
\[
e_\pm: K(\gamma_+,\gamma_-)\longrightarrow \overline{\gamma_\pm}.
\]
If $d$ is an integer, let $K_d(\gamma_+,\gamma_-)$ denote the set of $(R,u)\in K(\gamma_+,\gamma_-)$ with
\[
\CZ_\tau(\gamma_+) - \CZ_\tau(\gamma_-) + 2c_1(u^*T(\overline{X_-\circ_R X_+}),\tau) = d-2.
\]
(Here we are continuing to assume that $\lambda_\pm$ are nondegenerate so that this makes sense.)
Similarly to Lemma~\ref{lem:homotopy}, we have:

\begin{lemma}
\label{lem:homotopy2}
Under the assumptions of Proposition~\ref{prop:cobordismmapproperties}, suppose that $\J^+$ and $\J^-$ are generic. Given $L>0$, there exists $R_0\ge 0$ such that after a small perturbation of the family $\{\J^R\}_{R\ge 0}$ supported where $R\le R_0$, the following is true. Let $\gamma_+$ be a Reeb orbit for $\lambda_1$ with action less than $L$, and let $\gamma_-$ be a Reeb orbit for $\lambda_3$. Then:
\begin{description}
\item{(a)}
The moduli space $K_d(\gamma_+,\gamma_-)$ is a smooth manifold of dimension $d$ with a canonical orientation taking values in $e_+^*\mc{O}_{\gamma_+}\tensor e_-^*\mc{O}_{\gamma_-}$, and the evaluation maps $e_\pm$ on it are smooth.
\item{(b)} Transversality as in \eqref{eqn:bigfiberproduct} holds, with $\J_+$ replaced by $\J_1$, with $\J_-$ replaced by $\J_3$, and with $\Phi^\J_{d_0}(\gamma_0^+,\gamma_0^-)$ replaced by $K_{d_0}(\gamma_0^+,\gamma_0^-)$.
\item{(c)} $K_0(\gamma_+,\gamma_-)$ is finite.
\item{(d)} If $\gamma_0$ is a Reeb orbit for $\lambda_2$, then the fiber product
\[
\Phi^{\J^+}_{d_+}(\gamma_+,\gamma_0)\times_{\overline{\gamma_0}} \Phi^{\J^-}_{d_-}(\gamma_0,\gamma_-)
\]
is cut out transversely. The same also holds for more general such fiber products in which some moduli spaces $\mc{M}^{\J_1}$, $\mc{M}^{\J_2}$, and/or $\mc{M}^{\J_3}$ are inserted.
\item{(e)} $K_1(\gamma_+,\gamma_-)$ has a compactification to a compact oriented topological one-manifold $\overline{K_1}(\gamma_+,\gamma_-)$ with oriented boundary
\begin{equation}
\label{eqn:K2boundary}
\begin{split}
\partial \overline{K_1}(\gamma_+,\gamma_-) =&
-\Phi^{\J^0}_0(\gamma_+,\gamma_-)\\
&\bigsqcup \coprod_{\gamma_0}\Phi^{\J^+}_0(\gamma_+,\gamma_0)\times_{\overline{\gamma_0}} \Phi^{\J^-}_0(\gamma_0,\gamma_-)\\
& \bigsqcup \coprod_{\substack{\gamma_+'\neq\gamma_+\\d_++d=1}} (-1)^{d_+}\M^{\J_1}_{d_+}(\gamma_+,\gamma_+')\times_{\overline{\gamma_+'}}K_{d}(\gamma_+',\gamma_-)\\
&\bigsqcup \coprod_{\substack{\gamma_-'\neq\gamma_-\\d+d_-=1}} (-1)^d K_d(\gamma_+,\gamma_-') \times_{\overline{\gamma_-'}} \M^{\J_3}_{d_-}(\gamma_-',\gamma_-).
\end{split}
\end{equation}
\item{(f)}
We also have analogues of \eqref{eqn:K2boundary} with point constraints as in \cite[Eqs.\ (2.35)-(2.37)]{td}.
\end{description}
\end{lemma}

Part (d) implies that if we restrict to Reeb orbits of $\lambda_+$ with action less than $L$, then the morphisms of Morse-Bott systems given by the moduli spaces $\Phi^{\J^+}$ and $\Phi^{\J^-}$ are composable in the sense of \cite[Def.\ 2.10]{td}. The rest of Lemma~\ref{lem:homotopy2} then implies that, again restricting to Reeb orbits of $\lambda_+$ with action less than $L$, the moduli spaces $K_d$ give a homotopy between the composition and the morphism given by the moduli spaces $\Phi^{\J^0}$. It then follows by applying Theorem~\ref{thm:td}(b)(iii,iv) and taking the direct limit as $L\to\infty$ that the Composition property in Proposition~\ref{prop:cobordismmapproperties} holds.
\end{proof}

\section{$S^1$-equivariant contact homology}
\label{equicurrents}

Continue to assume that $\lambda$ is a nondegenerate hypertight contact form on a closed manifold $Y$. We now define the $S^1$-equivariant contact homology $CH_*^{S^1}(Y,\lambda;\frak{J})$, where $\frak{J}$ is a generic $S^1$-equivariant $S^1\times ES^1$ family of $\lambda$-compatible almost complex structures, and we prove that the $S^1$-equivariant contact homology depends only on $Y$ and $\xi=\Ker(\lambda)$.  This construction closely parallels the definition of nonequivariant contact homology in \S\ref{section:plainmoduli} and \S\ref{cascades}, with minor modifications which we will explain.

\subsection{$S^1$-equivariant moduli spaces}
\label{sec:familymoduli}

We regard $ES^1=\lim_{N\to\infty} S^{2N+1}$. Let $\pi:ES^1\to BS^1={\mathbb C} P^\infty$ denote the projection.

Let
\[
\frak{J} = \left\{\frak{J}_{t,z} \mid t\in S^1,\; z\in ES^1\right\}
\]
be an $S^1\times ES^1$ family of $\lambda$-compatible almost complex structures on $\R\times Y$. We assume that $\frak{J}$ is smooth in the sense that its restriction to $S^1\times S^{2N+1}$ is smooth for each $N$. We further assume that $\frak{J}$ is $S^1$-equivariant in the sense that
\begin{equation}
\label{eqn:frakJequiv}
\frak{J}_{t,z} = \frak{J}_{t+\varphi,\varphi\cdot z}
\end{equation}
for each $t,\varphi\in S^1$ and $z\in ES^1$

Define a function $\widetilde{f}_N:S^{2N+1}\to\R$ by
\[
\widetilde{f}_N(z_0,\ldots,z_N) = \frac{1}{2}\sum_{i=1}^Ni|z_i|^2.
\]
Under the projection $\pi:S^{2N+1}\to {\mathbb C} P^N$, the function $\widetilde{f}_N$ descends to a Morse function $f_N:{\mathbb C} P^N\to\R$ with one critical point of each index $0,2,\ldots,2N$. If we write points in ${\mathbb C} P^N$ in the form $[z_0:z_1:\cdots:z_N]$, then the critical point of index $2k$ has $z_j=0$ for $j\neq k$.

Let $\widetilde{V}_N$ denote the gradient of $\widetilde{f}_N$ with respect to the standard metric on $S^{2N+1}$. Our convention is that a ``parametrized flow line'' of $\widetilde{V}_N$ is a map $\eta:\R\to S^{2N+1}$ such that
$\eta'(s)=\widetilde{V}_N(\eta(s))$ for all $s$. For such a parametrized flow line, if we write $\eta(0)=(z_0,\ldots,z_N)\in {\mathbb C}^{N+1}$, then we have
\begin{equation}
\label{eqn:efl}
\eta(s) = \frac{\left(z_0,e^sz_1,\ldots,s^{Ns}z_N\right)} {\sqrt{|z_0|^2+e^{2s}|z_1|^2+\cdots+e^{2Ns}|z_N|^2}}.
\end{equation}

Let $\widetilde{f}:ES^1\to\R$ denote the direct limit of the functions $\widetilde{f}_N$, and let $f:BS^1\to\R$ denote the direct limit of the functions $f_N$.  The vector field $\widetilde{V}_N$ pushes forward, under the inclusion $S^{2N+1}\to S^{2N+3}$, to the vector field $\widetilde{V}_{N+1}$; thus we can regard the vector fields $\widetilde{V}_N$ as defining a ``direct limit vector field'' $\widetilde{V}$ on $ES^1$. In particular, we use the terminology ``parametrized flow line of $\widetilde{V}$'' to refer to a parametrized flow line of $\widetilde{V}_N$ for some $N$.

\begin{definition}
\label{def:fammod}
Let $\gamma_+$ and $\gamma_-$ be Reeb orbits, and let $x_+$ and $x_-$ be critical points of $f$. Define $\widehat{\M}^{\frak{J}}((x_+,\gamma_+),(x_-,\gamma_-))$ to be the set of pairs $(\eta,u)$, where:
\begin{itemize}
	\item $\eta:\R\to ES^1$ is a parametrized flow line of $\widetilde{V}$ with $\lim_{s\to\pm\infty}\eta(s) \in\pi^{-1}(x_\pm)$.
	\item 
	$u:\R\times S^1\to \R\times Y$ satisfies the equation
	\begin{equation}
	\label{eqn:familycr}
	\partial_s u + \frak{J}_{t,\eta(s)}\partial_t u = 0.
	\end{equation}
	\item
	$\lim_{s\to\pm\infty}\pi_\R(u(s,\cdot))=\pm\infty$, and $\lim_{s\to\pm\infty}\pi_Y(u(s,\cdot))$ is a parametrization of the Reeb orbit $\gamma_\pm$.
\end{itemize}
\end{definition}

\begin{definition}
\label{def:fammodaction}
Observe that $\R$ acts on $\widehat{\M}^{\frak{J}}$ by translation of the parameter $s$ in $\eta$ and $u$ simultaneously. Moreover, it follows from \eqref{eqn:frakJequiv} that $S^1$ acts on $\widehat{\M}^{\frak{J}}$ by
\begin{equation}
\label{eqn:s1actm}
(\varphi\cdot(\eta,u))(s,t) = (\varphi\cdot\eta(s),u(s,t-\varphi)).
\end{equation}
Let $\widetilde{\M}^\frak{J}((x_+,\gamma_+),(x_-,\gamma_-))$ denote the quotient of $\widehat{\M}^\frak{J}((x_+,\gamma_+),(x_-,\gamma_-))$ by this $\R\times S^1$ action.
\end{definition}

Finally, if the pairs $(x_+,\gamma_+)$ and $(x_-,\gamma_-)$ are distinct, then $\R$ acts freely on $\widetilde{\M}^\frak{J}((x_+,\gamma_+),(x_-,\gamma_-))$ by composing $u$ with translations in the target $\R\times Y$, and we let $\M^\frak{J}((x_+,\gamma_+),(x_-,\gamma_-))$ denote the quotient by this $\R$ action.

If $x\in\op{Crit}(f)$ and $\gamma$ is a Reeb orbit, define
\[
\overline{(x,\gamma)} = \left(\pi^{-1}(x)\times\overline{\gamma}\right)/S^1,
\]
where $S^1$ acts on $\pi^{-1}(x)\times\overline{\gamma}$ as follows: If $\gamma$ has period $T$, meaning that $\gamma: \R/T\Z \to Y$, and if $t\in\R/T\Z$, then $\varphi \in S^1=\R/\Z$ acts by
\begin{equation}
\label{eqn:equivrel}
\varphi\cdot(x,\gamma(t)) = \left(\varphi\cdot x,\gamma(t-T\varphi)\right).
\end{equation}
We then have well-defined evaluation maps
\[
e_\pm: \M^\frak{J}((x_+,\gamma_+),(x_-,\gamma_-)) \longrightarrow \overline{(x_\pm,\gamma_\pm)}
\]
defined by
\begin{equation}
\label{eqn:fameval}
e_\pm(\eta,u) = \left(\lim_{s\to\pm\infty}\eta(s), \lim_{s\to\pm\infty}\pi_Y(u(s,0))\right).
\end{equation}

If $d$ is an integer, let $\M^\frak{J}_d((x_+,\gamma_+),(x_-,\gamma_-))$ denote the set of $(\eta, u)$ in the moduli space $\M^\frak{J}((x_+,\gamma_+),(x_-,\gamma_-))$ such that
\[
\CZ_\tau(\gamma_+)-\CZ_\tau(\gamma_-)+2c_1(u^*\xi,\tau)+\op{ind}(f,x_+) - \op{ind}(f,x_-) =d.
\]
Here $\op{ind}(f,x)$ denotes the Morse index of $f$ at the critical point $x$, which is a nonnegative even integer.

Analogously to Proposition~\ref{prop:transversality}, we have:

\begin{proposition}
If $\frak{J}$ is generic, then:
\begin{description}
\item{\emph{(a)}}
If $\gamma_+,\gamma_-$ are Reeb orbits and $x_+,x_-$ are critical points of $f$, such that the pairs $(x_+,\gamma_+)$ and $(x_-,\gamma_-)$ are distinct, and if $d$ is an integer, then the moduli space $\M_d^\frak{J}((x_+,\gamma_+),(x_-,\gamma_-))$ is cut out transversely and is a smooth manifold of dimension $d$, and the evaluation maps $e_+$ and $e_-$ on it are smooth.
\item{\emph{(b)}}
Each $k$-fold fiber product
\begin{gather*}
\M_{d_1}^\frak{J}((x_0,\gamma_0),(x_1,\gamma_1))\times_{\overline{(x_1,\gamma_1)}}\M_{d_2}^\frak{J}((x_1,\gamma_1),(x_2,\gamma_2))\\
 \times_{\overline{(x_2,\gamma_2)}} \cdots \times_{\overline{(x_{k-1},\gamma_{k-1})}}
 \\
 \M_{d_k}^\frak{J}((x_{k-1},\gamma_{k-1}),(x_k,\gamma_k))
\end{gather*}
is cut out transversely, and in particular is a smooth manifold of dimension $\sum_{i=1}^kd_i - k + 1$.
\end{description}
\end{proposition}

The precise meaning of transversality in part (a) is that each element of the moduli space is ``regular'' in the sense of Definition~\ref{def:famreg}.

Assume for the rest of this subsection that $\frak{J}$ is generic in the above sense. Analogously to Proposition~\ref{prop:orientations}, we have the following proposition, which is proved in \S\ref{sec:familysigns}:

\begin{proposition}
\label{prop:famor}
For each critical point $x$ of $f$ and each Reeb orbit $\gamma$, there is a canonical local system $\mc{O}_{(x,\gamma)}$ over $\overline{(x,\gamma)}$, locally isomorphic to $\Z$, such that:
\begin{description}
\item{\emph{(a)}}
The local system $\mc{O}_{(x,\gamma)}$ is trivial if and only if $\gamma$ is a good Reeb orbit.
\item{\emph{(b)}}
The moduli space $\M^\frak{J}((x_+,\gamma_+),(x_-,\gamma_-))$ has a canonical orientation with values in $e_+^*\mc{O}_{(x_+,\gamma_+)} \tensor e_-^*\mc{O}_{(x_-,\gamma_-)}$.
\end{description}
\end{proposition}

\begin{definition}
Let $\overline{\M}^\frak{J}_d((x_+,\gamma_+),(x_-,\gamma_-))$ denote the set of $k$-tuples $(u_1,\ldots,u_k)$ such that:
\begin{itemize}
\item There exist distinct pairs $(x_+,\gamma_+)=(x_0,\gamma_0),(x_1,\gamma_1),\ldots,(x_k,\gamma_k)=(x_-,\gamma_-)$ such that $u_i\in\M^\frak{J}_{d_i}((x_{i-1},\gamma_{i-1}),(x_i,\gamma_i))$ for $i=1,\ldots,k$.
\item
$e_-(u_i)=e_+(u_{i+1})$ for $i=1,\ldots,k-1$.
\item
$\sum_{i=1}^k d_i=d$.
\end{itemize}
Define the topology and evaluation maps on $\overline{M}^\frak{J}_d((x_+,\gamma_+),(x_-,\gamma_-))$ as in Definition~\ref{def:Mdbar}.
\end{definition}

Analogously to Proposition~\ref{prop:compactness}, we have:

\begin{proposition}
\label{prop:famcom}
If the pairs $(x_+,\gamma_+)$ and $(x_-,\gamma_-)$ are distinct and $d$ is an integer, then $\overline{\M}^\frak{J}_d((x_+,\gamma_+),(x_-,\gamma_-))$ is compact.
\end{proposition}

Define constrained moduli spaces $\M^\frak{J}_d((x_+,\gamma_+),p_+,(x_-,\gamma_-))$ for $p_+\in\overline{(x_+,\gamma_+)}$ and so on as in \S\ref{sec:constrained}. We now have the following analogue of Proposition~\ref{prop:currentcompact}:

\begin{proposition}
\label{prop:famcurcom}
Let $(x_+,\gamma_+)$ and $(x_-,\gamma_-)$ be distinct pairs of a critical point of $f$ and a Reeb orbit. Let $p_\pm\in\overline{(x_\pm,\gamma_\pm)}$. Assume that the pair $(p_+,p_-)$ is generic so that:
\begin{itemize}
\item
$p_+$ is a regular value of all evaluation maps
\[
e_+: \M_d^\frak{J}((x_+,\gamma_+),(x_0,\gamma_0)) \longrightarrow \overline{(x_+,\gamma_+)}
\]
for $d\le 2$.
\item
$p_-$ is a regular value of all evaluation maps
\[
e_-:\M_d^\frak{J}((x_0,\gamma_0),(x_-,\gamma_-) \longrightarrow \overline{(x_-,\gamma_-)}
\]
for $d\le 2$.
\item
$(p_+,p_-)$ is a regular value of all products of evaluation maps
\[
e_+\times e_-: \M_d^\frak{J}((x_+,\gamma_+),(x_-,\gamma_-)) \longrightarrow \overline{(x_+,\gamma_+)}\times\overline{(x_-,\gamma_-)}
\]
for $d\le 3$.
\end{itemize}
Then:
\begin{description}
\item{\emph{(a)}} $\M^\frak{J}_0((x_+,\gamma_+),(x_-,\gamma_-))$ is finite.

\item{\emph{(b)}} Analogously to \eqref{eqn:boundaryhatcheck}, $\overline{\M}^\frak{J}_1((x_+,\gamma_+),(x_-,\gamma_-))$ is a compact oriented topological one-manifold with oriented boundary
\[
\begin{split}
\partial\overline{\M}^\frak{J}_1((x_+,\gamma_+),(x_-,\gamma_-)) = \coprod_{\substack{(x_0,\gamma_0)\neq (x_+,\gamma_+), (x_-,\gamma_-)\\d_++d_-=1}} & (-1)^{d_+}\M^\frak{J}_{d_+}((x_+,\gamma_+),(x_0,\gamma_0))
\\
&\times_{\overline{(x_0,\gamma_0)}}\M^\frak{J}_{d_-}((x_0,\gamma_0),(x_-,\gamma_-)).
\end{split}
\]

\item{\emph{(c)}}
Likewise, analogues of \eqref{eqn:boundarycheckcheck} and \eqref{eqn:boundaryhathat} hold with $\overline{\M^\J_2}(\gamma_+,p_+,\gamma_-)$ replaced by $\overline{\M}^\frak{J}_2((x_+,\gamma_+),p_+,(x_-,\gamma_-))$ and so forth.

\item{\emph{(d)}}
Likewise, an analogue of \eqref{eqn:boundarycheckhat} holds with $\overline{\M^\J_3}(\gamma_+,p_+,\gamma_-,p_-)$ replaced by

$\overline{\M}^\frak{J}_3((x_+,\gamma_+), p_+, (x_-,\gamma_-), p_-))$ and so forth.

\end{description}
\end{proposition}

\begin{proof}
(a) This is analogous to the proof of Proposition~\ref{prop:currentcompact}(a).

(b)--(d). This follows from an analogue of Proposition~\ref{prop:gbs}.
\end{proof}

\subsection{Definition of $S^1$-equivariant contact homology}

Continue to assume that $\frak{J}$ is generic. By analogy with \S\ref{sec:mbs}, define a Morse-Bott system $(X,|\cdot|,S,\mc{O},M_*,e_\pm)$ as follows.

\begin{itemize}
	\item $X$ is the set of pairs $(x,\gamma)$ where $x$ is a critical point of $f$ on $BS^1$ and $\gamma$ is a Reeb orbit.
	\item If $(x,\gamma)\in X$, then $|(x,\gamma)|$ is the mod $2$ Conley-Zehnder index $\op{CZ}(\gamma)$; $S_{(x,\gamma)} = \overline{(x,\gamma)}$; and $\mc{O}_{(x,\gamma)}$ is the local system in Proposition~\ref{prop:famor}.
	\item If $(x_+,\gamma_+),(x_-,\gamma_-)\in X$ are distinct and $d\in\{0,1,2,3\}$, then
	\[
	M_d((x_+,\gamma_+),(x_-,\gamma_-)) = \M^\frak{J}_d((x_+,\gamma_+),(x_-,\gamma_-)).
	\]
	The evaluation maps $e_\pm$ on $M_d$ are defined by \eqref{eqn:fameval}, and the orientation on $M_d$ is given by Proposition~\ref{prop:famor}.
\end{itemize}

\begin{proposition}
\label{prop:fammbs}
If $\frak{J}$ is generic, then the above data constitute a Morse-Bott system.
\end{proposition}

\begin{proof}
This parallels the proof of Proposition~\ref{prop:mbs}. The one new ingredient is that in the proof of the Finiteness axiom, we need to know that if $\M^\frak{J}_d((x_+,\gamma_+)(x_-,\gamma_-))$ is nonempty, then $\op{ind}(f,x_+)\ge \op{ind}(f,x_-)$ and $\mc{A}(\gamma_+)\ge \mc{A}(\gamma_-)$. The inequality on Morse indices holds because the vector field $\widetilde{V}$ on $ES^1$ projects to a Morse-Smale vector field on $BS^1$.

To prove the action inequality, suppose that $(\eta,u)\in\M^\frak{J}((x_+,\gamma_+),(x_-,\gamma_-))$. For $s\in\R$, define a map $\gamma_s:S^1\to\R\times Y$ by $\gamma_s(t)=u(s,t)$. Then
\[
\lim_{s\to\pm\infty}\int_{S^1}\gamma_s^*\lambda = \mc{A}(\gamma_\pm),
\]
so it is enough to show that $\frac{d}{ds}\int_{S^1}\gamma_s^*\lambda\ge 0$. We compute that
\[
\begin{split}
\frac{d}{ds}\int_{S^1}\gamma_s^*\lambda &= \int_{S^1}d\lambda(\partial_su,\partial_tu)dt\\
&= \int_{S^1}d\lambda(-\frak{J}_{t,\eta(s)}\partial_tu,\partial_tu)dt.
\end{split}
\]
The integrand is pointwise nonnegative because the almost complex structure $\frak{J}_{t,\eta(s)}$ is $\lambda$-compatible.
\end{proof}

\begin{definition}
We define the {\em $S^1$-equivariant contact homology\/} $CH_*^{S^1}(Y,\lambda;\frak{J})$ to be the cascade homology $H_*^\ca$ (defined in \cite{td}) of the above Morse-Bott system.
\end{definition}


Concretely, $CH_*^{S^1}(Y,\lambda;\frak{J})$ is the homology of a chain complex $(CC_*^{S^1}(Y,\lambda),\partial^{S^1,\frak{J}})$ over $\Z$. The module $CC_*^{S^1}(Y,\lambda)$ has a ``check'' and a ``hat'' generator for each pair $(x,\gamma)$ where $x$ is a critical point of $f$ and $\gamma$ is a Reeb orbit. For convenience, we denote these generators by $\widecheck{\gamma}\tensor U^k$ and $\widehat{\gamma}\tensor U^k$ respectively, where $2k$ is the Morse index of $x$. Equivalently we can write
\[
CC_*^{S^1}(Y,\lambda) = NCC_*(Y,\lambda) \tensor \Z[U].
\]
The mod 2 gradings of the generators are given by
\[
\begin{split}
|\widecheck{\gamma}\tensor U^k| &= \op{CZ}(\gamma),\\
|\widehat{\gamma}\tensor U^k| &= \op{CZ}(\gamma) + 1.
\end{split}
\]
The differential $\partial^{S^1,\frak{J}}$ is defined by counting cascades just as in \S\ref{sec:cms} and \ref{sec:defnch}, except that now the cascades are defined using the moduli spaces $\M^\frak{J}_d((x_+,\gamma_+),(x_-,\gamma_-))$ instead of the moduli spaces $\M^\J_d(\gamma_+,\gamma_-)$. Here one needs to choose a base point $p_{(x,\gamma)}\in\overline{(x,\gamma)}$ for each pair $(x,\gamma)$ to define the cascades. One also needs to choose a generator of $\mc{O}_{(x,\gamma)}(p_{(x,\gamma)})$ for each pair $(x,\gamma)$ to fix the signs in the differential.

We now proceed to show that $CH_*^{S^1}$ is an invariant of the contact structure. (For the computation of $CH_*^{S^1}$ in the main example of interest, see Proposition~\ref{prop:chaut} below.)

\subsection{Cobordism maps}

\begin{proposition}
\label{prop:famcobmap}
Let $(X,\lambda)$ be an exact symplectic cobordism from $(Y_+,\lambda_+)$ to $(Y_-,\lambda_-)$, where the contact forms $\lambda_\pm$ are nondegenerate and hypertight. Assume further that every Reeb orbit for $\lambda_+$ is noncontractible in $X$. Let $\frak{J}_\pm$ be $S^1\times ES^1$-families of almost complex structures as needed to define the $S^1$-equivariant contact homology of $(Y_\pm,\lambda_\pm)$. Then $(X,\lambda)$ induces a well-defined map
\[
\Phi(X,\lambda;
\frak{J}_-,\frak{J}_+): CH_*^{S^1}(Y_+,\lambda_+;\frak{J}_+) \longrightarrow CH_*^{S^1}(Y_-,\lambda_-;\frak{J}_-).
\]
\end{proposition}

\begin{proof}
This parallels the proof of Proposition~\ref{prop:cobordismmap}.

Let $\frak{J}$ be an $S^1\times ES^1$-family of cobordism-compatible almost complex structures on $\overline{X}$ such that $\frak{J}_{t,z}$ agrees with $\frak{J}_{+,t,z}$ on $[0,\infty)\times Y$ and $\frak{J}_{t,z}$ agrees with $\frak{J}_{-,t,z}$ on $(-\infty,0]\times Y_-$. Assume that $\frak{J}$ is $S^1$-equivariant as in \eqref{eqn:frakJequiv}.

If $x_\pm$ are critical points of $f$ on $BS^1$ and if $\gamma_\pm$ are Reeb orbits for $\lambda_\pm$, let $\widehat{\Phi}^\frak{J}((x_+,\gamma_+),(x_-,\gamma_-))$ denote the set of pairs $(\eta,u)$ such that:
\begin{itemize}
	\item $\eta:\R\to ES^1$ is a flow line of $\widetilde{V}$ with $\lim_{s\to\pm\infty}\eta(t) \in\pi^{-1}(x_\pm)$.
	\item 
	$u:\R\times S^1\to \overline{X}$ satisfies the equation
	\[
	\partial_s u + \frak{J}_{t,\eta(s)}\partial_t u = 0.
	\]
	\item $u(s,t)\in[0,\infty)\times Y_+$ for $s>>0$ and $u(s,t)\in(-\infty,0]\times Y_-$ for $s<<0$.
	\item
	$\lim_{s\to\pm\infty}\pi_\R(u(s,\cdot))=\pm\infty$.
	\item
	$\lim_{s\to\pm\infty}\pi_{Y_\pm}(u(s,\cdot))$ is a parametrization of the Reeb orbit $\gamma_\pm$.
\end{itemize}
Observe that $\R$ acts on $\widehat{\Phi}^\frak{J}((x_+,\gamma_+),(x_-,\gamma_-))$ by translating the $s$ coordinate in the domain of both $\eta$ and $u$; and $S^1$ acts on $\widehat{\Phi}^\frak{J}$ by \eqref{eqn:s1actm}. Let $\Phi^\frak{J}((x_+,\gamma_+),(x_-,\gamma_-))$ denote the quotient by $\R\times S^1$. We have well-defined evaluation maps
\[
e_\pm: \Phi^\frak{J}((x_+,\gamma_+),(x_-,\gamma_-)) \longrightarrow \overline{(x_\pm,\gamma_\pm)}
\]
defined by \eqref{eqn:fameval}. If $d$ is an integer, let $\Phi_d^\frak{J}((x_+,\gamma_+),(x_-,\gamma_-))$ denote the set of $u\in\Phi^\frak{J}((x_+,\gamma_+),(x_-,\gamma_-))$ such that
\[
\op{CZ}_\tau(\gamma_+) - \op{CZ}_\tau(\gamma_-) + 2c_1(u^*T\overline{X},\tau) + \op{ind}(f,x_+) - \op{ind}(f,x_-) = d-1.
\]

Similarly to Lemmas~\ref{lem:morphism1} and \ref{lem:morphism2}, the moduli spaces $\Phi_d^\frak{J}((x_+,\gamma_+),(x_-,\gamma_-))$ constitute a morphism, in the sense of \cite[Def.\ 2.7]{td}, from the Morse-Bott system for $(Y_+,\lambda_+;\frak{J}_+)$ to the Morse-Bott system for $(Y_-,\lambda_-;\frak{J}_-)$. It then follows from Theorem~\ref{thm:td}(b)(i) that we have an induced map
\[
\Phi(X,\lambda;\frak{J}): CH_*^{S^1}(Y_+,\lambda_+;\frak{J}_+)\longrightarrow CH_*^{S^1}(Y_-,\lambda_-;\frak{J}_-).
\]
Similarly to Lemma~\ref{lem:homotopy}, this map does not depend on the choice of $S^1$-equivariant $S^1\times ES^1$-family of almost complex structures $\frak{J}$ extending $\frak{J}_+$ and $\frak{J}_-$, so we can denote it by $\Phi(X,\lambda;\frak{J}_+,\frak{J}_-)$.
\end{proof}

\subsection{Invariance of $S^1$-equivariant contact homology}

If $\frak{J}$ is an $S^1$-equivariant $S^1\times ES^1$-family of $\lambda$-compatible almost complex structures on $Y$, and if $c>0$, then there is a unique $S^1$-equivariant $S^1\times ES^1$-family of $c\lambda$-compatible almost complex structures $^c\frak{J}$ which agrees with $\frak{J}$ on $\xi$.
As in \S\ref{sec:nchinv}, we have canonical diffeomorphisms of moduli spaces
\[
\M^\frak{J}_d((x_+,\gamma_+),(x_-,\gamma_-))\simeq \M^{^c\frak{J}}((x_+,^c\gamma_+),(x_-,^c\gamma_-))
\]
which preserve the orientations and evaluation maps. As a result, we have a canonical isomorphism of chain complexes
\[
\left(CC_*^{S^1}(Y,\lambda),\partial^{S^1,\frak{J}}\right) = \left(CC_*^{S^1}(Y,c\lambda),\partial^{S^1,{^c\frak{J}}}\right).
\]
We denote the induced map on homology by
\begin{equation}
\label{eqn:famscaling}
s_c: CH_*^{S^1}(Y,\lambda;\frak{J}) \stackrel{\simeq}{\longrightarrow} CH_*^{S^1}(Y,c\lambda;{^c\frak{J}}).
\end{equation}

\begin{proposition}
\label{prop:famcobmapprops}
The cobordism maps in Proposition~\ref{prop:famcobmap} have the following properties:
\begin{description}
\item{(a) (Scaling)}
Suppose $(Y,\lambda_0)$ is nondegenerate and hypertight. Let $\frak{J}$ be an $S^1$-equivariant $S^1\times ES^1$-family of $\lambda_0$-compatible almost complex structures as needed to define $CH_*^{S^1}(Y,\lambda_0;\frak{J})$. Then for the trivial cobordism \eqref{eqn:trivcob}, the cobordism map
\[
\Phi(X,\lambda;{^{e^a}}\frak{J},{^{e^b}}\frak{J}): NCH_*(Y,e^b\lambda_0;{^{e^b}}\frak{J}) \longrightarrow NCH_*(Y,e^a\lambda_0;{^{e^a}}\frak{J})
\]
agrees with the scaling isomorphism $s_{e^{a-b}}$ in \eqref{eqn:famscaling}.
\item{(b) (Composition)}
Let $(Y_i,\lambda_i)$ be nondegenerate and hypertight, and let $\frak{J}_i$ be an $S^1$-equivariant $S^1\times ES^1$-family of $\lambda_i$-compatible almost complex structures as needed to define $CH_*^{S^1}(Y,\lambda_i;\frak{J}_i)$, for $i=1,2,3$. Let $(X_+,\lambda_+)$ be an exact symplectic cobordism from $(Y_1,\lambda_1)$ to $(Y_2,\lambda_2)$, and let $(X_-,\lambda_-)$ be an exact symplectic cobordism from $(Y_2,\lambda_2)$ to $(Y_3,\lambda_3)$. Assume further that every Reeb orbit for $\lambda_1$ is noncontractible in $X_-\circ X_+$, and every Reeb orbit for $\lambda_2$ is noncontractible in $X_-$. Then
\[
\Phi(X_-\circ X_+,\lambda;\frak{J}_3,\frak{J}_1) = \Phi(X_-,\lambda_-;\frak{J}_3,\frak{J}_2) \circ \Phi(X_+,\lambda_+;\frak{J}_2,\frak{J}_1).
\]
\end{description}
\end{proposition}

\begin{proof}
This follows the proof of Proposition~\ref{prop:cobordismmapproperties}, modified as in the proof of Proposition~\ref{prop:famcobmap}.
\end{proof}

We can now deduce that $S^1$-equivariant contact homology depends only on $Y$ and $\xi$.

\begin{proof}[Proof of Theorem~\ref{thm:CHS1invariant}.]
This follows from Proposition~\ref{prop:famcobmapprops} in the same way that Theorem~\ref{thm:NCHinvariant} is deduced from Proposition~\ref{prop:cobordismmapproperties} in \S\ref{sec:nchinv}.
\end{proof}

\section{Computations in the autonomous case}
\label{equicasc}

Continue to assume that $\lambda$ is a nondegenerate hypertight contact form on a closed manifold $Y$. In this section we study the nonequivariant and $S^1$-equivariant contact homology in the special case when the $S^1$-family of almost complex structures $\J$ or the $S^1\times ES^1$-family of almost complex structures $\frak{J}$ is constant, given by a single almost complex structure $J$ on $\R\times Y$. Here we need to assume that $J$ satisfies suitable transversality conditions, namely that $J$ is ``admissible'' in the sense of Definition~\ref{def:admissible} below, which also implies that cylindrical contact homology is defined.  We use these calculations to prove Theorem~\ref{thm:cchinv}, asserting that if $J$ is admissible, then cylindrical contact homology is canonically isomorphic to $S^1$-equivariant contact homology tensor $\Q$. Finally, we show that admissibility holds for generic $J$ when $\dim(Y)=3$.

\subsection{Nonequivariant contact homology in the autonomous case}
\label{sec:nchaut}

Let $J$ be a $\lambda$-compatible almost complex structure on $\R\times Y$. We now study nonequivariant contact homology for the constant $S^1$-family of almost complex structures $\J=\{J_t\}_{t\in S^1}$ where $J_t\equiv J$. Note that in this case, if $\alpha$ and $\beta$ are distinct Reeb orbits and $d$ is an integer, then $S^1$ acts on $\M^\J_d(\alpha,\beta)$ by precomposing maps $u:\R\times S^1\to\R\times Y$ with rotations of the $S^1$ factor, and we have\footnote{See sections \S\ref{sec:contactprelim} and \S\ref{sec:trans} for the notation in equation \eqref{eqn:mjds1}.}
\begin{equation}
\label{eqn:mjds1}
\M^J_d(\alpha,\beta) = \M^\J_d(\alpha,\beta)/S^1.
\end{equation}

\begin{definition}
\label{def:admissible}
Let $J$ be a $\lambda$-compatible almost complex structure on $\R\times Y$. Let $\J$ be the constant $S^1$-family $\{J_t\}_{t\in S^1}$ where $J_t\equiv J$. Let $\mc{P}$ be a choice of base point $p_\alpha\in\overline{\alpha}$ for each Reeb orbit $\alpha$. We say that the pair $(J,\mc{P})$ is {\em admissible\/} if the following hold for every pair of distinct Reeb orbits $\alpha,\beta$:
\begin{description}
	\item{(a)} If $d\le 0$, then $\M^\J_d(\alpha,\beta) = \emptyset$.
	\item{(b)} $\M^\J_1(\alpha,\beta)$ is cut out transversely, i.e.\ each cylinder in this moduli space is regular in the sense of Definition~\ref{def:regular}.
	\item{(c)} $\M^\J_1(\alpha,p_\alpha,\beta,p_\beta)=\emptyset$.
	\item{(d)} $\M^\J_2(\alpha,p_\alpha,\beta,p_\beta)$ is cut out transversely. That is, for each $u\in \M^\J_2(\alpha,p_\alpha,\beta,p_\beta)$, the moduli space $\M^\J_2(\alpha,\beta)$ is cut out transversely in a neighborhood of $u$, and $(p_\alpha,p_\beta)$ is a regular value of $e_+\times e_-$ on this neighborhood.
	\item{(e)} If $\gamma$ is a Reeb orbit distinct from $\alpha$ and $\beta$, then
\[	\M^\J_1(\alpha,p_\alpha,\gamma)\times_{\overline{\gamma}}\M^\J_1(\gamma,\beta,p_\beta) = \emptyset.
\] 
\end{description}
We say that $J$ is {\em admissible\/} if there exists $\mc{P}$ such that $(J,\mc{P})$ is admissible.
\end{definition}

Recall from \S\ref{sec:defnch} that $NCC_*(Y,\lambda)$ denotes the free $\Z$-module with two generators $\widecheck{\gamma}$ and $\widehat{\gamma}$ for each Reeb orbit $\gamma$. Proposition~\ref{prop:nchaut}(a) below asserts that an admissible pair $(J,\mc{P})$ determines a well-defined cascade differential on $NCC_*(Y,\lambda)$. By this we mean that if $\alpha$ is a Reeb orbit, then there are only finitely many cascades that contribute to the cascade differential of  $\widecheck{\alpha}$ or $\widehat{\alpha}$ as defined in \S\ref{sec:defnch} using $\J=\{J\}$ and $\mc{P}$; and each of these cascades, regarded as an element of a (product of) moduli space(s) of holomorphic cylinders, is cut out transversely. We denote this cascade differential by $\partial^J_\ca$. With respect to the decomposition into check and hat generators, we can write $\partial^J_\ca$ in block matrix form as
\begin{equation}
\label{eqn:blockmatrix}
\partial^J_\ca = \begin{pmatrix} \widecheck{\partial} & \partial_+ \\ \partial_- & \widehat{\partial}\end{pmatrix}
\end{equation}
where each entry in the block matrix sends the free $\Z$-module generated by the set of Reeb orbits to itself; $\langle\partial^J_\ca\widecheck{\alpha},\widehat{\beta}\rangle = \langle\partial_-\alpha,\beta\rangle$, and so forth.

\begin{proposition}
\label{prop:nchaut}
Suppose the pair $(J,\mc{P})$ is admissible. Fix a generator of $\mc{O}_\gamma(p_\gamma)$ for each Reeb orbit $\gamma$. Then:
\begin{description}
	\item{(a)} For the constant family $\J=\{J\}$, the moduli spaces $\M^\J_d(\alpha,\beta)$, with or without point constraints at $p_\alpha$ and/or $p_\beta$, determine a well-defined cascade differential $\partial^J_\ca$ on $NCC_*(Y,\lambda)$ that counts elements of cascade moduli spaces $\M^\ca_0$ as in Definition~\ref{def:cascadediff}.
	\item{(b)} $\left(\partial^J_\ca\right)^2=0$.
	\item{(c)} The homology of the chain complex $\left(NCC_*(Y,\lambda),\partial^J_\ca\right)$ is canonically isomorphic to the nonequivariant contact homology $NCH_*(Y,\xi)$.
	\item{(d)} In the block matrix \eqref{eqn:blockmatrix}, we have:
	\begin{description}
		\item{(i)}
		\begin{equation}
		\label{eqn:partial+}
		\partial_+\alpha = \left\{\begin{array}{cl} 0, & \mbox{if $\alpha$ is good},\\ -2\alpha, & \mbox{if $\alpha$ is bad}.\end{array}\right.
		\end{equation}
		\item{(ii)} If $\alpha$ and $\beta$ are distinct Reeb orbits with $\beta$ good, then\footnote{See \S\ref{sec:cchintro} for the notation $\delta$ and $\kappa$.}
		\begin{equation}
		\label{eqn:partialcheck}
		\left\langle \widecheck{\partial}\alpha,\beta \right\rangle = \left\{\begin{array}{cl} \langle\delta\kappa\alpha,\beta\rangle, & \mbox{$\alpha$ good},\\ 0, & \mbox{$\alpha$ bad}.\end{array}\right.
		\end{equation}
		\item{(iii)} If $\alpha$ and $\beta$ are distinct Reeb orbits with $\alpha$ good, then
		\begin{equation}
		\label{eqn:partialhat}
		\left\langle\widehat{\partial}\alpha,\beta\right\rangle = \left\{\begin{array}{cl} \langle -\kappa\delta\alpha,\beta\rangle, & \mbox{$\beta$ good},\\ 0, & \mbox{$\beta$ bad}.\end{array}\right.
		\end{equation}
	\end{description}
\end{description}
\end{proposition}

\begin{proof}
(a) Note that we cannot apply Proposition~\ref{prop:cascadekey}(a) directly, because we are not assuming all the transversality conditions (for example regarding three-dimensional moduli spaces) for this proposition to be applicable. Instead we argue more explicitly.

By part (a) of Definition~\ref{def:admissible}, if $(u_1,\ldots,u_k)$ is a cascade, then each $u_i$ is in a moduli space $\M_d^\J$ with $d\ge 1$. It follows that the cascade moduli spaces $\M^\ca_0\left(\widetilde{\alpha},\widetilde{\beta}\right)$ for $\alpha\neq\beta$ are described simply as follows:
\begin{align}
\nonumber
\M^\ca_0\left(\widehat{\alpha},\widecheck{\beta}\right) =& \emptyset,\\
\label{eqn:casc2}
\M^\ca_0\left(\widecheck{\alpha},\widecheck{\beta}\right) =& \M^\J_1\left(\alpha,p_{\alpha},\beta\right),\\
\label{eqn:casc3}
\M^\ca_0\left(\widehat{\alpha},\widehat{\beta}\right) =& \M^\J_1\left(\alpha,\beta,p_\beta\right),\\
\M^\ca_0\left(\widecheck{\alpha},\widehat{\beta}\right) =& \label{eqn:casc4}
\M^\J_2\left(\alpha,p_\alpha,\beta,p_\beta\right)\\
\nonumber
& \bigsqcup \coprod_{\gamma\neq\alpha,\beta} \M^\J_1\left(\alpha,p_\alpha,\gamma\right) \underset{\overline{\gamma}}{\circlearrowleft} \M^\J_1\left(\gamma,\beta,p_\beta\right).
\end{align}
The notation in the last line indicates the set of pairs
\[
(u_1,u_2)\in\M^\J_1(\alpha,p_\alpha,\gamma)\times\M^\J_1(\gamma,\beta,p_\beta)
\]
such that the points $p_\gamma$, $e_-(u_1)$, and $e_+(u_2)$ are distinct and positively cyclically ordered with respect to the orientation of $\overline{\gamma}$ given by the Reeb vector field.

We claim next that all of the cascade moduli spaces $\M^\ca_0$ are cut out transversely. To see this, recall that $\M^\J_1(\alpha,\beta)$ is cut out transversely by part (b) of Definition~\ref{def:admissible}. Then $\M^\J_1(\alpha,p_\alpha,\beta)$ and $\M^\J_1(\alpha,\beta,p_\beta)$ are also cut out transversely, because the evaluation maps $e_\pm$ on $\M^\J_1(\alpha,\beta)$ are submersions as in \eqref{eqn:mjds1}. It follows from this and part (d) of Definition~\ref{def:admissible} that the remaining cascade moduli space $\M^\ca_0\left(\widecheck{\alpha},\widehat{\beta}\right)$ is cut out transversely.

To complete the proof of assertion (a), we need to show that each cascade moduli space $\M^\ca_0$ is finite.

To show that the cascade moduli spaces \eqref{eqn:casc2} and \eqref{eqn:casc3} are finite, we first note that by Proposition~\ref{prop:compactness} and part (a) of Definition~\ref{def:admissible}, the moduli spaces $\M^\J_1(\alpha,\beta)$ are compact. Finiteness of \eqref{eqn:casc2} and \eqref{eqn:casc3} then follows from the above transversality. This finiteness also implies finiteness of the second term on the right hand side of \eqref{eqn:casc4}.

Finiteness of the first term on the right hand side of \eqref{eqn:casc4} follows similarly, with the help of part (e) of Definition~\ref{def:admissible}.

(b)
As in part (a), we cannot apply Proposition~\ref{prop:cascadekey}(b) directly. However we can perturb the constant $S^1$-family $\J=\{J\}$ to a nonconstant $S^1$-family $\J'$ which is generic so that Proposition~\ref{prop:cascadekey}(b) applies to show that $\left(\partial_\ca^{\J'}\right)^2=0$. To deduce $\left(\partial_\ca^J\right)^2=0$ from this, it is enough to show that for every real number $L$, if the perturbation is sufficiently small with respect to $L$, then $\partial_\ca^{\J'}$ agrees with $\partial_\ca^J$ when applied to generators $\widecheck{\alpha}$ or $\widehat{\alpha}$ for which the Reeb orbit $\alpha$ has action $\mc{A}(\alpha)<L$.

To prove the above claim, suppose to get a contradiction that there exist a real number $L$, and a sequence of generic $S^1$-families $\{\J^k\}_{k=1,\cdots}$ converging to the constant $S^1$-family $\J=\{J\}$, such that for each $k$, the cascade differential $\partial_\ca^{\J^k}$ disagrees with $\partial_\ca^J$ on some generator $\widetilde{\alpha}$ (equal to $\widecheck{\alpha}$ or $\widehat{\alpha}$) with $\mc{A}(\alpha)<L$. Since there are only finitely many Reeb orbits with action less than $L$, by passing to a subsequence we may assume that there are fixed generators $\widetilde{\alpha}$ and $\widetilde{\beta}$ with $\mc{A}(\alpha),\mc{A}(\beta)<L$ such that
\begin{equation}
\label{eqn:casccon}
\left\langle \partial_\ca^{\J^k}\widetilde{\alpha},\widetilde{\beta}\right\rangle \neq \left\langle\partial_\ca^J\widetilde{\alpha},\widetilde{\beta}\right\rangle
\end{equation}
for all $k$.

Since each cascade in $\M^\ca_0\left(\widetilde{\alpha},\widetilde{\beta}\right)$ for $\J=\{J\}$ is cut out transversely, the implicit function theorem
gives an injective, orientation-preserving map from the set of such cascades for $\J=\{J\}$ to the set of such cascades for $\J^k$ when $k$ is sufficiently large.
We claim that this map is also surjective for $k$ sufficiently large, which will then give a contradiction to \eqref{eqn:casccon}.

Suppose to get a contradiction that this surjectivity does not hold for all sufficiently large $k$. Then after passing to a subsequence, for each $k$ there is a cascade which is counted by $\left\langle \partial_\ca^{\J^k}\widetilde{\alpha},\widetilde{\beta}\right\rangle$ but which is not a perturbation (coming from the implicit function theorem as above) of a cascade counted by $\left\langle\partial_\ca^J\widetilde{\alpha},\widetilde{\beta}\right\rangle$. We claim that we can pass to a subsequence so that these cascades for each $k$ converge to a cascade counted by $\left\langle\partial_\ca^J\widetilde{\alpha},\widetilde{\beta}\right\rangle$, which will give the desired contradiction.

We will just explain the trickiest case of this, which is when $\widetilde{\alpha}=\widecheck{\alpha}$ and $\widetilde{\beta}=\widehat{\beta}$. Then the cascade counted by $\left\langle \partial_\ca^{\J^k}\widetilde{\alpha},\widetilde{\beta}\right\rangle$ has the form
$(u^k_1, \ldots,u^k_{m_k})$, where there are distinct Reeb orbits $\alpha=\gamma_{k,0},\gamma_{k,1},\ldots,\gamma_{k,m_k}=\beta$, and integers $d_{k,i}$ for $i=1,\ldots,m_k$ with $\sum_{i=1}^{m_k}d_{k,i}=2$, such that $u^k_i\in\M^{\J^k}_{d_{k,i}}(\gamma_{k,i-1},\gamma_{k,i})$ for $i=1,\ldots,m_k$. By passing to a subsequence we may assume that $m_k=m$ does not depend on $k$, and the Reeb orbit $\gamma_{k,i}=\gamma_i$ and the integer $d_{k,i}=d_i$ do not depend on $k$ either.

As in Proposition~\ref{prop:compactness}, we may pass to a further subsequence such that for each $i=1,\ldots,m$, the sequence $\{u^k_i\}_{k=1,\ldots}$ converges to an element of $\overline{\M}^\J_{d_i}(\gamma_{i-1},\gamma_i)$. Since $\sum_{i=1}^md_i=2$, by part (a) of Definition~\ref{def:admissible}, we must have either $m=1$ and $d_1=2$, or $m=2$ and $d_1=d_2=1$.

If $m=1$ and $d_1=2$, then by parts (a) and (e) of Definition~\ref{def:admissible}, the limit of the sequence $\{u^k_1\}_{k=1,\ldots}$ is an element of $\M^\J_2(\alpha,p_\alpha,\beta,p_\beta)$, and thus a cascade in the first term on the right hand side of \eqref{eqn:casc4}.

If $m=2$ and $d_1=d_2=1$, then the sequence $\{u^k_1\}_{k=1,\ldots}$ converges to an element $u^\infty_1\in {\M}^\J_1(\alpha,\gamma)$, and the sequence $\{u^k_2\}_{k=1,\ldots}$ converges to an element $u^\infty_2\in {\M}^\J_1(\gamma,\beta)$, where we are writing $\gamma=\gamma_1$.   By parts (c) and (e) of Definition~\ref{def:admissible}, the three points $p_\gamma$, $e_-(u^\infty_1)$, and $e_+(u^\infty_2)$ in $\overline{\gamma}$ are distinct. Since the three points $p_\gamma$, $e_-(u^k_1)$, and $e_+(u^k_2)$ are positively cyclically ordered for each $k$, it follows that the three points $p_\gamma$, $e_-(u^\infty_1)$, and $e_+(u^\infty_2)$ are also positively cyclically ordered. Thus the pair $(u^\infty_1,u^\infty_2)$ is a cascade in the second term on the right hand side of \eqref{eqn:casc4}. 

(c)
As shown in the proof of part (b), for any real number $L$ we can choose a generic $S^1$-family $\J'$ close to $\J=\{J\}$ such that $\partial^{\J'}_\ca = \partial^J_\ca$ on all generators $\widetilde{\alpha}$ with $\mc{A}(\alpha)<L$. It follows that there is a canonical isomorphism
\begin{equation}
\label{eqn:nchl}
NCH_*^L(Y,\lambda) = H_*\left(NCC_*^L(Y,\lambda),\partial^J_\ca\right).
\end{equation}
Here $NCC_*^L(Y,\lambda)$ denotes the free $\Z$-module generated by $\widecheck{\gamma}$ and $\widehat{\gamma}$ for Reeb orbits $\gamma$ with action $\mc{A}(\gamma)<L$. On the left hand side, $NCH_*^L(Y,\lambda)$ denotes the ``filtered nonequivariant contact homology'', which is the homology of $\left(NCC_*^L(Y,\lambda),\partial^{\J'}_\ca\right)$ for any generic $\J'$. The proof of Theorem~\ref{thm:NCHinvariant} shows that this depends only on $(Y,\lambda)$; see \cite[Thm.\ 1.3]{cc2} for a similar argument. Moreover, if $L<L'$, then inclusion of chain complexes induces a well-defined map $NCH_*^L(Y,\lambda)\to NCH_*^{L'}(Y,\lambda)$, and the canonical isomorphisms \eqref{eqn:nchl} fit into a commutative diagram
\[
\begin{CD}
NCH_*^L(Y,\lambda) @>{=}>> H_*\left(NCC_*^L(Y,\lambda),\partial^J_\ca\right)\\
@VVV  @VVV\\
NCH_*^{L'}(Y,\lambda) @>{=}>> H_*\left(NCC_*^{L'}(Y,\lambda),\partial^J_\ca\right).
\end{CD}
\]
Taking the direct limit over $L$ in \eqref{eqn:nchl} then proves (c).

(d)
(i) We have seen that $\M^\ca_0(\widehat{\alpha},\widecheck{\beta})=\emptyset$ when $\alpha\neq\beta$, so equation \eqref{eqn:partial+} follows from Definition~\ref{def:cascaa}.

(ii)
Let $\alpha$ and $\beta$ be distinct Reeb orbits, and assume that $\beta$ is good. Recall from the proof of part (a) that
\begin{equation}
\label{eqn:rfppa}
\M^\ca_0\left(\widecheck{\alpha},\widecheck{\beta}\right) = \M^\J_1(\alpha,p_\alpha,\beta).
\end{equation}
The coefficient
\[
\left\langle\widecheck{\partial}\alpha,\beta\right\rangle = \left\langle \partial_\ca^J\widecheck{\alpha},\widecheck{\beta}\right\rangle
\]
that we need to understand is a signed count of points in the moduli space \eqref{eqn:rfppa}.

To review how the signs work, recall from \S\ref{sec:defnch} that we are fixing generators of $\mc{O}_\alpha(p_\alpha)$ and $\mc{O}_\beta(p_\beta)$. Since $\beta$ is good, the local system $\mc{O}_\beta$ is trivial by Proposition~\ref{prop:orientations}(a), and our choice of generator of $\mc{O}_\beta(p_\beta)$ trivializes it. By Proposition~\ref{prop:orientations}(b), the moduli space $\M^\J_1(\alpha,\beta)$ has a canonical orientation with values in $e_+^*\mc{O}_\alpha\tensor e_-^*\mc{O}_\beta$. Using the above trivialization of $\mc{O}_\beta$, we obtain an orientation of $\M^\J_1(\alpha,\beta)$ with values in $e_+^*\mc{O}_\alpha$.  Using the chosen generator of $\mc{O}_\alpha(p_\alpha)$, each point in the right hand side of \eqref{eqn:rfppa} then has a sign associated to it by the convention in \cite[Convention 2.2]{td}. This convention says here that the sign is positive if and only if the derivative of $e_+:\M^\J_1(\alpha,\beta)\to\overline{\alpha}$ is orientation preserving with respect to the orientation of $\overline{\alpha}$ determined by the Reeb vector field. By the sign convention in \cite[\S3.2]{td}, the corresponding point on the left hand side of \eqref{eqn:rfppa} is counted with the same sign.

By \eqref{eqn:mjds1}, the moduli space $\M^\J_1(\alpha,\beta)$ consists of a circle for each $u\in \M^J_1(\alpha,\beta)$. Fix such a $u$ and let $M_u$ denote the corresponding circle. The map $e_+:M_u\to \overline{\alpha}$ is a submersion of degree $d(\alpha)/d(u)$ (see \S\ref{sec:cchintro} for this notation). Thus the subset of the moduli space \eqref{eqn:rfppa} coming from $M_u$ consists of $d(\alpha)/d(u)$ points.

If $\alpha$ is bad, then the local system $\mc{O}_\alpha$ is nontrivial, which means that when one travels $d(u)/d(\alpha)$ rotations around the circle $M_u$, the orientation of $M_u$ with values in $e_+^*\mc{O}_\alpha$, because it is continuous, must switch sign. Thus consecutive points in the subset of \eqref{eqn:rfppa} coming from $M_u$ have opposite signs. It follows that the signed count of these points is zero. Since the same is true for every $u\in\M^J_1(\alpha,\beta)$, we conclude that $\left\langle\partial_\ca^J\widecheck{\alpha},\widecheck{\beta}\right\rangle = 0$.

If $\alpha$ is good, then the local system $\mc{O}_\alpha$ is trivial, so all of the points in \eqref{eqn:rfppa} coming from $M_u$ count with the same sign. Moreover, by the convention in Definition~\ref{def:EGHsigns}, this sign agrees with the sign $\epsilon(u)$ in \eqref{eqn:defdelta}. We conclude that
\[
\left\langle\partial_\ca^J\widecheck{\alpha}, \widecheck{\beta}\right\rangle = \sum_{u\in\M^J_1(\alpha,\beta)}\frac{\epsilon(u)d(\alpha)}{d(u)} = \langle\delta\kappa\alpha,\beta\rangle.
\]

(iii) This is proved by a symmetric argument to (ii). Note that there is an extra minus sign in \cite[Convention 2.2]{td} in this case, which is why there is a minus sign in \eqref{eqn:partialhat} which is not present in \eqref{eqn:partialcheck}.
\end{proof}

\begin{corollary}
\label{cor:degh20}
If the pair $(J,\mc{P})$ is admissible, and if $\partial^{EGH}$ denotes the cylindrical contact homology differential determined by $J$ as in \S\ref{sec:cchintro}, then $(\partial^{EGH})^2=0$. 
\end{corollary}

\begin{proof}
By Proposition~\ref{prop:nchaut}(a), we can use the pair $(J,\mc{P})$ to define a cascade differential $\partial^J_\ca$. By Proposition~\ref{prop:nchaut}(d), the part of $\partial^J_\ca$ going from good Reeb orbits to good Reeb orbits, again written as a block matrix with respect to check and hat generators, has the form
\[
\left(\partial^J_\ca\right)^{\op{good}}_{\op{good}} = \begin{pmatrix} \delta\kappa & 0 \\ \mbox{\Biohazard} & -\kappa\delta \end{pmatrix}.
\]
Here the biohazard symbol {\Biohazard} indicates an unknown matrix. Likewise, the parts of $\partial^J_\ca$ going from bad to good orbits and from good to bad orbits have the form
\[
\left(\partial^J_\ca\right)^{\op{bad}}_{\op{good}} = \begin{pmatrix} 0 & 0 \\ \mbox{\Biohazard} & \mbox{\Biohazard} \end{pmatrix},
\quad\quad\quad
\left(\partial^J_\ca\right)^{\op{good}}_{\op{bad}} = \begin{pmatrix} \mbox{\Biohazard} & 0 \\ \mbox{\Biohazard} & 0 \end{pmatrix}.
\]
It follows that the part of $\left(\partial^J_\ca\right)^2$ going from good to good orbits has the form
\[
\left(\left(\partial^J_\ca\right)^2\right)^{\op{good}}_{\op{good}} = \begin{pmatrix} (\delta\kappa)^2 & 0 \\ \mbox{\Biohazard} & (\kappa\delta)^2 \end{pmatrix}.
\]
It then follows from Proposition~\ref{prop:nchaut}(b) that $(\delta\kappa)^2=0$. Since $\partial^{\op{EGH}}=\delta\kappa$ (after tensoring with $\Q$), the result follows.
\end{proof}

\begin{remark}
\label{rem:jowantsthis}
The above calculations show that if $(J,\mc{P})$ is admissible, then the part of $\widecheck{\partial}$ mapping between good Reeb orbits is a differential, which after tensoring with $\Q$ agrees with $\partial^{\op{EGH}}$. By contrast, if $\J$ is a generic $S^1$-family of $\lambda$-compatible almost complex structures, if $\mc{P}$ is a generic choice of base points, and if we write the associated cascade differential in check-hat block form as
\[
\partial^\J_\ca = \begin{pmatrix} \widecheck{\partial} & \partial_+\\ \partial_- & \widehat{\partial}\end{pmatrix},
\]
then $\widecheck{\partial}$ does {\em not\/} always give a differential. We know from $\left(\partial^\J_\ca\right)^2=0$ that
\[
\left(\widecheck{\partial}\right)^2 + \partial_+\partial_- = 0,
\]
but in general $\partial_+\partial_-$ can be nonzero, even between good Reeb orbits.
\end{remark}

\subsection{$S^1$-equivariant contact homology in the autonomous case}
\label{sec:chaut}

Let $J$ be a $\lambda$-compatible almost complex structure on $\R\times Y$, let $\mc{P}$ be a choice of base points, and assume that the pair $(J,\mc{P})$ is admissible. Use these to define a differential $\partial^J_\ca$ on $NCC_*(Y,\lambda)$ by Proposition~\ref{prop:nchaut}(a).

\begin{definition}
Define a ``BV operator''
\[
\partial_1:NCC_*(Y,\lambda) \longrightarrow NCC_*(Y,\lambda)
\]
by equation \eqref{eqn:defbv}.
Define a differential $\partial^{S^1,J}$ on
\[
CC_*^{S^1}(Y,\lambda) = NCC_*(Y,\lambda)\tensor\Z[U]
\]
by
\begin{equation}
\label{eqn:ds1j}
\partial^{S^1,J} = \partial^J_\ca\tensor 1 + \partial_1\tensor U^{-1}.
\end{equation}
Here $U^{-1}$ denotes the operator sending $U^k\mapsto U^{k-1}$ when $k>0$ and sending $1\mapsto 0$.
\end{definition}

\begin{lemma}
Suppose that the pair $(J,\mc{P})$ is admissible. Then $\left(\partial^{S^1,J}\right)^2=0$.
\end{lemma}

\begin{proof}
This can be shown indirectly using the proof of Proposition~\ref{prop:chaut} below, similarly to the proof of Proposition~\ref{prop:nchaut}(b), but we will give a direct proof here for clarity.

By \eqref{eqn:ds1j}, we have
\[
\left(\partial^{S^1,J}\right)^2 = \left(\partial^J_\ca\right)^2\tensor 1 + \left(\partial^J_\ca\partial_1 + \partial_1\partial^J_\ca\right)\tensor U^{-1} + (\partial_1)^2\tensor U^{-2}.
\]
We know from Proposition~\ref{prop:nchaut} that $\left(\partial^J_\ca\right)^2=0$, and it follows immediately from \eqref{eqn:defbv} that $(\partial_1)^2=0$. So we just need to check that
\[
\partial^J_\ca\partial_1 + \partial_1\partial^J_\ca = 0.
\]
As a block matrix with respect to check and hat generators, we have
\[
\partial_1 = \begin{pmatrix} 0 & 0 \\ \kappa & 0 \end{pmatrix},
\]
where we define $\kappa$ of a bad Reeb orbit to be zero. By \eqref{eqn:blockmatrix}, we then have
\[
\partial^J_\ca\partial_1 + \partial_1\partial^J_\ca = \begin{pmatrix} \partial_+\kappa & 0 \\ \kappa\widecheck{\partial} + \widehat{\partial}\kappa & \kappa\partial_+\end{pmatrix}.
\]
It follows from \eqref{eqn:partial+} that $\partial_+\kappa = \kappa\partial_+ = 0$. 
It follows from \eqref{eqn:partialcheck} that $\kappa\widecheck{\partial} = \kappa\delta\kappa$, and it follows from \eqref{eqn:partialhat} that $\widehat{\partial}\kappa = -\kappa\delta\kappa$. Hence the above matrix is $0$.
\end{proof}

\begin{proposition}
\label{prop:chaut}
Suppose that the pair $(J,\mc{P})$ is admissible. Then the homology of the chain complex $\left(CC_*^{S^1}(Y,\lambda),\partial^{S^1,J}\right)$ is canonically isomorphic to the equivariant contact homology $CH_*^{S^1}(Y,\xi)$.
\end{proposition}

\begin{proof} We proceed in five steps.

{\em Step 1.\/} Let $\frak{J}$ be the constant $S^1\times ES^1$-family of $\lambda$-compatible almost complex structures on $\R\times Y$ given by $\frak{J}_{t,z}\equiv J$. The family $\frak{J}$ is automatically $S^1$-equivariant as in \eqref{eqn:frakJequiv} since $\frak{J}_{t,z}$ does not depend on $t\in S^1$ or $z\in ES^1$. Given distinct pairs $(x_+,\gamma_+)$ and $(x_-,\gamma_-)$, we now describe the moduli space $\M^\frak{J}((x_+,\gamma_+),(x_-,\gamma_-))$.

Let $\J$ denote the constant $S^1$-family $\{J\}$ as before. It follows from Definition~\ref{def:fammod} that $\widehat{\M}^\frak{J}((x_+,\gamma_+),(x_-,\gamma_-))$ is the set of pairs $(\eta,u)$ where $\eta$ is a parametrized flow line of $\widetilde{V}$ from a lift of $x_-$ to a lift of $x_+$, and $u\in\widetilde{\widetilde{\M}}^\J(\gamma_+,\gamma_-)$; here the latter space is defined like $\widetilde{\M}^\J(\gamma_+,\gamma_-)$ but without modding out by $\R$ translation in the domain, cf.\ \S\ref{sec:tms}. We can write this as
\begin{equation}
\label{eqn:MtildefrakJ}
\widehat{\M}^\frak{J}((x_+,\gamma_+),(x_-,\gamma_-)) = \widetilde{\M}^{\op{Morse}}_{\widetilde{V}}(\pi^{-1}(x_+),\pi^{-1}(x_-)) \times \widetilde{\widetilde{\M}}^\J(\gamma_+,\gamma_-),
\end{equation}
where $\widetilde{\M}^{\op{Morse}}_{\widetilde{V}}$ denotes the set of parametrized flow lines of $\widetilde{V}$ as above. In particular, if $i_\pm$ denotes the Morse index of $x_\pm$, then
\begin{equation}
\label{eqn:conaf}
\widehat{\M}_d^\frak{J}((x_+,\gamma_+),(x_-,\gamma_-)) = \widetilde{\M}^{\op{Morse}}_{\widetilde{V}}(\pi^{-1}(x_+),\pi^{-1}(x_-)) \times \widetilde{\widetilde{\M}}_{d-i_++i_-}^\J(\gamma_+,\gamma_-).
\end{equation}
Taking the quotient of \eqref{eqn:conaf} by the $\R\times S^1$ action in Definition~\ref{def:fammodaction}, we obtain a fiber bundle
\begin{equation}
\label{eqn:bundle}
\begin{CD}
\widetilde{\M}^\J_{d-i_++i_-}(\gamma_+,\gamma_-) @>>> \widetilde{\M}^\frak{J}_d((x_+,\gamma_+),(x_-,\gamma_-)) \\
& & @VVV \\
& & \widetilde{\M}^{\op{Morse}}_V(x_+,x_-).
\end{CD}
\end{equation}
Here $\widetilde{\M}^{\op{Morse}}_V(x_+,x_-)$ denotes the moduli space of parametrized flow lines of $V$ from $x_-$ to $x_+$.

An important special case of \eqref{eqn:bundle} is when $x_+=x_-=x$, so that the base of \eqref{eqn:bundle} is a single point. In this case, after choosing a lift $\widetilde{x}\in ES^1$ of $x$, the constant flow line from $\widetilde{x}$ to itself determines a diffeomorphism
\[
\widetilde{\M}_d^\frak{J}((x,\gamma_+),(x,\gamma_-)) \simeq \widetilde{\M}_d^\J(\gamma_+,\gamma_-).
\]
We can further mod out by the $\R$ action in the target to obtain a diffeomorphism
	\begin{equation}
	\label{eqn:cflb}
	\M_d^\frak{J}((x,\gamma_+),(x,\gamma_-)) \simeq \M_d^\J(\gamma_+,\gamma_-).
	\end{equation}
	This last diffeomorphism is orientation preserving, by Lemma~\ref{lem:chautor}.

Another important special case of \eqref{eqn:bundle} is when $\gamma_+=\gamma_-=\gamma$. In this case, every curve $u\in\widetilde{\M}^\J(\gamma,\gamma)$ maps to the ``trivial cylinder'' $\R\times\overline{\gamma}$. Thus modding out by the $\R$ action on $\widetilde{\M}^\frak{J}$ by translation of the targets is equivalent to modding out by the reparametrization action on $\widetilde{\M}^{\op{Morse}}_V(x_+,x_-)$, so we can replace the bundle \eqref{eqn:bundle} by
\begin{equation}
\label{eqn:lastbundle}
\begin{CD}
\widetilde{\M}^\J_{d-i_++i_-}(\gamma,\gamma) @>>> \M^\frak{J}_d((x_+,\gamma),(x_-,\gamma)) \\
& & @VVV \\
& & \M^{\op{Morse}}_V(x_+,x_-).
\end{CD}
\end{equation}
Here $\M^{\op{Morse}}_V(x_+,x_-)$ denotes the space of flow lines of $V$ from $x_-$ to $x_+$ modulo reparametrization. The fiber
\begin{equation}
\label{eqn:thefiber}
\widetilde{\M}_{d-i_++i_-}^\J(\gamma,\gamma) \simeq \left\{\begin{array}{cl}
S^1, & d=i_+-i_-,\\
\emptyset, & \mbox{otherwise}.
\end{array}
\right.
\end{equation}
When $d=i_+-i_-$, the $S^1$ above is just the set of parametrizations of $\gamma$. 

\medskip

{\em Step 2.}
Recall that we are given a set $\mc{P}$, consisting of a choice of base point $p_\gamma\in\overline{\gamma}$ for each Reeb orbit $\gamma$, for the purpose of defining nonequivariant cascades. To make an analogous choice to define equivariant cascades, for each critical point $x$ of $f$ on $BS^1$, fix a lift $\widetilde{x}\in ES^1$. For each pair $(x,\gamma)$ where $x$ is a critical point of $f$ and $\gamma$ is a Reeb orbit, fix the base point
\[
p_{(x,\gamma)} = [(\widetilde{x},p_\gamma)] \in \overline{(x,\gamma)}.
\]
Let $\frak{P}$ denote the set of these choices.

We claim that the pair $(\frak{J},\frak{P})$ satisfies an analogue of the admissibility in Definition~\ref{def:admissible}. That is, for all distinct pairs $(x_+,\gamma_+)$ and $(x_-,\gamma_-)$, we have:
\begin{description}
	\item{(a)} If $d\le 0$, then $\M^\frak{J}_d((x_+,\gamma_+),(x_-,\gamma_-)) = \emptyset$.
	\item{(b)} $\M^\frak{J}_1((x_+,\gamma_+),(x_-,\gamma_-))$ is cut out transversely.
	\item{(c)} $\M^\frak{J}_1\left((x_+,\gamma_+),p_{(x_+,\gamma_+)}, (x_-,\gamma_-),p_{(x_-,\gamma_-)}\right) = \emptyset$.
	\item{(d)} $\M^\frak{J}_2\left((x_+,\gamma_+),p_{(x_+,\gamma_+)}, (x_-,\gamma_-),p_{(x_-,\gamma_-)}\right)$ is cut out transversely.
	\item{(e)} If $(x_0,\gamma_0)$ is distinct from $(x_+,\gamma_+)$ and $(x_-,\gamma_-)$, then
	\[
	\M^\frak{J}_1\left((x_+,\gamma_+),p_{(x_+,\gamma_+)},(x_0,\gamma_0)\right) \times_{\overline{(x_0,\gamma_0)}} \M^\frak{J}_1\left((x_0,\gamma_0),(x_-,\gamma_-),p_{(x_-,\gamma_-)}\right) = \emptyset.
	\] 
\end{description}

To prove (a), observe by \eqref{eqn:bundle} that if $\M^\frak{J}_d((x_+,\gamma_+),(x_-,\gamma_-))$ is nonempty, then $\M^\J_{d'}(\gamma_+,\gamma_-)$ is nonempty for some $d'\le d$. Since $(J,\mc{P})$ is admissible, we have $d'\ge 0$, so $d\ge 0$.

Next, observe that by \eqref{eqn:bundle} and \eqref{eqn:cflb}, we have
\begin{equation}
\label{eqn:mfrakj1}
\M^\frak{J}_1((x_+,\gamma_+),(x_-,\gamma_-)) \simeq \left\{\begin{array}{cl} \M^\J_1(\gamma_+,\gamma_-), & \mbox{if $x_+=x_-$},\\
\emptyset, & \mbox{if $x_+\neq x_-$}.
\end{array}
\right.
\end{equation}
Given \eqref{eqn:mfrakj1}, assertions (b), (c), and (e) follow from the hypothesis that the pair $(J,\mc{P})$ is admissible.

To prove assertion (d), if $x_+=x_-$ then we are done by \eqref{eqn:cflb} and the admissibility of $(J,\mc{P})$. If $x_+\neq x_-$, then by \eqref{eqn:bundle}, the moduli space $\M^\frak{J}_2((x_+,\gamma_+),(x_-,\gamma_-))$ is empty unless  $\gamma_+=\gamma_-$. (Otherwise the fiber in \eqref{eqn:bundle} would have the form $\widetilde{\M}^\J_{d'}(\gamma_+,\gamma_-)$ with $d'\le 0$ and $\gamma_+\neq\gamma_-$, contradicting the hypothesis that $(J,\mc{P})$ is admissible.) When $\gamma_+=\gamma_-$ we must have $\op{ind}(x_+)-\op{ind}(x_-)=2$ by \eqref{eqn:thefiber}. In this case, assertion (d) follows from \eqref{eqn:wcti} below, together with the transversality calculation in Example~\ref{ex:trivcyl}.

\medskip

{\em Step 3.} We now show that if $\op{ind}(x_+) - \op{ind}(x_-)=2$ and $\gamma$ is any Reeb orbit, then
\begin{equation}
\label{eqn:covspace}
e_+\times e_-: \M^\frak{J}_2((x_+,\gamma),(x_-,\gamma)) \longrightarrow \overline{(x_+,\gamma)}\times \overline{(x_-,\gamma)}
\end{equation}
is a covering space of degree $d(\gamma)$.

We can identify $\overline{(x_\pm,\gamma)}$ with $\overline{\gamma}$ by sending the equivalence class of $(\widetilde{x}_\pm,y)$ to $y$.

If $\eta$ is a flow line of $\widetilde{V}$ from $\pi^{-1}(x_-)$ to $\pi^{-1}(x_+)$, that is $\lim_{s\to\pm\infty}\eta(s)\in\pi^{-1}(x_\pm)$, define
\[
e_\pm(\eta) = \lim_{s\to\pm\infty}\eta(s) \in \pi^{-1}(x_\pm).
\]
If $\overline{\eta}\in\M^{\op{Morse}}_V(x_+,x_-)$, then since the vector field $V$ is $S^1$-invariant, there is a unique $\eta\in\M^{\op{Morse}}_{\widetilde{V}}(\pi^{-1}(x_+),\pi^{-1}(x_-))$ lifting $\overline{\eta}$ with $e_-(\eta)=\widetilde{x}_-$. We can then identify
\begin{equation}
\label{eqn:wcti}
\M^\frak{J}_2((x_+,\gamma),(x_-,\gamma)) \simeq \M^{\op{Morse}}_V(x_+,x_-) \times \widetilde{\M}^\J_0(\gamma,\gamma)
\end{equation}
by sending $(\overline{\eta},u)$ on the right hand side to the equivalence class of the pair $(\eta,u)$ on the left hand side.

Now define a map
\begin{equation}
\label{eqn:Deltadiffeo}
\Delta: \M^{\op{Morse}}_V(x_+,x_-) \longrightarrow S^1
\end{equation}
by
\[
\Delta(\overline{\eta}) = e_+(\eta) - \widetilde{x}_+.
\]
It then follows from \eqref{eqn:s1actm} that under the identification \eqref{eqn:wcti}, the map \eqref{eqn:covspace} is given by
\begin{equation}
\label{eqn:confmap}
(e_+\times e_-)(\overline{\eta},u) = \left( e_+(u) + d(\gamma)\Delta(\overline{\eta}), e_-(u) \right) \in \overline{\gamma}\times\overline{\gamma}.
\end{equation}

It follows from equation \eqref{eqn:efl} that the map \eqref{eqn:Deltadiffeo} is a diffeomorphism. Moreover, because of the sign conventions for the $S^1$ action in \eqref{eqn:s1actm}, this diffeomorphism is orientation preserving.

In addition, we know from \eqref{eqn:thefiber} that the map
\begin{equation}
\label{eqn:ediffeo}
e_\pm: \widetilde{\M}^\J_0(\gamma,\gamma) \longrightarrow \overline{\gamma}
\end{equation}
is a diffeomorphism. Moreover, this diffeomorphism is orientation preserving, by the argument in the proof of Lemma~\ref{lem:scacob}.

Under the identifications \eqref{eqn:Deltadiffeo} and \eqref{eqn:ediffeo}, we can rewrite \eqref{eqn:confmap} as
\begin{equation}
\label{eqn:submersion}
\begin{split}
e_+\times e_- : S^1\times \overline{\gamma} & \longrightarrow \overline{\gamma}\times\overline{\gamma},\\
(t,y) & \longmapsto (d(\gamma)t + y, y).
\end{split}
\end{equation}
This is an orientation preserving covering of degree $d(\gamma)$.

\medskip

{\em Step 4.\/} We claim now that the moduli spaces $\M_d^\frak{J}((x_+,\gamma_+),(x_-,\gamma_-))$, with or without point constraints at $p_{(x_+,\gamma_+)}$ and/or $p_{(x_-,\gamma_-)}$, determine a well-defined cascade differential $\partial^{\frak{J}}_\ca$ on $NCC_*(Y,\lambda)\tensor \Z[U]$. The proof closely follows the proof of Proposition~\ref{prop:nchaut}(a), with the following three modifications. First, in place of conditions (a)--(e) in Definition~\ref{def:admissible}, one uses the corresponding conditions (a)--(e) in Step 2 above. Second, in place of equation \eqref{eqn:mjds1} when $d=1$ or $d=2$, one uses equations \eqref{eqn:mfrakj1} and \eqref{eqn:wcti} respectively. Third, in place of Proposition~\ref{prop:compactness} one uses Proposition~\ref{prop:famcom}.

We now show that
\[
\partial^{\frak{J}}_\ca = \partial^{S^1,J}.
\]
More precisely, let $(x_+,\gamma_+)$ and $(x_-,\gamma_-)$ be distinct. Write $\op{ind}(x_\pm)=2k_\pm$. We need to show that
\begin{equation}
\label{eqn:step4}
\left\langle \partial^{\frak{J}}_\ca\left(\widetilde{\gamma}_+\tensor U^{k_+}\right), \widetilde{\gamma}_-\tensor U^{k_-} \right\rangle = \left\langle \partial^{S^1,J}\left(\widetilde{\gamma}_+\tensor U^{k_+}\right), \widetilde{\gamma}_-\tensor U^{k_-}\right\rangle.
\end{equation}
Here $\widetilde{\gamma}_+$ denotes $\widecheck{\gamma}_+$ or $\widehat{\gamma}_+$, and $\widetilde{\gamma}_-$ denotes $\widecheck{\gamma}_-$ or $\widehat{\gamma}_-$. The left hand side of \eqref{eqn:step4} counts cascades from $(x_+,\gamma_+)$ to $(x_-,\gamma_-)$, and the right hand side of \eqref{eqn:step4} is defined by equation \eqref{eqn:ds1j}.

If $x_+=x_-$, then equation \eqref{eqn:step4} follows from the orientation preserving diffeomorphism \eqref{eqn:cflb}. If $x_+\neq x_-$, then by \eqref{eqn:ds1j} and \eqref{eqn:bundle}, both sides of \eqref{eqn:step4} are zero except when $\op{ind}(x_+) - \op{ind}(x_-) = 2$ and there is a Reeb orbit $\gamma$ such that $\widetilde{\gamma}_+=\widecheck{\gamma}$ and $\widetilde{\gamma}_-=\widehat{\gamma}$. So by \eqref{eqn:defbv}, to complete the proof of \eqref{eqn:step4}, we need to check that if $k\ge 1$ and $\gamma$ is any Reeb orbit, then
\begin{equation}
\label{eqn:step4end}
\left\langle \partial^{\frak{J}}_\ca\left(\widecheck{\gamma}\tensor U^{k}\right), \widehat{\gamma}\tensor U^{k-1} \right\rangle = \left\{\begin{array}{cl}
d(\gamma), & \mbox{$\gamma$ good},\\
0, & \mbox{$\gamma$ bad}.
\end{array}\right.
\end{equation}

By Step 3, the left hand side of \eqref{eqn:step4end} is a signed count of $d(\gamma)$ points. These points are the $d(\gamma)$ inverse images of $(p_{(x_+,\gamma)},p_{(x_-,\gamma)})$ under the map \eqref{eqn:covspace}. In the notation of \eqref{eqn:submersion}, these inverse images all have the same $y$ coordinate, while their $t$ coordinates are evenly spaced around $S^1$.

If $\gamma$ is a bad orbit, then if we rotate $t$ from one such point to the next, the signs alternate as in the proof of \eqref{eqn:partial+}, because $e_+$ rotates once around $\overline{(p_+,\gamma)}\simeq \overline{\gamma}$, while $e_-$ stays fixed. Thus the signed count is zero.

If $\gamma$ is a good orbit, then the signs are all the same. Moreover these signs are all positive, because the diffeomorphism \eqref{eqn:wcti} is orientation preserving\footnote{Note that this statement only makes sense when $\gamma$ is good, because the local system in which the orientation of the left hand side of \eqref{eqn:wcti} takes values is trivial if and only if $\gamma$ is good.}. This last statement follows from the fact that the diffeomorphism \eqref{eqn:ediffeo} is orientation preserving, together with the exact sequence \eqref{eqn:fces} used to orient the left hand side of \eqref{eqn:wcti}.

\medskip

{\em Step 5.\/}
We now conclude the proof, similarly to Proposition~\ref{prop:nchaut}(c).

As in the proof of Proposition~\ref{prop:nchaut}(b), for any real number $L$ and any integer $K$, we can choose a generic $S^1$-equivariant $S^1\times ES^1$-family of almost complex structures $\frak{J}'$ which is close on $S^1\times S^{2K+1}$ to $\frak{J}=\{J\}$, such that $\partial^{\frak{J}'}_\ca = \partial^{J,S^1}$ on all generators $\widetilde{\alpha}\tensor U^k$ with $\mc{A}(\alpha)<L$ and $k\le K$. It follows that
\begin{equation}
\label{eqn:yadl}
CH_*^{S^1,L,K}(Y,\lambda) = H_*\left(NCC_*^L(Y,\lambda)\tensor \Z[U]/U^{K+1}, \partial^{S^1,J}\right).
\end{equation}
Here $CH_*^{S^1,L,K}(Y,\lambda)$ denotes the homology of $\left(NCC_*^L(Y,\lambda) \tensor \Z[U]/U^{K+1},\partial^{\frak{J}'}_\ca\right)$ for generic $\frak{J}'$; the proof of Theorem~\ref{thm:CHS1invariant} shows that this depends only on $(Y,\lambda)$. Taking the direct limit over $L$ and $K$ in \eqref{eqn:yadl} completes the proof of the proposition.
\end{proof}

\subsection{Comparison with cylindrical contact homology}
\label{sec:comparison}

We now prove Theorem~\ref{thm:cchinv}, asserting that equivariant contact homology tensor $\Q$ agrees with cylindrical contact homology, when the latter is defined.

\begin{proof}[Proof of Theorem~\ref{thm:cchinv}.]
Assume that the pair $(J,\mc{P})$ is admissible. The pair $(J,\mc{P})$ determines operators $\partial^J_\ca$ and $\partial^{J,S^1}$ as in \S\ref{sec:nchaut} and \S\ref{sec:chaut}. By Proposition~\ref{prop:chaut}, we have a canonical isomorphism
\[
CH_*^{S^1}(Y,\xi) = H_*\left(CC_*^{S^1}(Y,\lambda), \partial^{S^1,J}\right).
\]
So to complete the proof of Theorem~\ref{thm:cchinv}, we need to show that there is a canonical isomorphism
\[
H_*\left(CC_*^{S^1}(Y,\lambda), \partial^{S^1,J}\right) \tensor \Q = CH_*^{\op{EGH}}(Y,\lambda;J).
\]
We proceed in three steps.

\medskip

{\em Step 1.\/} Recall that $CC_*^{S^1}(Y,\lambda)$ is the free $\Z$-module generated by symbols $\widecheck{\alpha}\tensor U^k$ and $\widehat{\alpha}\tensor U^k$ where $\alpha$ is a Reeb orbit and $k$ is a nonnegative integer. Let $C_*'$ denote the submodule generated by all of the above generators {\em except\/} generators of the form $\widecheck{\beta}\tensor 1$ where $\beta$ is a good Reeb orbit. We claim that $C_*'$ is a subcomplex of $CC_*^{S^1}(Y,\lambda)$.

To prove this, let $\beta$ be a good Reeb orbit. We need to show that $\widecheck{\beta}\tensor 1$ does not appear in $\partial^{S^1,J}$ of any generator that does not have the form $\widecheck{\alpha}\tensor 1$ where $\alpha$ is a good Reeb orbit.  That is, we need to show the following:
\begin{description}
	\item{(i)} If $\alpha$ is a bad Reeb orbit then
	\[
	\left\langle \partial^{S^1,J}\left(\widecheck{\alpha}\tensor 1\right),\widecheck{\beta}\tensor 1\right\rangle = 0.
	\]
	\item{(ii)} If $\alpha$ is any Reeb orbit and $k$ is a positive integer then
	\[
	\left\langle\partial^{S^1,J} \left(\widecheck{\alpha}\tensor U^k\right),\widecheck{\beta}\tensor 1\right\rangle = 0.
	\]
	\item{(iii)} If $\alpha$ is any Reeb orbit and $k$ is a nonnegative integer then
	\[
	\left\langle \partial^{S^1,J}\left(\widehat{\alpha}\tensor U^k\right),\widecheck{\beta}\tensor 1\right\rangle = 0.
	\]
\end{description}
Assertion (i) follows from equation \eqref{eqn:partialcheck}. Assertion (ii) follows from the definition of $\partial^{S^1,J}$. Assertion (iii) follows from the definition of $\partial^{S^1,J}$ when $k>0$, and from equation \eqref{eqn:partial+} when $k=0$.

\medskip

{\em Step 2.} We now show that the homology of the subcomplex $C_*'$ vanishes after tensoring with $\Q$, that is
\begin{equation}
\label{eqn:ifsc}
H_*\left(C'\tensor\Q,\partial^{S^1,J}\tensor 1\right) = 0.
\end{equation}

To see this, note that $\partial^{S^1,J}$ does not increase symplectic action, where the ``symplectic action'' of a generator $\widetilde{\alpha}\tensor U^k$ is understood to be the symplectic action of the Reeb orbit $\alpha$. Let $\partial^{S^1,J}_0$ denote the part of $\partial^{S^1,J}$ that fixes symplectic action. By a spectral sequence argument\footnote{Concretely, let $0<A_1<A_2<\cdots$ denote the (discrete) values of the symplectic action. We then define an integer-valued filtration on the complex in \eqref{eqn:ifsc}, where the $i^{th}$ filtration level is spanned by generators $\widetilde{\alpha}\tensor U^k$ with $\mc{A}(\alpha)\le A_i$. If \eqref{eqn:e1term} holds, then the homology of the associated graded complex vanishes, so by induction on $i$, the homology of the $i^{th}$ filtered subcomplex also vanishes.}, it is enough to show that
\begin{equation}
\label{eqn:e1term}
H_*\left(C_*'\tensor\Q,\partial^{S^1,J}_0\tensor 1\right) = 0.
\end{equation}

The chain complex $\left(C_*',\partial^{S^1,J}_0\right)$ splits into a direct sum over subcomplexes indexed by the set of Reeb orbits. For a Reeb orbit $\alpha$, the corresponding subcomplex is the span of the generators $\widecheck{\alpha}\tensor U^k$ for $k>0$ (and also $k=0$ if $\alpha$ is bad) and $\widehat{\alpha}\tensor U^k$ for $k\ge 0$. We need to show that the homology of this subcomplex vanishes after tensoring with $\Q$. When $\alpha$ is good, the subcomplex is the sum over nonnegative integers $k$ of two-term complexes
\[
\widecheck{\alpha}\tensor U^{k+1} \stackrel{d(\alpha)}{\longrightarrow} \widehat{\alpha}\tensor U^k.
\]
Thus the homology of the subcomplex is an infinite direct sum of copies of $\Z/d(\alpha)\Z$, which vanishes after tensoring with $\Q$. If $\alpha$ is bad, then the subcomplex is the sum over nonnegative integers $k$ of two-term complexes
\[
\widehat{\alpha}\tensor U^k \stackrel{2}{\longrightarrow} \widecheck{\alpha}\tensor U^k.
\]
Thus the homology of the subcomplex is an infinite direct sum of copies of $\Z/2\Z$, which again vanishes after tensoring with $\Q$. (Similar calculations appeared previously in \cite[\S3.2]{bo12} and \cite[\S3.2]{gu}.)

\medskip

{\em Step 3.\/} We now complete the proof. It follows from Step 2 that after tensoring with $\Q$, the homology of the chain complex $\left(CC_*^{S^1}(Y,\lambda),\partial^{S^1,J}\right)$ is canonically isomorphic to the homology of the quotient complex by $C_*'$. That is,
\begin{equation}
\label{eqn:hq1}
CH_*^{S^1}(Y,\xi)\tensor \Q = H_*\left(\left(CC_*^{S^1}(Y,\lambda)/C_*'\right)\tensor\Q, \partial^{S^1,J}\tensor 1\right).
\end{equation}
A basis for the quotient complex by $C_*'$ is given by the generators $\widecheck{\alpha}\tensor 1$ where $\alpha$ is good. By the definition of $\partial^{S^1,J}$, the differential on the quotient complex is induced by $\widecheck{\partial}$. By equation \eqref{eqn:partialcheck}, this differential agrees with $\partial^{EGH}$ after tensoring with $\Q$. Thus we have a canonical isomorphism of chain complexes
\begin{equation}
\label{eqn:hq2}
\left(\left(CC_*^{S^1}(Y,\lambda)/C_*'\right)\tensor\Q,\partial^{S^1, J}\tensor 1\right) = \left(CC_*^{EGH}(Y,\lambda;J),\partial^{EGH}\right).
\end{equation}
Theorem~\ref{thm:cchinv} now follows from \eqref{eqn:hq1} and \eqref{eqn:hq2}.
\end{proof}

\subsection{Existence of admissible $J$ in dimension 3}
\label{sec:admissible3}

\begin{proposition}
\label{prop:admissible}
Suppose that $\op{dim}(Y)=3$. Then a generic $\lambda$-compatible almost complex structure $J$ on $\R\times Y$ is admissible.
\end{proposition}

To prove Proposition~\ref{prop:admissible}, we will use the following facts from \cite{dc}:

\begin{lemma}
\label{lem:fromdc}
Let $(Y,\lambda)$ be a nondegenerate\footnote{Lemma~\ref{lem:fromdc} does not require $\lambda$ to be hypertight.} contact three-manifold, and let $J$ be a generic $\lambda$-compatible almost complex structure on $\R\times Y$. Then for distinct Reeb orbits $\alpha,\beta$, we have:
\begin{description}
\item{{\em (a)}} If $u\in\M^J(\alpha,\beta)$ with $\alpha\neq \beta$, and if $\overline{u}$ is the somewhere injective curve underlying $u$, then
\[
1\le \op{ind}(\overline{u}) \le \op{ind}(u).
\]
\item{\emph{(b)}} $\M_1^J(\alpha,\beta)$ is cut out transversely. More precisely, each $u\in\M_1^J(\alpha,\beta)$ is an immersion, whose normal deformation operator (see Remark~\ref{rem:altreg}) is surjective.
\item{\emph{(c)}} For each $u\in\M_2^J(\alpha,\beta)$, if $u$ is not a double cover of an index $1$ cylinder $\overline{u}$, then $\M_2^J(\alpha,\beta)$ is a smooth manifold near $u$ cut out transversely.
\end{description}
\end{lemma}

\begin{proof}[Proof of Lemma~\ref{lem:fromdc}.]
(a) This is \cite[Lem.\ 2.5(a)]{dc}.

(b) This is \cite[Lem.\ 4.2(a)]{dc}.

(c) This follows from the proof of \cite[Lem.\ 4.2(b)]{dc}.
\end{proof}

\begin{remark}
In part (c), if $u\in\M_2^J(\alpha,\beta)$ is a double cover of an index $1$ cylinder $\overline{u}$, then $u$ may be an orbifold point of $\M_2^J(\alpha,\beta)$. We expect that $\M_2^\J(\alpha,\beta)$ is still a manifold near $u$.
\end{remark}

\begin{proof}[Proof of Proposition~\ref{prop:admissible}.]
Choose generic $J$ as in Lemma~\ref{lem:fromdc}. We claim that $J$ is admissible.

Note that if $u\in \M^\J_d(\alpha,\beta)$, then $\M^\J_d(\alpha,\beta)$ is cut out transversely near $u$, i.e.\ $u$ is regular in the sense of Definition~\ref{def:regular}, if $\M^J_d(\alpha,\beta)$ is cut out transversely near $u$ (more precisely the equivalence class of $u$ in equation \eqref{eqn:mjds1}) in the sense of Lemma~\ref{lem:fromdc}. This is explained in Remark~\ref{rem:altreg}.

Admissibility conditions (a) and (b) in Definition~\ref{def:admissible} now follow from parts (a) and (b) of Lemma~\ref{lem:fromdc}.  By Sard's theorem, we can then choose $\mc{P}$ generically so that admissibility conditions (c) and (e) hold.

Next observe that $\M^\J_2(\alpha,\beta)$ is a smooth manifold cut out transversely near each $u\in\M^\J_2(\alpha,p_\alpha,\beta,p_\beta)$. Otherwise, by Lemma~\ref{lem:fromdc}, $u$ would be a cover of an index $1$ curve $\overline{u}\in \M^\J_1(\alpha,p_\alpha,\beta,p_\beta)$, but the latter cannot exist by admissibility condition (c).

By Sard's theorem again, we can then choose $\mc{P}$ generically so that admissibility condition (d) holds also. 
\end{proof}

\section{Examples}
\label{examples}

The nonequivariant and $S^1$-equivariant contact homology that we have defined have integer coefficients and contain some interesting torsion. We now compute some examples of this, and we also explain more about how to define local versions of these contact homologies, as sketched in \S\ref{sec:addstr}.

\subsection{Prequantization spaces}


Let $\Sigma$ be a Riemann surface of genus $g>0$, let $\rho:Y\to\Sigma$ be an $S^1$-bundle over $\Sigma$ with Euler class $-e$ where $e>0$, and let $\lambda_0$ be a connection $1$-form on $Y$ with positive curvature. Then $\lambda_0$ is a contact form; let $\xi$ denote the contact structure $\Ker(\lambda_0)$. The simple Reeb orbits of $\lambda_0$ consist of the fibers of $\rho$, which all have action $2\pi$, so that $\lambda_0$ is hypertight, although degenerate. Let $\Gamma$ be the free homotopy class of loops in $Y$ given by $d$ times a fiber where $d>0$. We now sketch how to compute the nonequivariant and equivariant contact homologies of $(Y,\xi)$ in the class $\Gamma$. A related calculation was given in \cite[Thm.\ 1.19]{jo2}.

By the discussion in \S\ref{sec:addstr}, we have
\[
NCH_*(Y,\xi,\Gamma) = NCH_*^{<L}(Y,\lambda,\Gamma)
\]
where $L>2\pi d$, and $\lambda$ is a nondegenerate perturbation of $\lambda_0$ which is small with respect to $L$. An analogous equation holds for equivariant contact homology.

The usual approach for perturbing $\lambda_0$ is to take
\[
\lambda = (1+\rho^*H)\lambda_0
\]
where $H:\Sigma\to\R$ is a $C^2$-small Morse function. The Reeb orbits of $\lambda$ in the class $\Gamma$ of action less than $L$ then consist of the $d$-fold covers of the fibers over the critical points of $H$. These Reeb orbits are all good.

As explained in \S\ref{sec:addstr}, since $c_1(\xi)$ vanishes on toroidal classes in $H_2(Y)\simeq H_1(\Sigma)$, we can noncanonically refine the canonical $\Z/2$ grading on $NCC_*$ to a $\Z$ grading. We can choose this $\Z$ grading such that if $p$ is a critical point of $H$, and if $\gamma_p$ denotes the $d$-fold cover of the fiber over $p$, then
\[
\begin{split}
\left|\widecheck{\gamma_p}\right| &= \op{ind}(H,p) - 1,\\
\left|\widehat{\gamma_p}\right| &= \op{ind}(H,p)
\end{split}
\]
where $\op{ind}(H,p)$ denotes the Morse index of $H$ at $p$.

Next, by Proposition~\ref{prop:admissible}, we can choose an admissible $\lambda$-compatible almost complex structure $J$ on $\R\times Y$, which is close to an $S^1$-invariant $\lambda_0$-compatible almost complex structure coming from a Riemannian metric on $\Sigma$. For suitable orientation choices, if $p$ and $q$ are critical points of $H$, then there is an orientation-preserving bijection
\begin{equation}
\label{eqn:opbhm}
\M^J(\gamma_p,\gamma_q) = \M^{\op{Morse}}(p,q),
\end{equation}
where the left hand side is the moduli space of $J$-holomorphic cylinders \eqref{eqn:modholcyl}, while the right hand side denotes the moduli space of downward gradient flow lines of $H$ from $p$ to $q$, modulo reparametrization. Furthermore, each of the holomorphic cylinders on the left hand side is a $d$-fold cover which is cut out transversely. The above is proved by Moreno \cite[\S3.5.1,\S6.1]{moreno}, modifying the classic arguments of Salamon-Zehnder \cite{sz} computing Hamiltonian Floer homology of $C^2$-small autonomous Hamiltonians.

By Proposition~\ref{prop:nchaut}, we can now compute most of the nonequivariant differential \eqref{eqn:blockmatrix} as follows:
\begin{equation}
\label{eqn:preq1}
\begin{split}
\partial_+ &= 0,\\
\widecheck{\partial} &= \partial_{\op{Morse}},\\
\widehat{\partial} &= -\partial_{\op{Morse}},
\end{split}
\end{equation}
where $\partial_{\op{Morse}}$ denotes the differential on the Morse complex for $H$. 

For simplicity, we can choose $H$ to be a perfect Morse function with one index $2$ critical point $p$, with $2g$ index $1$ critical points $q_1,\ldots,q_{2g}$, and with one index $0$ critical point $r$. Then by \eqref{eqn:preq1}, the only possibly nonzero block in the cascade differential \eqref{eqn:blockmatrix} is $\partial_-$. Moreover, for grading reasons, the only possible nonzero coefficient of $\partial_-$ is the coefficient from $\widecheck{p}$ to $\widehat{r}$. A calculation similar to the proof of equation \eqref{eqn:step4end} shows that this coefficient is given by
\begin{equation}
\label{eqn:preq2}
\langle \partial_-\widecheck{p},\widehat{r}\rangle = \pm de.
\end{equation}

We conclude that the nonequivariant contact homology is generated by $\widehat{p}, \widecheck{q}_i, \widehat{q}_i, \widecheck{r},$ which are free, and $\widehat{r}$, which is $de$-torsion. Thus
\[
NCH_*(Y,\xi,\Gamma) \simeq \left\{\begin{array}{cl} H_2(\Sigma), & *=2,\\
H_1(\Sigma), & *=1,\\
H_1(\Sigma)\oplus \Z/de, & *=0,\\
H_0(\Sigma), & *=-1,\\
0, & \mbox{else}. 
\end{array}
\right.
\]

Next we can compute the equivariant contact homology. By Proposition~\ref{prop:chaut}, this is the homology of the differential $\partial^{S^1,J}$ in equation \eqref{eqn:ds1j}.  By \eqref{eqn:preq1} and \eqref{eqn:preq2}, this differential is given by
\[
\begin{split}
\partial^{S^1,J}(\widecheck{p}\tensor 1) &= \pm de \widehat{r}\tensor 1,\\
\partial^{S^1,J}(\widecheck{p}\tensor U^k) &= \pm de \widehat{r}\tensor U^k + d\widehat{p}\tensor U^{k-1}, \quad k>0,\\
\partial^{S^1,J}(\widecheck{q_i}\tensor U^k) &= d\widehat{q_i}\tensor U^{k-1}, \quad k>0,\\
\partial^{S^1,J}(\widecheck{r}\tensor U^k) &= d\widehat{r}\tensor U^{k-1}, \quad k>0, 
\end{split}
\]
and $\partial^{S^1,J}$ sends all other generators, namely $\widecheck{q_i}\tensor 1, \widecheck{r}\tensor 1$ and $\widehat{p}\tensor U^k, \widehat{q_i}\tensor U^k, \widehat{r}\tensor U^k$ with $k\ge 0$, to zero. It follows that the homology is generated by $\widecheck{p}\tensor 1 \mp e\widecheck{r}\tensor U$, $\widecheck{q_i}\tensor 1$, and $\widecheck{r}\tensor 1$, each of which is free, together with $\widehat{p}\tensor U^k, \widehat{q_i}\tensor U^k, \widehat{r}\tensor U^k$ for $k\ge 0$, each of which is $d$-torsion. We conclude the following:

\begin{proposition}
Let $\Sigma$ be a Riemann surface of genus $g>0$, let $\rho:Y\to\Sigma$ be an $S^1$-bundle over $\Sigma$ with negative Euler class, let $\xi$ be the corresponding prequantization contact structure on $Y$, and let $\Gamma$ be the free homotopy class of loops in $Y$ given by $d$ times a fiber where $d>0$. Then as $\Z$-graded $\Z$-modules we have
\begin{equation}
\label{eqn:preqeq}
CH_*^{S^1}(Y,\xi,\Gamma) \simeq H_*(\Sigma)[-1] \oplus \left(H_*(\Sigma;\Z/d) \tensor \Z[U]\right)
\end{equation}
where the formal variable $U$ has degree $2$.
\end{proposition}

\begin{remark}
Strictly speaking, we have not defined the cylindrical contact homology $CH_*^{\op{EGH}}(Y,\xi,\Gamma)$; we have only defined cylindrical contact homology for pairs $(\lambda,J)$ satisfying the hypotheses of Corollary~\ref{cor:eghinv}. But if we define $CH_*^{\op{EGH}}(Y,\xi,\Gamma)$ to be the $S^1$-equivariant contact homology tensor $\Q$ as in Theorem~\ref{thm:cchinv}, then it follows from \eqref{eqn:preqeq}, as one might expect from \eqref{eqn:opbhm}, that
\begin{equation}
\label{eqn:preqcyl}
CH_*^{\op{EGH}}(Y,\xi,\Gamma) \simeq H_*(\Sigma;\Q)[-1].
\end{equation}
\end{remark}

\subsection{Local contact homology}
\label{sec:localcontact}

Let $\lambda_0$ be a contact form on a manifold $Y$ (not necessarily compact), let $\gamma_0$ be a simple Reeb orbit of $\lambda_0$, and let $d$ be a positive integer. Assume that the Reeb orbits $\gamma_0^k$ are isolated in the loop space for $1\le k\le d$ (but not necessarily nondegenerate). We now explain how to define the local nonequivariant contact homology $NCH_*(Y,\lambda_0,\gamma_0,d)$ and the local $S^1$-equivariant contact homology $CH_*^{S^1}(Y,\lambda_0,\gamma_0,d)$.

\paragraph{Isolating neighborhood.} As in \cite[\S3]{hm}, let $N\subset Y$ be a compact tubular neighborhood of (the image of) $\gamma_0$. Choose $N$ sufficiently small so that:
\begin{itemize}
\item There exists a closed $1$-form $\theta$ on $N$ which pairs positively with the Reeb vector field.
\item The only Reeb orbits in $N$ in the homotopy class $k[\gamma_0]$ for $1\le k\le d$ are the iterates $\gamma_0^k$.
\end{itemize}
Let $\lambda$ be a nondegenerate perturbation of $\lambda_0$ on $N$. Assume that the perturbation is small enough so that $\theta$ pairs positively with the Reeb vector field of $\lambda$. (We may need to choose the perturbation smaller below.) Then $\lambda$ is hypertight on $N$. 

\paragraph{Local nonequivariant contact homology.}

Let $\J$ be a generic $S^1$-family of $\lambda$-compatible almost complex structures on $\R\times N$. If $\gamma_+$ and $\gamma_-$ are Reeb orbits of $\lambda$, let $\M^\J(\gamma_+,\gamma_-)$ denote the moduli space of holomorphic cylinders as in \S\ref{section:plainmoduli}, where now
\[
u:\R\times S^1 \to \R\times N.
\]
A compactness argument as in \cite[Lem.\ 3.4]{hm} shows that if $\lambda$ is sufficiently close to $\lambda_0$, and if $\gamma_+$ and $\gamma_-$ are Reeb orbits of $\lambda$ in the class $d[\gamma_0]$, then for any generic $\J$ as above, every holomorphic cylinder in $\M^\J(\gamma_+,\gamma_-)$ avoids $\R\times\partial N$. We then have compactness as in Proposition~\ref{prop:compactness}, i.e.\ the moduli spaces have no additional boundary points arising from holomorphic curves hitting $\R\times\partial N$. We then obtain a well-defined local nonequivariant contact homology as in \S\ref{sec:defnch}, using only Reeb orbits in the class $d[\gamma_0]$, which we denote by $NCH_*(N,\lambda,d;\J)$. Similarly, if $\lambda_+$ and $\lambda_-$ are two sufficiently small perturbations of $\lambda$ as above, and if $\J_\pm$ is a generic $S^1$-family of $\lambda_\pm$-compatible almost complex structures on $\R\times N$, then as in \S\ref{sec:nchcobmaps} and \S\ref{sec:nchinv}, we obtain a canonical isomorphism
\begin{equation}
\label{eqn:lnchinv}
NCH_*(N,\lambda_+,d;\J_+) \simeq NCH_*(N,\lambda_-,d;\J_-).
\end{equation}
Consequently, the local nonequivariant contact homology does not depend on the choice of sufficiently small perturbation $\lambda$ or $\J$, and we can denote this homology by $NCH_*(Y,\lambda_0,\gamma_0,d)$.

\paragraph{Local $S^1$-equivariant contact homology.} 

The local $S^1$-equivariant contact homology $CH_*^{S^1}(Y,\lambda_0,\gamma_0,d)$ is defined likewise, by repeating the construction in \S\ref{equicurrents}, where $\lambda$ is a sufficiently small nondegenerate perturbation of $\lambda_0$ in $N$ as above, $\frak{J}$ is a generic $S^1$-equivariant $S^1\times ES^1$ family of $\lambda$-compatible almost complex structures on $\R\times N$, and we consider Reeb orbits in the class $d[\gamma_0]$ and moduli spaces of pairs $(\eta,u)$ where $u:\R\times S^1\to \R\times N$.

\paragraph{Local cylindrical contact homology.}

Finally, suppose that $\lambda$ is a sufficiently small nondegenerate perturbation of $\lambda_0$ on $N$ as above, and suppose that there exists an admissible $\lambda$-compatible almost complex structure $J$ on $\R\times N$ as in Definition~\ref{def:admissible}. (As in Proposition~\ref{prop:admissible}, we can always find such a $J$ when $\dim(Y)=3$.) Then the cylindrical contact homology $CH_*^{\op{EGH}}(N,\lambda,d;J)$ is defined, where we just consider good Reeb orbits in the class $d[\gamma_0]$. The proof of Theorem~\ref{thm:cchinv} carries over to show that there is a canonical isomorphism
\[
CH_*^{\op{EGH}}(N,\lambda,d;J) = CH_*^{S^1}(Y,\lambda_0,\gamma_0,d)\tensor\Q.
\]

\paragraph{Grading.}

Each of the above versions of local contact homology has a noncanonical $\Z$-grading, as in \S\ref{sec:addstr}. In fact, a symplectic trivialization $\tau_0$ of $\gamma_0^*\xi$ determines a choice of this $\Z$-grading, in which the grading of a Reeb orbit $\gamma$ in cylindrical contact homology, or the corresponding generator $\widehat{\gamma}$ or $\widehat{\gamma}\tensor 1$ of nonequivariant or $S^1$-equivariant contact homology, is the Conley-Zehnder index $\op{CZ}_\tau(\gamma)$, where $\tau$ is a trivialization of $\gamma^*\xi$ in a homotopy class determined by $\tau_0$.


\begin{remark}
There is an alternate and very simple definition of local contact homology over $\Z$, introduced in \cite{adlch}, which avoids transversality difficulties without using $S^1$-equivariant theory. Here one replaces the neighborhood $N$ by a $d$-fold cyclic cover, considers the cylindrical contact homology of this cover in degree $1$ (for which no multiply covered holomorphic cylinders can arise), and takes the $\Z/d$-invariant part of this cylindrical contact homology. Let us denote the resulting local contact homology by $CH_*^\Z(Y,\lambda_0,\gamma_0,d)$. Simple examples show that $CH_*^\Z$ does not agree with $CH_*^{S^1}$ (the former has less torsion), but we expect that there is a canonical isomorphism
\[
CH_*^\Z(Y,\lambda_0,\gamma_0,d)\tensor\Q = CH_*^{S^1}(Y,\lambda_0,\gamma_0,d)\tensor\Q.
\]
\end{remark}

\subsection{The period doubling bifurcation}

On a three-manifold $Y$, a one-parameter family of contact forms $\{\lambda_t\}_{t\in\R}$ can undergo a {\em period-doubling bifurcation\/} in which the following happens\footnote{There is another version of period-doubling, in which $e_1$ has rotation number slightly greater than $1/2$, and $e_2$ has rotation number slightly greater than $1$. This other version behaves similarly and we will not consider it here.}:

\begin{itemize}
	\item The contact form $\lambda_0$ has an simple Reeb orbit $\gamma_0$ for which the linearized return map \eqref{slm} has $-1$ as a single eigenvalue. In particular, the double cover of $\gamma_0$ is nondegenerate. However $\gamma_0$ and its double cover are isolated Reeb orbits in the loop space of $Y$. Let $N$ be a small tubular neighborhood of $\gamma_0$ as in the definition of local contact homology for $d=2$. We can choose a trivialization $\tau_0$ of $\gamma_0^*\xi$ with respect to which the linearized Reeb flow along $\gamma_0$ has rotation number $1/2$.
	\item For $t<0$ small, the only Reeb orbit in $N$ in the homotopy class $[\gamma_0]$ is an elliptic Reeb orbit $e_1$ for which the rotation number with respect to $\tau_0$ is slightly less than $1/2$. The only Reeb orbit in $N$ in the homotopy class $2[\gamma_0]$ is the double cover of $e_1$, which we denote by $E_1$.
	\item For $t>0$ small, the only Reeb orbit in $N$ in the homotopy class $[\gamma_0]$ is a negative hyperbolic Reeb orbit $h_1$ (this means that the linearized return map has distinct negative eigenvalues) with rotation number $1/2$ with respect to $\tau_0$; and the only Reeb orbits in the homotopy class $2[\gamma_0]$ are the double cover of $h_1$, which we denote by $H_1$, together with a simple elliptic orbit $e_2$ whose rotation number with respect to $\tau_0$ is slightly less than $1$.
\end{itemize}

For explicit formulas for this bifurcation see \cite[\S8.3.2]{celest}. The above bullet points are all that we need to know here, but the following picture might be helpful. Let $D$ be a disk transverse to $\gamma_0$. Then the Reeb flow of $\lambda_t$ induces a partially defined return map $\phi_t:D\to D$. One can model the period-doubling bifurcation with
\[
\phi_t = \begin{pmatrix} -1 & 0 \\ 0 & -1\end{pmatrix} \circ \varphi_\epsilon^{X_t}
\]
where $\epsilon>0$ is small, and $\varphi_\varepsilon^{X_t}$ is the time $\varepsilon$ flow of a vector field $X_t$ which is invariant under rotation by $\pi$ and illustrated
in Figures~\ref{fig:before}--\ref{fig:after}.

\begin{figure}[ht]
\begin{minipage}[b]{0.29\linewidth}
\centering
\includegraphics[width=.6\linewidth]{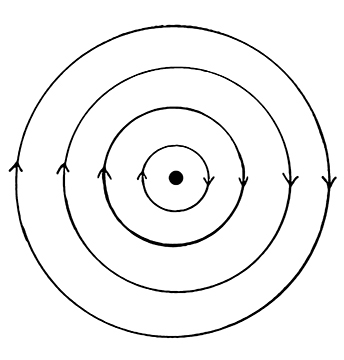}
\caption{\small Flow of $X_{-1}$; \hspace{\textwidth} before the bifurcation.}
\label{fig:before}
\end{minipage}
\hspace{.5cm}
\begin{minipage}[b]{0.29\linewidth}
\centering
\includegraphics[width=.95\linewidth]{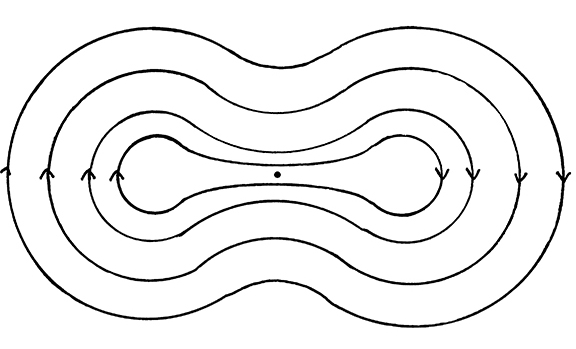}
\caption{\small Flow of $X_0$; \hspace{\textwidth} at the bifurcation.}
\label{fig:during}
\end{minipage}
\hspace{.5cm}
\begin{minipage}[b]{0.31\linewidth}
\centering
\includegraphics[width=\linewidth]{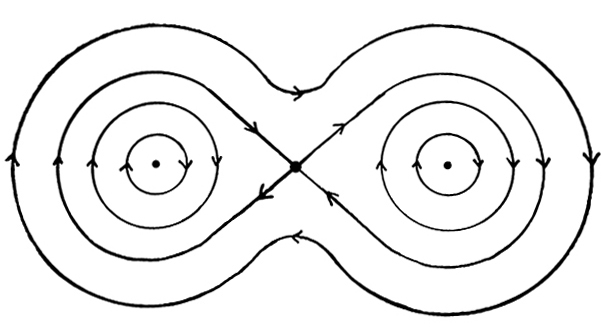}
\caption{\small Flow of $X_1$; \hspace{\textwidth} after the bifurcation.}
\label{fig:after}
\end{minipage}
\end{figure}

The critical point of $X_{-1}$ in Figure~\ref{fig:before} is a fixed point of $\phi_{-1}$, which corresponds to the elliptic orbit $e_1$. The critical point of $X_{0}$ in Figure~\ref{fig:during} is a fixed point of $\phi_{0}$ which corresponds to the degenerate Reeb orbit $\gamma_0$. The central critical point of $X_1$ in Figure~\ref{fig:after} is a fixed point of $\phi_1$ which corresponds to the negative hyperbolic orbit $h_1$. In addition, the left and right critical points of $X_1$ in Figure~\ref{fig:after} are exchanged by $\phi_1$, and so they constitute a period $2$ orbit of the map $\phi_1$, which corresponds to the elliptic orbit $e_2$.

It follows from \cite[Eq.\ (2.3)]{dc} that the Conley-Zehnder indices of the above Reeb orbits with respect to $\tau_0$ are given by $\op{CZ}(e_1)=\op{CZ}(h_1)=1$, $\op{CZ}(E_1)=\op{CZ}(e_2)=1$, and $\op{CZ}(H_1)=2$.

Let $\lambda_+=\lambda_t$ for some small $t>0$, and let $\lambda_-=\lambda_t$ for some small $t<0$. It is instructive to compute the different versions of local contact homology of $(Y,\lambda_0,\gamma_0,d=2)$ using $\lambda_-$ and $\lambda_+$. For this purpose, we can use admissible $\lambda_\pm$-compatible almost complex structures $J_\pm$ on $\R\times N$.

\paragraph{Local cylindrical contact homology.} For $\lambda_-$, the local cylindrical chain complex just has the one generator $E_1$ (this is a good Reeb orbit). Thus the cylindrical contact homology is given by
\begin{equation}
\label{eqn:pdegh}
CH_*^{\op{EGH}}(N,\lambda_-,d=2;J_-) \simeq \left\{\begin{array}{cl} \Q, & *=1,\\ 0, & \mbox{otherwise}.
\end{array}\right.
\end{equation}

For $\lambda_+$, the local cylindrical chain complex has just the one generator $e_2$, because $H_2$ is a bad Reeb orbit. Thus we get the same answer\footnote{This is one of the reasons why bad orbits have to be discarded in the definition of cylindrical contact homology; otherwise one would not have invariance under this bifurcation.}.

\paragraph{Local nonequivariant contact homology.} For $\lambda_-$, the local nonequivariant chain complex has the two generators $\widecheck{E}_1$ and $\widehat{E}_1$ of gradings $1$ and $2$ respectively. Since $E_1$ is a good Reeb orbit, the differential \eqref{eqn:blockmatrix} is zero by equation \eqref{eqn:partial+}. Thus we obtain
\begin{equation}
\label{eqn:codd}
NCH_*(Y,\lambda_0,\gamma_0,d=2) \simeq \left\{\begin{array}{cl} \Z, & *=1,2,\\ 0, & \mbox{otherwise}.\end{array}\right.
\end{equation}

For $\lambda_+$, the local nonequivariant chain complex now has four generators: $\widecheck{e}_2$ which has grading $1$; $\widehat{e}_2$ and $\widecheck{H}_1$ which have grading $2$; and $\widehat{H}_1$ which has grading $3$. By Proposition~\ref{prop:nchaut}(c), the differential \eqref{eqn:blockmatrix} satisfies
\[
\partial^{J_+}_\ca\widehat{H_1} = -2\widecheck{H_1} + c\widehat{e_2},
\]
and the differential of all other generators is zero. Here $c\in\Z$ is determined by a count of holomorphic cylinders in the moduli space $\M^J(H_1,e_2)$.  We will not compute $c$ here, but we do know that $c$ is odd, by equation \eqref{eqn:codd} and the invariance \eqref{eqn:lnchinv} of local nonequivariant contact homology.

\paragraph{Local $S^1$-equivariant contact homology.} 

For $\lambda_-$, the $S^1$-equivariant chain complex has the generators $\widecheck{E_1}\tensor U^k$ of grading $2k+1$, and $\widehat{E_1}\tensor U^k$ of grading $2k+2$, for each $k\ge 0$. By equation \eqref{eqn:ds1j}, the differential is given by
\[
\partial^{S^1,J_-}\left(\widecheck{E_1}\tensor U^k\right) = 2\left(\widehat{E_1}\tensor U^{k-1}\right), \quad k>0,
\]
and the differentials of all other generators are zero. It follows that the homology is generated by the cycle $\widecheck{E_1}$, which is free, and the cycles $\widehat{E_1}\tensor U^k$ for $k\ge 0$, which are $2$-torsion. Thus
\begin{equation}
\label{eqn:pdchs1}
CH_*^{S^1}(Y,\lambda_0,\gamma_0,d=2) \simeq \left\{\begin{array}{cl} \Z, & *=1,\\
\Z/2, & *=2,4,\ldots,\\
0, & \mbox{otherwise}.
\end{array}\right.
\end{equation}
Note that tensoring this with $\Q$ correctly recovers \eqref{eqn:pdegh}.

For $\lambda_+$, the $S^1$-equivariant chain complex has generators $\widecheck{e_2}\tensor U^k$ of grading $2k+1$, generators $\widehat{e_2}\tensor U^k$ and $\widecheck{H_1}\tensor U^k$ of grading $2k+2$, and generators $\widehat{H_1}\tensor U^k$ of grading $2k+3$, for each $k\ge 0$. By equation \eqref{eqn:ds1j}, the differential is given by
\[
\begin{split}
\partial^{S^1,J_+}\left(\widecheck{e_2}\tensor U^k\right) &= \widehat{e_2}\tensor U^{k-1}, \quad k>0,\\
\partial^{S^1,J_+}\left(\widehat{H_1}\tensor U^k\right) &= -2\left(\widecheck{H_1}\tensor U^k\right) + c\left(\widehat{e_2}\tensor U^k\right), \quad k\ge 0,
\end{split}
\]
and the differentials of all other generators are zero. It follows that the homology is generated by the cycle $\widecheck{e_2}\tensor 1$, which is free, and the cycles $\widecheck{H_1}\tensor U^k$ for $k\ge 0$, which are $2$-torsion since
\[
\partial^{S^1,J_+}\left(c\left(\widecheck{e_2}\tensor U^{k+1}\right) - \widehat{H_1}\tensor U^k \right) = 2\left(\widecheck{H_1}\tensor U^k\right).
\]
Thus we again obtain \eqref{eqn:pdchs1}.

\appendix

\section{Orientations}
\label{sec:orientations}

In this section we define the orientations on the moduli spaces that we consider, and we justify the claims that we make about signs. (We omit a few cases which do not involve additional ideas.) The main references that we will use are \cite{fh}, which first introduced coherent orientations in the context of Hamiltonian Floer theory; \cite{bm}, which extended \cite{fh} to the case of symplectic field theory; and \cite{ht2}, which worked out more details in connection with obstruction bundle gluing.

\subsection{Orienting the operators}

We first review how to orient the various Fredholm operators that we need to consider, spelling out the conventions that we will be using. We will consider a very general class of Fredholm operators, and not just the specific operators that arise from holomorphic curves in $\R\times Y$. The conclusion that we need is stated in Proposition~\ref{prop:orcy} below. This discussion follows \cite[\S9]{ht2} with minor modifications.

\paragraph{Preliminaries.}
If $V$ is a finite dimensional real vector space, define $\mc{O}(V)$ to be the (two-element) set of orientations of $V$. If $\frak{o}\in\mc{O}(V)$, denote the opposite orientation by $-\frak{o}$. If $W$ is another finite dimensional real vector space, define $\mc{O}(V)\tensor\mc{O}(W)$ to be the set of pairs $(\frak{o}_V,\frak{o}_W)\in\mc{O}(V)\times\mc{O}(W)$, modulo the relation $(\frak{o}_V,\frak{o}_W) \sim (-\frak{o}_V,-\frak{o}_W)$. We denote the equivalence class of $(\frak{o}_V,\frak{o}_W)$ by $\frak{o}_V\tensor\frak{o}_W$. There is a canonical bijection
\[
\mc{O}(V\oplus W) = \mc{O}(V) \tensor \mc{O}(W).
\]
Note that switching the order of $V$ and $W$ multiplies this bijection by $(-1)^{\op{dim}(V)\op{dim}(W)}$. More generally, a finite exact sequence of finite dimensional real vector spaces
\begin{equation}
\label{eqn:fesfdvs}
0 \to V_1 \stackrel{f_1}{\to} V_2 \stackrel{f_2}{\to} \cdots \stackrel{f_{k-1}}{\to} V_k \to 0
\end{equation}
induces a canonical element
\begin{equation}
\label{eqn:eso}
\frak{o}(f_1,\ldots,f_{k-1})\in \mc{O}(V_1)\tensor \cdots \tensor \mc{O}(V_k)
\end{equation}
which is invariant under homotopy of exact sequences. If orientations $\frak{o}_i\in\mc{O}(V_i)$ have already been chosen for $i=1,\ldots,k$, we say that the exact sequence \eqref{eqn:fesfdvs} is {\em orientation preserving\/} if the canonical element \eqref{eqn:eso} agrees with $\frak{o}_1\tensor\cdots\tensor\frak{o}_k$.

If $D$ is a real linear Fredholm operator, define
\[
\mc{O}(D) = \mc{O}(\Ker(D))\tensor\mc{O}(\Coker(D)).
\]
We define an {\em orientation\/} of the Fredholm operator $D$ to be an element of the two-element set $\mc{O}(D)$. A homotopy of Fredholm operators from $D$ to $D'$ induces a bijection $\mc{O}(D) \stackrel{\simeq}{\to} \mc{O}(D')$; see the review in \cite[\S9.1]{ht2}.

\paragraph{Introducing the operators.}
We now introduce the kinds of Fredholm operators that we will need to orient.

\begin{definition}
An {\em orientation loop\/} is a pair $\mc{L} = (E,\nabla)$ where:
\begin{itemize}
\item
 $E$ is a rank $n$ Hermitian vector bundle over $S^1$.
\item
$\nabla$ is a Hermitian connection on $E$.
\end{itemize}
Associated to the orientation loop $\mc{L}$ is a differential operator
\[
A_{\mc{L}} = i\nabla_t: C^\infty(S^1,E) \longrightarrow C^\infty(S^1,E),
\]
where $\nabla_t$ denotes the covariant derivative along $S^1$.
We say that $\mc{L}$ is {\em nondegenerate\/} if $\op{Ker}(A_{\mc{L}})=\{0\}$. Also, we define $\delta_+(\mc{L})>0$ to be the smallest positive eigenvalue of the operator $A_{\mc{L}}$, and $\delta_-(\mc{L})<0$ to be the largest negative eigenvalue.
\end{definition}

\begin{definition}
An {\em orientation surface\/} is a quadruple $\mc{C}=(C,E,\mathscr{L}^+,\mathscr{L}^-)$ where:
\begin{itemize}
\item
$C$ is a (possibly disconnected) punctured compact Riemann surface with $k+l\ge 0$ ordered punctures (which we will regard as ends), of which the first $k$ are designated ``positive'' and the last $l$ are designated ``negative''. Each positive end is conformally identified with $[0,\infty)\times S^1$, and each negative end is conformally identified with $(-\infty,0]\times S^1$. On each end, denote the $[0,\infty)$ or $(-\infty,0]$ coordinate by $s$ and the $S^1$ coordinate by $t$.
\item
$\mathscr{L}^+$ is a list of $k$ orientation loops $\mc{L}^+_j=(E^+_j,\nabla^+_j)$, and $\mathscr{L}^-$ is a list of $l$ orientation loops $\mc{L}^-_j=(E^-_j,\nabla^-_j)$.
\item
$E$ is a rank $n$ Hermitian vector bundle over $C$. An identification is fixed between the restriction of $E$ to the $j^{th}$ positive end of $C$ and the pullback of the bundle $E^+_j$ over $S^1$. Likewise, the restriction of $E$ to the $j^{th}$ negative end of $C$ is identified with the pullback of $E^-_j$.
\end{itemize}
\end{definition}

In this paper we will only need to consider the cases where $C$ is the plane, cylinder, or sphere, but one can also consider more general Riemann surfaces as in \cite{bm,ht2}.

\begin{definition}
\label{def:osdo}
For an orientation surface $\mc{C}=(C,E,\mathscr{L}^+,\mathscr{L}^-)$ as above, define $\mc{D}(\mc{C})$ to be the set of differential operators
\[
D: C^\infty(E) \longrightarrow C^\infty(T^{0,1}C\tensor E)
\]
such that:
\begin{itemize}
\item
In a complex local coordinate on $C$ and a local trivialization of $E$, the operator $D$ equals $\dbar$ plus a zeroth order term.
\item
On the $j^{th}$ positive end of $C$, write
\[
D = \frac{1}{2}(ds-idt) \tensor (\partial_s + i\nabla_t + M_j(s,t))
\]
where $M_j(s,t)$ is an endomorphism of the real vector space $(E^+_j)_t$. Then
\[
\lim_{s \to\infty}M_j(s,\cdot) = 0
\]
in the sense of \cite[\S2]{fh}. Likewise for the negative ends with $s\to-\infty$.
\end{itemize}
\end{definition}

Given an orientation surface $\mc{C}=(C,E,\mathscr{L}^+,\mathscr{L}^-)$, let $\delta=(\delta^+_1,\ldots,\delta^+_k,\delta^-_1,\ldots,\delta^-_l)$ be a tuple of real numbers such that $\delta^+_j$ is not an eigenvalue of $A_{\mc{L}^+_j}$ for any $j=1,\ldots,k$, and $\delta^-_j$ is not an eigenvalue of $-A_{\mc{L}^-_j}$ for any $j=1,\ldots,l$. Then it a standard fact that any operator $D\in\mc{D}(\mc{C})$ extends to a Fredholm operator
\begin{equation}
\label{eqn:Dextended}
D: L^{2,\delta}_1(E) \longrightarrow L^{2,\delta}(T^{0,1}C\tensor E).
\end{equation}
Here $L^{2,\delta}_1$ denotes the exponentially weighted Sobolev space consisting of sections $\psi$ such that $\beta\psi\in L^2_1$, where $\beta:C\to\R$ is a positive smooth function such that on the $j^{th}$ positive end we have $\beta=e^{\delta^+_js}$ for $s$ large, and on the $j^{th}$ negative end we have $\beta=e^{\delta^-_j|s|}$for $|s|$ large.

Assume further that $0\le \delta^+_j< \delta_+(\mc{L}^+_j)$ for each $j=1,\ldots,k$, and $0\le \delta^-_j<-\delta_-(\mc{L}^-_j)$ for each $j=1,\ldots,l$. (Note here that $\delta^\pm_j=0$ is allowed only if the orientation loop $\mc{L}^\pm_j$ is nondegenerate.) We will call such a choice of exponential weights {\em admissible\/}.
Then the set of orientations of the Fredholm operator \eqref{eqn:Dextended} does not depend\footnote{One could also use Sobolev spaces $L^{p,\delta}_k$; the choice of $p$ and $k$ is immaterial for orientations as it does not affect the kernel or (up to canonical isomorphism) the cokernel of the operator.} on $\delta$, and in fact depends only on the orientation surface $\mc{C}$; see the review in \cite[\S9.2]{ht2}. Denote this set of orientations by $\mc{O}(\mc{C})$.

\begin{remark}
\label{rem:conjugate}
It is sometimes useful to eliminate the exponential weights as follows. If $\beta:C\to\R$ is a function as in the definition of $L^{2,\delta}_1$, then
the operator \eqref{eqn:Dextended} is conjugate to the operator
\[
\beta D \beta^{-1}: L^2_1(E) \longrightarrow L^2(T^{0,1}C\tensor E).
\]
This conjugate operator is itself the extension of a differential operator
\[
\beta D \beta^{-1} \in \mc{D}(\mc{C}_\delta),
\]
where $\mc{C}_\delta$ is an orientation surface obtained by shifting the orientation loops by $\delta$. Namely,
\[
\mc{C}_\delta = (C,E,\{(E^+_j,\nabla^+_j+i\delta^+_j)\},\{(E^-_j,\nabla^-_j -i\delta^-_j)\}).
\]
The orientation surface $\mc{C}_\delta$ has all orientation loops at its ends nondegenerate, and the above conjugation operation induces a canonical bijection $\mc{O}(\mc{C}_\delta) \simeq \mc{O}(\mc{C})$.
\end{remark}

\paragraph{Gluing orientations.}
The key operation is now to ``glue'' orientations. Let $\mc{C}'=(C',E',(\mathscr{L}')^+,(\mathscr{L}')^-)$ be another orientation surface. Suppose that $\mc{L}^-_j = (\mc{L}')^+_j$ for $j=1,\ldots,l$. We can then glue the first $l$ negative ends of $\mc{C}$ to the first $l$ positive ends of $\mc{C}'$ to obtain a new orientation surface, which we denote by $\mc{C}\#_l\mc{C}'$. This glued orientation surface also depends on a parameter $\mathscr{R}>0$, which determines where to cut off the ends of $\mc{C}$ and $\mc{C}'$ before gluing; we omit this parameter from the notation.

Given operators $D\in\mc{D}(\mc{C})$ and $D'\in\mc{D}(\mc{C}')$, we can use cutoff functions with derivatives of order $\mathscr{R}^{-1}$ to patch them to a ``glued'' operator
\[
D\# D' \in \mc{D}(\mc{C}\#_l\mc{C}').
\]

\begin{lemma}
\label{lem:glulin}
Fix operators $D,D'$ as above. Suppose that for each $j=1,\ldots,l$, the orientation loop $\mc{L}^-_j=(\mc{L}')^+_j$ is nondegenerate. Choose admissible exponential weights $\delta$ and $\delta'$ such that for each $j=1,\ldots,l$ we have $\delta^-_j=(\delta')^+_j=0$. Then if the gluing parameter $\mathscr{R}$ is sufficiently large, there is up to homotopy a canonical\footnote{The exact sequence just depends on the choice of gluing parameter and cutoff functions; different choices will give homotopic exact sequences.} exact sequence
\[
0\to\Ker(D\#D') \stackrel{f}{\to} \Ker(D)\oplus \Ker(D') \stackrel{g}{\to} \Coker(D)\oplus\Coker(D') \stackrel{h}{\to} \Coker(D\#D')\to 0.
\]
\end{lemma}

\begin{proof}
By Remark~\ref{rem:conjugate}, we can assume without loss of generality that all the orientation loops are nondegenerate and there are no exponential weights. The lemma in this case now follows from \cite[Prop.\ 9.3]{ht2}.
\end{proof}

\begin{remark}
The exact sequence in Lemma~\ref{lem:glulin} describes a linear version of obstruction bundle gluing, in which one attempts to glue elements of $\Ker(D)$ and $\Ker(D')$ to obtain an element of $\Ker(D\#D')$. The obstruction to this gluing is an element of $\Coker(D)\oplus\Coker(D')$, specified by the map $g$ in the exact sequence.
\end{remark}

As in \eqref{eqn:eso}, the above exact sequence determines a bijection
\begin{equation}
\label{eqn:poi}
\frak{o}(f,g,h): \mc{O}(D)\tensor\mc{O}(D')\stackrel{\simeq}{\longrightarrow} \mc{O}(D\#D').
\end{equation}

\begin{lemma}
\label{lem:gluor}
Under the assumptions of Lemma~\ref{lem:glulin}, the bijection
\begin{equation}
\label{eqn:iho}
(-1)^{\dim(\Ker(D'))\dim(\Coker(D))} \frak{o}(f,g,h): \mc{O}(D)\tensor\mc{O}(D')\stackrel{\simeq}{\longrightarrow} \mc{O}(D\#D')
\end{equation}
is invariant under homotopy of $D$ and $D'$ and deformation of the gluing parameter, and thus determines a canonical bijection
\begin{equation}
\label{eqn:gluor}
\mc{O}(\mc{C})\tensor\mc{O}(\mc{C}') \stackrel{\simeq}{\longrightarrow} \mc{O}(\mc{C} \#_l \mc{C}').
\end{equation}
\end{lemma}

\begin{proof}
This follows the proof\footnote{The statement of \cite[Lem.\ 9.6]{ht2} needs to be corrected by including a sign as in \eqref{eqn:iho}. The last sentence of Step 1 of the proof also needs to be corrected accordingly. Fortunately, the missing sign in \cite[Lem.\ 9.6]{ht2} has no effect on the use of this lemma in the rest of \cite{ht2}.} of \cite[Lem.\ 9.6]{ht2}.
\end{proof}

\begin{remark}
The gluing of orientations in \eqref{eqn:gluor} depends on our choice of convention for writing $D$ first and $D'$ second in the direct sums in Lemma~\ref{lem:glulin}. Our convention here agrees with \cite{bm} and disagrees with \cite{ht2}. Switching this convention would multiply the map \eqref{eqn:gluor} by $(-1)^{\op{ind}(D)\op{ind}(D')}$.
\end{remark}

\begin{lemma}
\label{lem:gluass}
The gluing of orientations in \eqref{eqn:gluor} is associative. That is, if $\mc{C}''$ is another orientation surface, for which the orientation loops of the first $l'$ positive ends are nondegenerate and agree with those of the first $l'$ negative ends of $\mc{C}'$, then we have a well-defined bijection
\[
\mc{O}(\mc{C}) \tensor \mc{O}(\mc{C}') \tensor \mc{O}(\mc{C}'') \stackrel{\simeq}{\longrightarrow} \mc{O}(\mc{C}\#_l\mc{C}'\#_{l'}\mc{C}'').
\]
\end{lemma}

\begin{proof}
This follows the proof of \cite[Lem.\ 9.7]{ht2} (with sign corrections as for the proof of \cite[Lem.\ 9.6]{ht2}).
\end{proof}

In Lemma~\ref{lem:gluor} we can weaken (but not completely drop) the assumption that the orientation loops $\mc{L}^-_j=(\mc{L}')^+_j$ are nondegenerate. Namely:

\begin{definition}
An orientation loop $\mc{L}=(E,\nabla)$ is {\em weakly nondegenerate\/} if $\Ker(A_{\mc{L}})$ is a complex vector space, with respect to multiplication by $i$ on $C^\infty(S^1,E)$. (This includes the case when $\Ker(A_{\mc{L}})=\{0\}$ and $\mc{L}$ is nondegenerate.)
\end{definition}

\begin{lemma}
\label{lem:scb}
Let $\mc{C}=(C,E,\mathscr{L}^+,\mathscr{L}^-)$ and $\mc{C}'=(C',E',(\mathscr{L}')^+,(\mathscr{L}')^-)$ be orientation surfaces such that $\mc{L}^-_j=(\mc{L}')^+_j$ for $j=1,\ldots,l$. Assume that $\mc{L}^-_j$ is weakly nondegenerate for each $j=1,\ldots,l$. Then there is still a canonical bijection as in \eqref{eqn:gluor} which is associative in the sense of Lemma~\ref{lem:gluass}.
\end{lemma}

To prove Lemma~\ref{lem:scb} we will need the following:

\begin{lemma}
\label{lem:wnco}
Suppose $\mc{L}=(E,\nabla)$ is weakly nondegenerate. Let $\delta_+\in(0,-\delta_-(\mc{L}))$ and let $\delta_-\in(0,\delta_+(\mc{L}))$. Let
\[
\mc{C}=(\R\times S^1,E,(E,\nabla-i\delta_+),(E,\nabla+i\delta_-)).
\]
Then there is a canonical orientation in $\mc{O}(\mc{C})$.
\end{lemma}

\begin{proof}
There is a distinguished class of operators $D\in\mc{D}(\mc{C})$. Namely, write $A=i\nabla_t$. Choose a smooth function $\varphi:\R\to\R$ such that $\varphi(s)=\delta_+s$ for $s>>0$ and $\varphi(s)=-\delta_-s$ for $s<<0$. We can then take
\[
D = \frac{1}{2}(\partial_s + A +\varphi(s)).
\]
(We omit $ds-idt$ from the notation here and below.) Of course $D$ depends on the choice of function $\varphi$, but different choices will be canonically homotopic and thus will have canonically isomorphic sets of orientations.

The operator
\[
D: L^2_1(E)\longrightarrow L^2(T^{0,1}(\R\times S^1)\tensor E)
\]
is conjugate to the operator
\[
e^{\varphi(s)} D e^{-\varphi(s)} = \frac{1}{2}(\partial_s + A)
\]
acting on Sobolev spaces with exponential weights $-\delta_+$ on the positive end and $-\delta_-$ on the negative end. Any element of the kernel of this operator is a linear combination of functions of the form
\[
\psi(s,t)=e^{-\lambda s}\phi(t)
\]
where $\phi$ is an eigenfunction of $A$ with eigenvalue $\lambda$. Because of our choice of exponential weights, $\psi$ is in the domain of the operator only when $\lambda=0$. Thus $\Ker(D)$ is canonically identified with $\Ker(A)$.
A similar argument shows that $\Coker(D)=\{0\}$. By hypothesis, $\Ker(A)$ is a complex vector space, so it has a canonical orientation. We conclude that $D$ has a canonical orientation, and thus there is a canonical orientation in $\mc{O}(\mc{C})$.
\end{proof}

\begin{proof}[Proof of Lemma~\ref{lem:scb}.]
By Remark~\ref{rem:conjugate}, we can assume without loss of generality that each $\mc{L}^+_j$ is nondegenerate, each $(\mc{L}')^-_j$ is nondegenerate, and $\mc{L}^+_j$ is nondegenerate when $j>l$. Choose admissible exponential weights $\delta$ and $\delta'$ such that all weights are zero except possibly for $\delta^-_1,\ldots,\delta^-_l$ and $(\delta')^+_1,\ldots,(\delta')^+_l$. (The weights $\delta^-_j$ and $(\delta')^+_j$ must be positive when $\mc{L}^-_j=(\mc{L}')^+_j$ is degenerate.)

To reduce to Lemma~\ref{lem:gluor}, we would like to replace $\mc{C}$ and $\mc{C}'$ by $\mc{C}_\delta$ and $\mc{C}'_{\delta'}$ as in Remark~\ref{rem:conjugate}. However this does not work directly because the first $l$ negative ends of $\mc{C}_\delta$ do not agree with the first $l$ negative ends of $\mc{C}'_{\delta'}$ (as the orientation loops are shifted in opposite directions), so these two orientation surfaces cannot be glued. The trick is to glue in a third orientation surface and write
\[
\mc{C}\#_l \mc{C}' \simeq \mc{C}_\delta \#_l \mc{C}'' \#_l \mc{C'}_{\delta'}.
\]
Here $\mc{C}''$ consists of $l$ cylinders. On the $j^{th}$ cylinder, the bundle $E$ is the pullback of the bundle $E^-_j = (E')^+_j$. The notation `$\simeq$' means that the orientation surfaces are homotopic so that there is a canonical bijection between the orientations. Now by Lemmas~\ref{lem:gluor} and \ref{lem:gluass}, we have a canonical bijection
\begin{equation}
\label{eqn:occc}
\begin{split}
\mc{O}(\mc{C}\#_l\mc{C}') \stackrel{\simeq}{\longrightarrow} & \mc{O}(\mc{C}_\delta)\tensor \mc{O}(\mc{C}'') \tensor \mc{O}(\mc{C}'_\delta)\\
& = \mc{O}(\mc{C})\tensor \mc{O}(\mc{C}'') \tensor \mc{O}(\mc{C}').
\end{split}
\end{equation}
By Lemma~\ref{lem:wnco}, there is a canonical orientation in $\mc{O}(\mc{C}'')$.
Putting this canonical orientation into \eqref{eqn:occc} gives the desired canonical bijection \eqref{eqn:gluor}. This is associative by Lemma~\ref{lem:gluass}.
\end{proof}

\paragraph{Choosing orientations.}

We now explain what choices are needed to orient the operators of interest.

\begin{lemma}
\label{lem:comcan}
Let $\mc{C}$ be an orientation surface with no ends. Then there is a canonical orientation in $\mc{O}(\mc{C})$.
\end{lemma}

\begin{proof}
If $D\in\mc{D}(\mc{C})$, then $D$ is homotopic to a complex linear operator $D'$, which has a canonical orientation. This induces an orientation of $D$, which does not depend on the choice of homotopy or on the choice of complex linear operator $D'$, because the spaces of real-linear and complex-linear operators in $\mc{D}(\mc{C})$ are both contractible.
\end{proof}

\begin{definition}
\label{def:ownol}
Let $\mc{L}=(E,\nabla)$ be a weakly nondegenerate orientation loop. 
\begin{itemize}
	\item Let $\mc{C}_+=(C_+,E_+,\emptyset,(\mc{L}))$ be an orientation surface such that $C_+$ is a plane with one negative end, on which $E_+$ is pulled back from $E$. Define $\mc{O}_+(\mc{L}) = \mc{O}(\mc{C}_+)$.
	\item Let $\mc{C}_- = (C_-,E_-,(\mc{L}),\emptyset)$ be an orientation surface such that $C_-$ is a plane with one positive end, on which $E_-$ is pulled back from $E$. Define $\mc{O}_-(\mc{L}) = \mc{O}(\mc{C}_-)$.
\end{itemize}
\end{definition}

\begin{lemma}
\label{lem:ownol}
Let $\mc{L}$ be a weakly nondegenerate orientation loop. Then:
\begin{description}
	\item{(a)} The sets $\mc{O}_+(\mc{L})$ and $\mc{O}_-(\mc{L})$ do not depend on the choices of orientation surfaces $\mc{C}_+$ and $\mc{C}_-$ in Definition~\ref{def:ownol}.
	\item{(b)} There is a canonical bijection $\mc{O}_+(\mc{L}) = \mc{O}_-(\mc{L})$.
\end{description}
\end{lemma}

\begin{proof}
Let $\mc{C}_+$ and $\mc{C}_-$ be orientation surfaces as in Definition~\ref{def:ownol}. By Lemmas~\ref{lem:scb} and \ref{lem:comcan}, there is a canonical bijection
\begin{equation}
\label{eqn:oc+c-}
\mc{O}(\mc{C}_+) \simeq \mc{O}(\mc{C}_-).
\end{equation}
It follows that if we fix a choice of $\mc{C}_-$, then the sets of orientations $\mc{O}(\mc{C}_+)$ for different choices of $\mc{C}_+$ are identified with each other. To prove that this identification does not depend on the choice of $\mc{C}_-$, one can use the argument in \cite[Lem.\ 2.46]{p} and \cite[Prop.\ 2.8]{solomon}. Thus $\mc{O}_+(\mc{L})$ is well-defined, and likewise $\mc{O}_-(\mc{L})$ is defined. Moreover, the bijection \eqref{eqn:oc+c-} descends to a well-defined bijection $\mc{O}_+(\mc{L}) = \mc{O}_-(\mc{L})$. 
\end{proof}

In light of Lemma~\ref{lem:ownol}, if $\mc{L}$ is a weakly nondegenerate orientation loop, we define $\mc{O}(\mc{L}) = \mc{O}_\pm(\mc{L})$. We can now prove the conclusion of this subsection:

\begin{proposition}
\label{prop:orcy}
Let $\mc{L}_+$, $\mc{L}_0$, and $\mc{L}_-$ be weakly nondegenerate orientation loops.
\begin{description}
\item{(a)}
If $\mc{C}=(C,E,(\mc{L}_+),(\mc{L}_-))$ is an orientation surface in which $C$ is a cylinder with one positive and one negative end, then there is a canonical bijection
\[
\mc{O}(\mc{C}) \simeq \mc{O}(\mc{L}_+) \tensor \mc{O}(\mc{L}_-).
\]
\item{(b)}
Let $\mc{C}_+=(C_+,E_+,(\mc{L}_+),(\mc{L}_0))$ and $\mc{C}_-=(C_-,E_-,(\mc{L}_0),(\mc{L}_-))$ be cylindrical orientation surfaces as above, so that by part (a) we have canonical bijections
\[
\begin{split}
\mc{O}(\mc{C}_+) & \simeq \mc{O}(\mc{L}_+) \tensor \mc{O}(\mc{L}_0),\\
\mc{O}(\mc{C}_-) & \simeq \mc{O}(\mc{L}_0) \tensor \mc{O}(\mc{L}_-).
\end{split}
\]
Then the identification
\[
\mc{O}(\mc{C}) \simeq \mc{O}(\mc{C}_+) \tensor \mc{O}(\mc{C}_-)
\]
given by the above three equations agrees with the canonical identification given by Lemma~\ref{lem:scb}.
\end{description}
\end{proposition}

\begin{proof}
(a) Let $\mc{C}_+$ be as in the definition of $\mc{O}_+(\mc{L}_+)$, and let $\mc{C}_-$ be as in the definition of $\mc{O}_-(\mc{L}_+)$. Then $\mc{C}_+\#_1\mc{C}\#_1\mc{C}_-$ is an orientation surface in which the underlying surface is a sphere, and it has a canonical orientation by Lemma~\ref{lem:comcan}. Then by Lemma~\ref{lem:scb} there is a canonical bijection
\[
\mc{O}(\mc{C}) \simeq \mc{O}(\mc{C}_+)\tensor \mc{O}(\mc{C}_-).
\] 
By Lemma~\ref{lem:ownol}, this implies part (a). Part (b) then follows from the definitions.
\end{proof}

\subsection{The tangent space to the moduli space}
\label{sec:tms}

Let $(Y^{2n-1},\lambda)$ be a closed nondegenerate contact manifold. (For this discussion we do not need to assume that $\lambda$ is hypertight.) Let $\J=\{J_t\}_{t\in S^1}$ be a generic $S^1$-family of $\lambda$-compatible almost complex structures on $\R\times Y$ as in \S\ref{section:plainmoduli}. To prepare to orient the moduli space $\mc{M}^\J(\gamma_+,\gamma_-)$, we now discuss its tangent space.

If $\gamma_+$ and $\gamma_-$ are distinct Reeb orbits, let $\widetilde{\widetilde{\mc{M}}}^{\J}(\gamma_+,\gamma_-)$ denote the space of maps $u:\R\times S^1\to\R\times Y$ satisfying the equations \eqref{eqn:cr}--\eqref{eqn:param}. Here we do not mod out by $\R$ translation in the domain, so that
\[
\widetilde{\mc{M}}^{\J}(\gamma_+,\gamma_-) = \widetilde{\widetilde{\mc{M}}}^{\J}(\gamma_+,\gamma_-)/\R.
\]

If $u\in\widetilde{\widetilde{\mc{M}}}^{\J}(\gamma_+,\gamma_-)$, then the derivative of the Cauchy-Riemann equation \eqref{eqn:cr} defines a linearized operator
\begin{equation}
\label{eqn:lincr}
D_u: L^{2,\delta}_1(u^*T(\R\times Y)) \longrightarrow L^{2,\delta}(u^*T(\R\times Y)).
\end{equation}
Here $\delta>0$ is a small exponential weight, smaller than the smallest positive eigenvalue of the asymptotic operator associated to $\gamma_+$ and $\J$, or minus the largest negative eigenvalue of the asymptotic operator associated to $\gamma_-$ and $\J$, see below. Thus the kernel of $D_u$ consists of infinitesimal deformations of $u$ which suitably decay on the ends of $u$.

\begin{definition}
\label{def:taupm}
We define $\R$-linear maps
\begin{equation}
\label{eqn:taumaps}
\tau_+, \tau_-: \bbC \longrightarrow \Coker(D_u)
\end{equation}
as follows. Recall that $r$ denotes the $\R$ coordinate on $\R\times Y$, while $R$ denotes the Reeb vector field on $Y$. Let $\psi_1^+$ be a smooth section of $u^*T(\R\times Y)$ with $\psi_1^+=\partial_r$ for $s>>0$ and $\psi_1^+=0$ for $s<<0$. Let $\psi_2^+$ be a smooth section of $u^*T(\R\times Y)$ with $\psi_2^+=R$ for $s>>0$ and $\psi_2^+=0$ for $s<<0$. Choose $\psi_1^-$ and $\psi_2^-$ analogously with the sign of $s$ switched. If $a,b\in\R$, we define
\[
\tau_+(a+bi) = \pi_{\Coker(D_u)}\text{``$D_u(a\psi^+_1 + b\psi^+_2)$''}.
\]
There are quotation marks on the right hand side because $a\psi^+_1 + b\psi^+_2$ is not in the domain of $D_u$ as in \eqref{eqn:lincr}.  To interpret the right hand side, if we regard $D_u$ as a differential operator on smooth sections, then $D_u(a\psi^+_1+b\psi^+_2)$ is a well-defined smooth section $\eta\in L^{2,\delta}(u^*T(\R\times Y))$, by standard asymptotics of holomorphic curves. We define $\tau_+(z)$ to be the projection of $\eta$ to $\Coker(D_u)$. Note that this does not depend on the choice of $\psi^+_1$ and $\psi^+_2$, because different choices will differ by a compactly supported section which is in the domain of $D_u$ as in \eqref{eqn:lincr}. We likewise define
\[
\tau_-(a+bi) = \pi_{\Coker(D_u)}\text{``$D_u(a\psi^-_1 + b\psi^-_2)$''}.
\]
\end{definition}

\begin{proposition}
\label{prop:tanx}
If $\J$ is generic, then $\widetilde{\widetilde{\mc{M}}}^{\J}(\gamma_+,\gamma_-)$ is naturally a smooth manifold, and there is a canonical exact sequence
\begin{equation}
\label{eqn:ces1}
0 \longrightarrow \Ker(D_u) \longrightarrow T_u\widetilde{\widetilde{\mc{M}}}^\J(\gamma_+,\gamma_-) \stackrel{(\sigma_+,\sigma_-)}{\longrightarrow} \bbC\oplus\bbC \stackrel{(\tau_+,\tau_-)}{\longrightarrow} \Coker(D_u) \longrightarrow 0.
\end{equation}
\end{proposition}

\begin{remark}
The proof of Proposition~\ref{prop:tanx} will show that the maps $\sigma_\pm$ are described as follows: If $\{u_\tau\}_{\tau\in\R}$ is a smooth family in $\widetilde{\widetilde{\mc{M}}}^{\J}(\gamma_+,\gamma_-)$, with $u_0=u$, and if we write $\dot{u}=\frac{d}{d\tau}\big|_{\tau=0}u_\tau\in u^*T(\R\times Y)$, then
\[
\sigma_\pm\left(\dot{u}\right) = \lim_{s\to\pm\infty}\left(dr\left(\dot{u}\right) + i\lambda\left(\dot{u}\right)\right).
\]
\end{remark}

\begin{proof}[Proof of Proposition~\ref{prop:tanx}.]
Fix $p>2$. Note that the kernel and cokernel of $D_u$ are unchanged if one replaces $L^{2,\delta}_1$ and $L^2_1$ in \eqref{eqn:lincr} by $L^{p,\delta}_1$ and $L^{p,\delta}$, and we will do this below.

 Let $\mc{B}(\gamma_+,\gamma_-)$ denote the set of continuous maps $u:\R\times S^1\to\R\times Y$ satisfying the asymptotic conditions \eqref{eqn:pirlim} and \eqref{eqn:param} such that for $|s|$ large, $u$ is the exponential of an $L^{p,\delta}_1$ section of the normal bundle to $\R\times\gamma$. We define a map
\begin{equation}
\label{eqn:lpdelta}
L^{p,\delta}_1(u^*(T(\R\times Y))) \oplus \bbC \oplus \bbC \longrightarrow \mc{B}(\gamma_+,\gamma_-)
\end{equation}
by
\[
(\psi,a_++b_+i,a_-+b_-i) \longmapsto \exp_u(\psi+a_+\psi^+_1+b_+\psi^+_2+a_-\psi^-_1+b_-\psi^-_2),
\]
where $\psi^\pm_1$ and $\psi^\pm_2$ are chosen as in the definition of $\tau_\pm$. Any element of $\widetilde{\widetilde{\mc{M}}}^{\J}(\gamma_+,\gamma_-)$ near $u$ is in the image of the map \eqref{eqn:lpdelta}. Thus a neighborhood of $u$ in $\widetilde{\widetilde{\mc{M}}}^{\J}(\gamma_+,\gamma_-)$ can be described as the zero set of a section of a Banach space bundle over the left side of \eqref{eqn:lpdelta}, whose fiber over $u$ is $L^{p,\delta}(u^*T(\R\times Y))$. The derivative of this section at $u$ is the map
\begin{equation}
\label{eqn:dersec}
\begin{split}
\Phi: L^{p,\delta}_1(u^*(T(\R\times Y))) \oplus \bbC \oplus \bbC &\longrightarrow L^{p,\delta}(u^*T(\R\times Y)),\\
(\psi,a_++b_+i,a_-+b_-i) & \longmapsto \text{``$D_u(\psi + a_+\psi^+_1 + b_+\psi^+_2 + a_-\psi^-_1 + b_-\psi^-_2)$''}.
\end{split}
\end{equation}
Here, as in Definition~\ref{def:taupm}, although $\psi^\pm_1$ and $\psi^\pm_2$ are not in the domain of $D_u$, the map \eqref{eqn:dersec} is still well-defined by regarding $D_u$ as a differential operator.

A standard transversality argument, which proves Proposition~\ref{prop:transversality}(a), shows that if $\J$ is generic, then for each $u$ the operator $\Phi$ in \eqref{eqn:dersec} is surjective, which in turn implies that $\widetilde{\widetilde{\mc{M}}}^\J(\gamma_+,\gamma_-)$ is naturally a smooth manifold near $u$ whose tangent space at $u$ is $\Ker(\Phi)$. 

To complete the proof of the proposition, we now define a sequence
\begin{equation}
\label{eqn:ces2}
0 \longrightarrow \Ker(D_u) \longrightarrow \Ker(\Phi) \stackrel{(\sigma_+,\sigma_-)}{\longrightarrow} \bbC\oplus\bbC \stackrel{(\tau_+,\tau_-)}{\longrightarrow} \Coker(D_u) \longrightarrow 0
\end{equation}
and show that it is exact if $\Phi$ is surjective. We define the first arrow to be the inclusion $\psi\mapsto (\psi,0,0)$, and we define the second arrow by
\[
\sigma_\pm(\psi,z_+,z_-) = z_\pm.
\]
By the definitions of $\tau_\pm$ and $\Phi$ we have
\[
(\pi_{\Coker(D_u)}\circ \Phi)(\psi,z_+,z_-) = \tau_+(z_+) + \tau_-(z_-).
\]
It follows that the sequence \eqref{eqn:ces2} is exact, except possibly for surjectivity of the last arrow. If $\Phi$ is surjective, then the last arrow is also surjective.
\end{proof}

For reference elsewhere, we now clarify the criteria for transversality.

\begin{definition}
\label{def:regular}
An element $u$ of $\widetilde{\widetilde{\mc{M}}}^{\J}(\gamma_+,\gamma_-)$, or its equivalence class in $\mc{M}^{\J}(\gamma_+,\gamma_-)$, is {\em regular\/} if the operator \eqref{eqn:dersec} above is surjective. Note that while the operator \eqref{eqn:dersec} itself depends on the choices of $p$, $\psi_1^\pm$, and $\psi_2^\pm$, its surjectivity does not.
\end{definition}

\begin{example}
\label{ex:trivcyl}
While we usually assume that $\gamma_+\neq\gamma_-$, it is instructive to consider the case where $\gamma_+=\gamma_-=\gamma$ and $u$ is a ``trivial cylinder''; that is, $\pi_R\circ u$ is a linear function of $s$ and does not depend on $t$, while $\pi_Y\circ u$ does not depend on $s$ and as a function of $t$ is a parametrization of $\gamma$.

In this case the operator $D_u$ has the form $\partial_s + J_t\nabla_t$ where $\nabla$ is the connection on $\gamma^*T(\R\times Y)$ defined in Definition~\ref{def:col} below. Direct calculation shows that $D_u$ is injective with two-dimensional cokernel. The cokernel is spanned by the images of $\psi_1^+$ and $\psi_2^+$ under $D_u$ (regarded as a differential operator). One can choose $\psi_1^\pm$ and $\psi_2^\pm$ so that $\psi_1^++\psi_1^-\equiv \partial_r$ and $\psi_2^++\psi_2^-\equiv R$. Then the operator \eqref{eqn:dersec} is surjective with a two-dimensional kernel consisting of triples $(0,z,z)$ for $z\in{\mathbb C}$.

In particular, the moduli space $\widetilde{\widetilde{\M}}^\J(\gamma,\gamma)$ is a two-dimensional manifold cut out transversely, consisting of compositions of $u$ with automorphisms of $\R\times S^1$. However $\M^\J(\gamma,\gamma)$ is not cut out transversely, because the $\R$ action on $\widetilde{\widetilde{\M}}^\J(\gamma,\gamma)$ by translations of the target is not free. In particular $\M^J(\gamma,\gamma)=\M^\J_0(\gamma,\gamma)$ is one dimensional, although its expected dimension is zero. (We usually assume that $\gamma_+\neq\gamma_-$ in order to ensure that this $\R$ action is free.)
\end{example}

\begin{remark}
\label{rem:altreg}
In the special case when $\J$ does not depend on $S^1$, that is $\J=\{J_t\}$ where $J_t\equiv J$, and when $u$ is an immersion, there is an alternate notion of regularity used in \cite[\S4]{dc}. In this case let $N$ denote the normal bundle to $u$ in $\R\times Y$. Then deformations of $u$, regarded as an immersed submanifold of $\R\times Y$, are equivalent to sections of $N$, and there is a ``normal deformation operator''
\[
D_N: L^{p,\delta}_1(N) \longrightarrow L^{p,\delta}(N)
\]
which measures the failure of these deformations to be $J$-holomorphic. We claim now that $u$ is regular as in Definition~\ref{def:regular} if $D_N$ is surjective.

To see this, note that with respect to the direct sum decomposition
\[
u^*T(\R\times Y) = T(\R\times S^1) \oplus N,
\]
and the corresponding direct sum decompositions of spaces of sections, the linearized operator $D_u$ in \eqref{eqn:lincr} can be written as a triangular block matrix
\[
D_u = \begin{pmatrix} D_T & \mbox{\Biohazard} \\ 0 & D_N \end{pmatrix}.
\]
Here $D_T$ is the restriction of $D_u$ to $L^{p,\delta}_1(T(\R\times S^1))$, which maps to $L^{p,\delta}(T(\R\times S^1))$. Similarly to Example~\ref{ex:trivcyl}, the operator $D_T$ is injective with a two-dimensional cokernel. Thus surjectivity of $D_N$ implies that $D_u$ has a two-dimensional cokernel, and this cokernel is covered by the images of $\psi_1^+$ and $\psi_2^+$ under $D_u$ (regarded as a differential operator), so that \eqref{eqn:dersec} is surjective.
\end{remark}

\subsection{Orienting the moduli space}
\label{sec:oms}

Let $\J=\{J_t\}_{t\in S^1}$ be a generic $S^1$-family of $\lambda$-compatible almost complex structures. We are now ready to orient the moduli space $\M^\J(\gamma_+,\gamma_-)$.

\begin{definition}
\label{def:col}
Let $\J=\{J_t\}_{t\in S^1}$ be any $S^1$-family of $\lambda$-compatible almost complex structures (not necessarily generic).
Let $\gamma:\R/T\Z\to Y$ be a Reeb orbit, and let $p\in\overline{\gamma}=\op{im}(\gamma)$. We define an orientation loop $\mc{L}_{\gamma,p,\J}=(S,\nabla)$ as follows.

First fix $t_0\in\R/T\Z$ with $\gamma(t_0)=p$, and define a map $\widehat{\gamma}:S^1=\R/\Z\to Y$ by
\[
\widehat{\gamma}(t) = \gamma(t_0+Tt).
\]
\begin{itemize}
	\item
Define $E=\widehat{\gamma}^*T(\R\times Y)$. We have a direct sum decomposition
\begin{equation}
\label{eqn:edsd}
E = \widehat{\gamma}^*\xi \oplus \bbC
\end{equation}
where we identify $a+bi\in\bbC$ with $a\partial_r + bR\in T(\R\times Y)$.
Then $E$ is a Hermitian vector bundle over $S^1$ with the almost complex structure and metric on $\xi_{\widehat{\gamma}(t)}$ determined by $J_t$ and $d\lambda$.
\item
The linearized Reeb flow determines a connection $\nabla^R$ on $\widehat{\gamma}^*\xi$. With respect to the direct sum decomposition \eqref{eqn:edsd}, define $\nabla = \nabla^R\oplus \nabla^0$, where $\nabla^0$ denotes the trivial connection on $S^1 \times \bbC$.
\end{itemize}
\end{definition} 

\begin{lemma}
\begin{description}
	\item{(a)}
The orientation loop $\mc{L}_{\gamma,p,\J}$ is weakly nondegenerate.
	\item{(b)}
If $\J$ and $\J'$ are two $S^1$-families of $\lambda$-compatible almost complex strutures, then there is a canonical bijection
\[
\mc{O}(\mc{L}_{\gamma,p,\J}) \simeq \mc{O}(\mc{L}_{\gamma,p,\J'}),
\]
so we can denote this set of orientations by $\mc{O}_{\gamma,p}$.
\end{description}
\end{lemma}

\begin{proof}
(a)
Since we are assuming that the Reeb orbit $\gamma$ is nondegenerate, the kernel of the operator $A_{\mc{L}_{\gamma,p,\J}}$ is canonically identified with the $\bbC$ summand in \eqref{eqn:edsd}.

(b)
This holds because the set of $S^1$-families $\J$ of $\lambda$-compatible almost complex structures is contractible, and a homotopy of such $\J$ defines a homotopy of the Fredholm operators used to define the orientation set $\mc{O}_{\mc{L}_{\gamma,p,\J}}$.
\end{proof}

\begin{definition}
\label{def:Ogamma}
If $\gamma$ is a Reeb orbit and $p\in\overline{\gamma}$, define
\[
\mc{O}_\gamma(p) = \mc{O}_{\gamma,p}\tensor\Z.
\]
Here the right hand side denotes the set of pairs $(\frak{o},k)\in\mc{O}_{\gamma,p}\times\Z$ modulo the equivalence relation $(-\frak{o},k)\sim(\frak{o},-k)$.
\end{definition}

Observe that the assignment to $p$ of $\mc{O}_{\gamma}(p)$ defines a local system over $\overline{\gamma}$, since a homotopy of Fredholm operators induces a bijection on orientations.

\begin{proof}[Proof of Proposition~\ref{prop:orientations}.]
Assertion (a) is proved in \cite[\S5]{bm}. The conventions in \cite{bm} are slightly different, but the argument given there is still valid here.

To prove assertion (b), let $u\in\widetilde{\widetilde{\mc{M}}}^\J(\gamma_+,\gamma_-)$, and consider its equivalence class $[u]\in \M^\J(\gamma_+,\gamma_-)$. From Definition~\ref{def:col} we have orientation loops
\[
\mc{L}_\pm=\mc{L}_{\gamma_\pm,e_\pm(u),\J}.
\]
We need to show that there is a canonical bijection between orientations of the tangent space $T_{[u]}\mc{M}^\J(\gamma_+,\gamma_-)$ and the orientation set $\mc{O}(\mc{L}_+)\tensor\mc{O}(\mc{L}_-)$, and we need to check that this bijection depends continuously on $[u]$.

Define an orientation surface
\[
\mc{C}=(\R\times S^1, u^*T(\R\times Y),(\mc{L}_+),(\mc{L}_-)).
\]
The operator $D_u$ in \eqref{eqn:lincr}, regarded as a differential operator, is an element of the set $\mc{D}(\mc{C})$. Note here that we are implicitly trivializing $T^{0,1}(\R\times S^1)$ using $\frac{1}{2}(ds-idt)$, so that the $T^{0,1}$ factor in Definition~\ref{def:osdo} is not needed in \eqref{eqn:lincr}.

By Proposition~\ref{prop:orcy}, there is a canonical bijection
\[
\mc{O}(\mc{L}_+)\tensor\mc{O}(\mc{L}_-) \simeq \mc{O}(D_u).
\]
By Proposition~\ref{prop:tanx}, an orientation of $D_u$ canonically determines an orientation of $T_u\widetilde{\widetilde{M}}^\J(\gamma_+,\gamma_-)$. (Here we use the canonical orientation of ${\mathbb C}\oplus{\mathbb C}$ in the exact sequence \eqref{eqn:ces1}.) The latter determines an orientation of $T_{[u]}\widetilde{\M}^J(\gamma_+,\gamma_-)$ by the ``$\R$-direction first'' convention, where $\R$ acts on $\widetilde{\widetilde{\M}}$ by
\begin{equation}
\label{eqn:ractiondomain}
(r\cdot u)(s,t) = u(r+s,t).
\end{equation}
This in turn determines an orientation of $T_{[u]}\M^\J(\gamma_+,\gamma_-)$, by the
$\R$-direction first convention again.

Finally, we need to prove that the orientation of $T_{[u]}\M^\J(\gamma_+,\gamma_-)$ determined by an element of $\mc{O}(\mc{L}_+)\tensor\mc{O}(\mc{L}_-)$ depends continuously on $[u]$. As one varies $[u]$, if the dimension of $\Ker(D_u)$ does not jump, then the exact sequences used to define the orientation vary continuously, so the orientation does not change. In the general case one uses a stabilization argument to arrange that the dimensions of the kernels of the operators in question do not jump; cf.\ \cite[\S9.1]{ht2}.
\end{proof}

We can now define the orientation convention in equation \eqref{eqn:defdelta} in the definition of cylindrical contact homology.

\begin{definition}
\label{def:EGHsigns}
Let $J$ be a $\lambda$-compatible almost complex structure on $\R\times Y$. Let $\alpha$ and $\beta$ be good Reeb orbits, so that the local systems $\mc{O}_\alpha$ and $\mc{O}_\beta$ are trivial (and thus can be regarded as $\Z$-modules which are noncanonically isomorphic to $\Z$). Assume that the moduli space $\M^J_1(\alpha,\beta)$ is cut out transversely (and in particular is a discrete set). For each $u\in\M^J_1(\alpha,\beta)$, we define an isomorphism of $\Z$-modules
\begin{equation}
\label{eqn:epsilonu}
\epsilon(u): \mc{O}_\alpha \stackrel{\simeq}{\longrightarrow} \mc{O}_\beta
\end{equation}
as follows.

Let $\J$ be the constant family of almost complex structures $J_t\equiv J$. By definition,
\[
\M^J_d(\alpha,\beta) = \M^\J_d(\alpha,\beta)/S^1
\]
where $S^1$ acts on $\M^\J_d$ by reparametrization. Our convention is that $\varphi\in S^1=\R/\Z$ acts by
\[
(\varphi\cdot u)(s,t) = u(s, t + \varphi).
\]
Transversality of $\M^J_d$ is equivalent to transversality of $\M^\J_d$, see \S\ref{sec:admissible3}. By Proposition~\ref{prop:orientations}, when this transversality holds, the moduli space $\M^\J_d(\alpha,\beta)$ has an orientation with values in $\mc{O}_\alpha\tensor\mc{O}_\beta$.  The orientation of $\M^\J_d$ then induces an orientation of $\M^J_d$ by the ``$S^1$-direction first'' convention. When $d=1$, the latter orientation simply assigns to each $u\in \M^J_1$ a generator of $\mc{O}_\alpha\tensor\mc{O}_\beta$, which is equivalent to an isomorphism \eqref{eqn:epsilonu}.
\end{definition}

\subsection{Properties of the moduli space orientations}
\label{sec:gbs}

\paragraph{Scaling.}

We now prove the following lemma which is needed in \S\ref{sec:nchinv}.

\begin{lemma}
\label{lem:scop}
Let $c>0$.
\begin{description}
\item{(a)} If $\gamma$ is a Reeb orbit of $\lambda$, and if ${^c}\gamma$ denotes the corresponding Reeb orbit of $c\lambda$, then there is a canonical isomorphism of local systems $\mc{O}_{\gamma} = \mc{O}_{{^c}\gamma}$.
\item{(b)}
With respect to the isomorphism in (a), the scaling diffeomorphism \eqref{eqn:scop} is orientation preserving.
\end{description}
\end{lemma}

\begin{proof}
(a)
Let ${^c}\J$ be the family of $c\lambda$-compatible almost complex structures in \eqref{eqn:scop}. It follows from Definition~\ref{def:col} that if $p\in\overline{\gamma}=\overline{{^c}\gamma}$, then there is a canonical isomorphism $\mc{L}_{\gamma,p,\J} = \mc{L}_{{^c}\gamma,p,{^c}\J}$. Then by Definition~\ref{def:Ogamma} we obtain a canonical isomorphism $\mc{O}_{\gamma}(p) = \mc{O}_{{^c}\gamma}(p)$, which by the reasoning after Definition~\ref{def:Ogamma} depends continuously on $p$.

(b) Let $\phi$ be the diffeomorphism of $\R\times Y$ defined above \eqref{eqn:scop}. The diffeomorphism \eqref{eqn:scop} lifts to a diffeomorphism
\begin{equation}
\label{eqn:scoplift}
\widetilde{\widetilde{\M}}^\J(\gamma_+,\gamma_-) \stackrel{\simeq}{\longrightarrow} \widetilde{\widetilde{\M}}^{{^c}\J}({^c}\gamma_+,{^c}\gamma_-)
\end{equation}
sending $u\mapsto \phi\circ u$. We have a commutative diagram
\[
\begin{CD}
L^{2,\delta}_1(u^*T(\R\times Y)) @>{D_u}>> L^{2,\delta}(u^*T(\R\times Y))\\
@V{\simeq}VV @V{\simeq}VV \\
L^{2,\delta}_1((\phi\circ u)^*T(\R\times Y)) @>{D_{\phi\circ u}}>> L^{2,\delta}((\phi\circ u)^*T(\R\times Y)).
\end{CD}
\]
This induces isomorphisms $\Ker(D_u)\simeq \Ker(D_{\phi\circ u})$ and $\Coker(D_u)\simeq \Coker(D_{\phi\circ u})$, and thus a bijection $\mc{O}(D_u) \simeq \mc{O}(D_{\phi\circ u})$. It follows from the definitions that this bijection respects the canonical isomorphism in (a). On the other hand, these isomorphisms of kernels and cokernels and the exact sequences from Proposition~\ref{prop:tanx} fit into a commutative diagram
\[
\begin{CD}
0 @>>> \Ker(D_u) @>>> T_u\widetilde{\widetilde{\mc{M}}}^\J(\gamma_+,\gamma_-) @>{(\sigma_+,\sigma_-)}>> \bbC\oplus\bbC @>{(\tau_+,\tau_-)}>> \Coker(D_u) @>>> 0\\
& & @V{\simeq}VV @V{\simeq}VV @V{\simeq}V{c}V @V{\simeq}VV & &\\
0 @>>> \Ker(D_{\phi\circ u}) @>>> T_{\phi\circ u}\widetilde{\widetilde{\mc{M}}}^{{^c}\J}({^c}\gamma_+,{^c}\gamma_-) @>{(\sigma_+,\sigma_-)}>> \bbC\oplus\bbC @>{(\tau_+,\tau_-)}>> \Coker(D_{\phi\circ u}) @>>> 0.
\end{CD}
\]
Here the second vertical arrow is the derivative of the diffeomorphism \eqref{eqn:scoplift}, and the third vertical arrow is multiplication by $c$. It follows from this isomorphism of exact sequences that the second vertical arrow is orientation preserving. Thus the diffeomorphism \eqref{eqn:scoplift} is orientation preserving. After modding out by the two $\R$ actions, we conclude that the diffeomorphism \eqref{eqn:scop} is orientation preserving.
\end{proof}

\paragraph{Gluing and boundary signs.}

We now justify the signs that appear in Proposition~\ref{prop:currentcompact}. Continue to make the assumptions from the beginning of \S\ref{sec:oms}.

\begin{proposition}
\label{prop:gbs}
Let $\gamma_+$, $\gamma_0$, and $\gamma_-$ be distinct Reeb orbits. Let $d_+$ and $d_-$ be nonnegative integers and let $d=d_++d_-$. Let
\[
([u_+],[u_-])\in \M^\J_{d_+}(\gamma_+,\gamma_0)\times_{\overline{\gamma_0}}\M^\J_{d_-}(\gamma_0,\gamma_-).
\]
Then there is a neighborhood
\[
U \subset \overline{\M^\J_d}(\gamma_+,\gamma_-)
\]
of $([u_+],[u_-])$, which has the structure of a smooth manifold with boundary, and a neighborhood
\begin{equation}
\label{eqn:gbs}
V \subset (-1)^{d_+}\M^\J(\gamma_+,\gamma_0)\times_{\overline{\gamma_0}} \M^\J(\gamma_0,\gamma_-)
\end{equation}
of $([u_+],[u_-])$, with a canonical orientation preserving diffeomorphism
\[
\partial U \simeq V.
\]
\end{proposition}

\begin{proof}
Aside from the signs in \eqref{eqn:gbs}, this follows from standard gluing arguments. So we will sketch one approach to the gluing and justify the signs. 

To set the notation, regard $u_+$ as an element of $\widetilde{\widetilde{\M}}^\J(\gamma_+,\gamma_0)$, and regard $u_-$ as an element of $\widetilde{\widetilde{\M}}^\J(\gamma_0,\gamma_-)$. To simplify notation, we will just consider the case where $\Ker(D_{u_+})=0$ and $\Ker(D_{u_-})=0$; the general case can be handled similarly.
By Proposition~\ref{prop:tanx}, we have short exact sequences
\begin{gather}
\label{eqn:Tu+}
0 \longrightarrow T_{u_+}\widetilde{\widetilde{\mc{M}}}^\J(\gamma_+,\gamma_0) \stackrel{(\sigma^+_+,\sigma^+_0)}{\longrightarrow} \bbC_+ \oplus\bbC_0 \stackrel{(\tau^+_+,\tau^+_0)}{\longrightarrow} \Coker(D_{u_+}) \longrightarrow 0,\\
\label{eqn:Tu-}
0 \longrightarrow T_{u_-}\widetilde{\widetilde{\mc{M}}}^\J(\gamma_0,\gamma_-) \stackrel{(\sigma^-_0,\sigma^-_-)}{\longrightarrow} \bbC_0\oplus\bbC_- \stackrel{(\tau^-_0,\tau^-_-)}{\longrightarrow} \Coker(D_{u_-}) \longrightarrow 0.
\end{gather}
Here $\bbC_+$, $\bbC_0$, and $\bbC_-$ are copies of $\bbC$, which one should think of as being associated with $\gamma_+$, $\gamma_0$, and $\gamma_-$ respectively. Here and below, in the notation for the maps $\sigma$ etc.\, we use the convention that where we previously indicated a Reeb orbit with $+$ or $-$, now we indicate it $+$, $0$, or $-$.

\medskip
{\em Step 1.\/}
We now consider how to glue $u_+$ and $u_-$ to obtain elements of $\widetilde{\widetilde{M}}^\J(\gamma_+,\gamma_-)$. To start, translate $u_+$ in the target up by $\mc{R}>>0$ in the $\R$ direction, and translate $u_-$ in the target down by $\mc{R}$ in the $\R$ direction. On $u_+$, choose sections $\psi_1^{+,+}$, $\psi_2^{+,+}$, $\psi_1^{+,0}$, and $\psi_2^{+,0}$ as in \S\ref{sec:tms} whose first derivatives have order $O(\mc{R}^{-1})$. Likewise choose sections $\psi_1^{-,0}$, $\psi_2^{-,0}$, $\psi_1^{-,-}$, and $\psi_2^{-,-}$ on $u_-$.
Given small complex numbers $z_+=a_++ib_+$, $z_0=a_0+ib_0$, and $z_-=a_-+ib_-$, we can ``preglue'' $u_+$ and $u_-$ as follows: We replace $u_+$ by the exponential of $a_+\psi_1^{+,+} + b_+\psi_2^{+,+} + a_0\psi_1^{+,0} + b_0\psi_2^{+,0}$, we replace $u_-$ by the exponential of $a_0\psi_1^{-,0} + b_0\psi_2^{-,0} + a_-\psi_1^{-,-} + b_-\psi_2^{-,-}$, and then we patch using appropriate cutoff functions. As in \cite[\S5]{ht2}, one can perturb the preglued curve to a curve which solves the Cauchy-Riemann equation \eqref{eqn:cr}, up to an error in $\Coker(D_{u_+})\oplus\Coker(D_{u_-})$, which we denote by ${\mathfrak s}(z_+,z_0,z_-)$. We can package these errors for different choices of $z_+$, $z_0$, and $z_-$ into an ``obstruction section''
\[
{\mathfrak s}: N_+\oplus N_0 \oplus N_- \longrightarrow \Coker(D_{u_+}) \oplus \Coker(D_{u_-}).
\]
Here $N_+$, $N_0$, and $N_-$ are small neighborhoods of the origin in ${\mathbb C}$. As in \cite[Lem. 6.3]{ht2}, the obstruction section ${\mathbf s}$ is smooth; and as in \cite[\S10.5]{ht2}, its zero set is cut out transversely. The gluing construction then defines a ``gluing map''
\[
G: {\mathfrak s}^{-1}(0) \longrightarrow \widetilde{\widetilde{\M}}^\J(\gamma_+,\gamma_-).
\]

As in \cite[Thm.\ 7.3]{ht2}, the gluing map is a local diffeomorphism. In particular, if $u$ is in the image of the gluing map, then we obtain a short exact sequence
\begin{equation}
 \label{eqn:doses}
0 \longrightarrow T_u\widetilde{\widetilde{\M}}(\gamma_+,\gamma_-) \stackrel{dG^{-1}}{\longrightarrow} {\mathbb C}_+ \oplus {\mathbb C}_0 \oplus {\mathbb C}_- \stackrel{\nabla {\mathfrak s}}{\longrightarrow} \Coker(D_{u_+}) \oplus \Coker(D_{u_-}) \longrightarrow 0.
 \end{equation}
Moreover, an analogue of \cite[Lem.\ 10.5]{ht2} shows that if we choose $\frak{o}_\pm\in\mc{O}(\gamma_\pm)$ and $\frak{o}_0\in\mc{O}(\gamma_0)$, and use these to orient $\Coker(D_{u_\pm})$ and $T_u\widetilde{\widetilde{\M}}^\J(\gamma_+,\gamma_-)$, then the exact sequence \eqref{eqn:doses} is orientation preserving in the sense of \eqref{eqn:eso}.

\medskip
{\em Step 2.\/}
The obstruction section 
can be approximated, in a sense to be specified below, by a ``linearized section''
\[
{\mathfrak s}_0: N_+\oplus N_0 \oplus N_- \longrightarrow \Coker(D_{u_+}) \oplus \Coker(D_{u_-})
\]
defined by
\[
{\mathfrak s}_0(z_+,z_0,z_-) = (\tau^+_+(z_+) + \tau^+_0(z_0), \tau^-_0(z_0) + \tau^-_-(z_-)).
\]

The zero set of $\frak{s}_0$ is also cut out transversely. More precisely, define
\[
V = \left\{(\eta_+,\eta_-) \in T_{u_+}\widetilde{\widetilde{\M}}^J(\gamma_+,\gamma_0) \oplus T_{u_-}\widetilde{\widetilde{\M}}^\J(\gamma_0,\gamma_-) \bigg| \sigma^+_0(\eta_+) = \sigma^-_0(\eta_-) \right\}.
\]
Here the right hand side is oriented as a level set of the linear map
\begin{equation}
\label{eqn:olsf}
\sigma^+_0 - \sigma^-_0 : T_{u_+}\widetilde{\widetilde{\M}}^J(\gamma_+,\gamma_0) \oplus T_{u_-}\widetilde{\widetilde{\M}}^\J(\gamma_0,\gamma_-) \longrightarrow {\mathbb C}_0.
\end{equation}
This makes sense because we are assuming that $\J$ is generic so that the fiber product
\[
\M^\J(\gamma_+,\gamma_0)\times_{\overline{\gamma_0}}\M^\J(\gamma_0,\gamma_-)
\]
is cut out transversely, which means that the linear map \eqref{eqn:olsf} is surjective. It also follows from this surjectivity and the short exact sequences \eqref{eqn:Tu+} and \eqref{eqn:Tu-} that we have a short exact sequence
\begin{equation}
\label{eqn:laes}
0 \longrightarrow V \stackrel{(\sigma^+_+,\sigma^+_0=\sigma^-_0,\sigma^-_-)}{\longrightarrow} \bbC_+ \oplus \bbC_0 \oplus \bbC_- \stackrel{\nabla\frak{s}_0}{\longrightarrow} \Coker(D_+)\oplus\Coker(D_-)\longrightarrow 0.
\end{equation}
Here of course, $\nabla\frak{s}_0$ is defined by the same formula as $\frak{s}_0$ since $\frak{s}_0$ is linear. Moreover, the exact sequence \eqref{eqn:laes} is orientation preserving in the sense of \eqref{eqn:eso}.

One can argue as in \cite[Cor.\ 8.6]{ht2} that if $\mc{R}$ is sufficiently large, then $\frak{s}^{-1}(0)$ is $C^1$ close to $\frak{s}_0^{-1}(0)$. In particular, if $u\in\widetilde{\widetilde{\M}}^J(\gamma_+,\gamma_-)$ is in the image of the gluing map, then by comparing the exact sequences \eqref{eqn:doses} and \eqref{eqn:laes}, we obtain, up to homotopy, a canonical orientation-preserving isomorphism
\begin{equation}
\label{eqn:okp}
T_u\widetilde{\widetilde{\M}}^\J(\gamma_+,\gamma_-) \simeq V.
\end{equation}

\medskip
{\em Step 3.\/} We now use \eqref{eqn:okp} to justify the signs in \eqref{eqn:gbs}.

First note that in the exact sequence \eqref{eqn:Tu+}, the tangent vector to $\widetilde{\widetilde{\M}}^\J(\gamma_+,\gamma_0)$ corresponding to the derivative of the $\R$ action (by translation of the domain) has $\sigma^+_+=\mc{A}(\gamma_+)$ and $\sigma^+_0=\mc{A}(\gamma_0)$, where $\mc{A}$ denotes the symplectic action (period) of a Reeb orbit. Thus projection defines an orientation preserving isomorphism
\begin{equation}
\label{eqn:tu+op}
T_{[u_+]}\widetilde{\M}^\J(\gamma_+,\gamma_0) \simeq(\op{Re}\sigma^+_0)^{-1}(0) \subset T_{u_+}\widetilde{\widetilde{\M}}^\J(\gamma_+,\gamma_0).
\end{equation}
Likewise, we have an orientation preserving isomorphism
\begin{equation}
\label{eqn:tu-op}
T_{[u_-]}\widetilde{\M}^\J(\gamma_0,\gamma_-) \simeq (\op{Re}\sigma^-_0)^{-1}(0) \subset T_{u_-}\widetilde{\widetilde{\M}}^\J(\gamma_0,\gamma_-).
\end{equation}
Combining the above two isomorphisms, we obtain an orientation preserving isomorphism
\begin{equation}
\label{eqn:tu+u-op}
T_{[u_+]}\widetilde{\M} \oplus T_{[u_-]}\widetilde{\M} \simeq (-1)^{d_++1}(\op{Re}\sigma^+_0\times \op{Re}\sigma^-_0)^{-1}(0,0) \subset T_{u_+}\widetilde{\widetilde{\M}}\oplus T_{u_-}\widetilde{\widetilde{\M}}.
\end{equation}

On the other hand, we can rewrite \eqref{eqn:okp} as
\begin{equation}
\label{eqn:okprw}
T_u\widetilde{\widetilde{\M}} \simeq ((\op{Re}\sigma^+_0 - \op{Re}\sigma^-_0) \times (\op{Im}\sigma^+_0 - \op{Im}\sigma^-_0))^{-1}(0,0) \subset T_{u_+}\widetilde{\widetilde{\M}} \oplus T_{u_-}\widetilde{\widetilde{\M}}.
\end{equation}
In this isomorphism, the tangent vector to $T_u\widetilde{\widetilde{\M}}$ corresponding to the derivative of the $\R$ action (by translation of the domain) has $\sigma^+_0$ and $\sigma^-_0$ close to $\mc{A}(\gamma_0)$.  It follows that we have an orientation preserving isomorphism
\begin{equation}
\label{eqn:okprwr}
T_{[u]}\widetilde{\M} \simeq (\op{Re}\sigma^+_0 \times \op{Re}\sigma^-_0 \times (\op{Im}\sigma^+_0 - \op{Im}\sigma^-_0))^{-1}(0,0,0) \subset T_{u_+}\widetilde{\widetilde{\M}} \oplus T_{u_-}\widetilde{\widetilde{\M}}.
\end{equation}
Comparing \eqref{eqn:tu+u-op} with \eqref{eqn:okprwr}, we obtain an orientation preserving isomorphism
\begin{equation}
\label{eqn:nwgsw}
T_{[u]}\widetilde{\M} \simeq (-1)^{d_++1}(\op{Im}\sigma^+_0 - \op{Im}\sigma^-_0)^{-1}(0) \subset T_{[u_+]}\widetilde{\M} \oplus T_{[u_-]}\widetilde{\M}.
\end{equation}

In the isomorphism \eqref{eqn:nwgsw}, the derivative of the $\R$ action (by translation of the target) on the left hand side corresponds to the direct sum of the derivative of the $\R$ action on both summands on the right hand side. On the other hand, when $\mc{R}>0$ is large, the direction pointing ``out of the boundary'' of the left hand side corresponds to the derivative of the $\R$ action on the first summand on the right hand side, minus the derivative of the $\R$ action on the second summand. It follows that there is a neighborhood $U$ of $([u_+],[u_-])$ in $\overline{\M}^\J_d(\gamma_+,\gamma_-)$ and a neighborhood
\begin{equation}
\label{eqn:locdiff}
V \subset -(e^+_0 - e^-_0)^{-1}(0) \subset \M^\J(\gamma_+,\gamma_0) \times \M^\J(\gamma_0,\gamma_-)
\end{equation}
of $([u_+],[u_-])$ with a canonical orientation preserving diffeomorphism $\partial U\simeq V$.
Here
\[
\begin{split}
e^+_0 : \M^\J(\gamma_+,\gamma_0) &\longrightarrow \overline{\gamma_0},\\
e^-_0: \M^\J(\gamma_0,\gamma_-)  &\longrightarrow \overline{\gamma_0}
\end{split}
\]
denote the evaluation maps. According to the convention from \cite[\S2.1]{td} which we are using to orient fiber products, the middle of \eqref{eqn:locdiff} agrees with the right hand side of \eqref{eqn:gbs}.
\end{proof}

\subsection{Cobordism orientations}
\label{sec:cobsigns}

We now explain the orientations of the cobordism moduli spaces in \S\ref{sec:nchcobmaps}.

Using the notation of \S\ref{sec:nchcobmaps}, let $\widetilde{\Phi}^\J(\gamma_+,\gamma_-)$ denote the moduli space of maps $u:R\times S^1\to\overline{X}$ satisfying the conditions \eqref{eqn:cobmoduli}--\eqref{eqn:cobmoduliend}, but without modding out by $\R$ translation in the domain. Thus
\[
\Phi^\J(\gamma_+,\gamma_-) = \widetilde{\Phi}^\J(\gamma_+,\gamma_-)/\R,
\]
where $\R$ acts by translating the domain as in \eqref{eqn:ractiondomain}.

Given $u\in\widetilde{\Phi}^\J(\gamma_+,\gamma_-)$, we have maps $\tau_\pm:\bbC\to\Coker(D_u)$ as in \eqref{eqn:taumaps}. This definition still make sense, where now $\psi_1^\pm$ and $\psi_2^\pm$ are sections of $u^*T\overline{X}$, since $u$ maps to $\R\times Y_+$ for $s>>0$ and to $\R\times Y_-$ for $s<<0$. In this situation we have the following analogue of Proposition~\ref{prop:tanx}, with the same proof:

\begin{proposition}
\label{prop:tanx2}
If $\J$ is generic, then $\widetilde{\Phi}^\J(\gamma_+,\gamma_-)$ is naturally a smooth manifold, and there is a canonical exact sequence
\begin{equation}
0 \longrightarrow \Ker(D_u) \longrightarrow T_u\widetilde{\Phi}^\J(\gamma_+,\gamma_-) \stackrel{(\sigma_+,\sigma_-)}{\longrightarrow} \bbC\oplus\bbC \stackrel{(\tau_+,\tau_-)}{\longrightarrow} \Coker(D_u) \longrightarrow 0.
\end{equation}
\end{proposition}

Given Proposition~\ref{prop:tanx2}, the argument at the end of \S\ref{sec:oms} shows that $\widetilde{\Phi}^\J(\gamma_+,\gamma_-)$ has a canonical orientation with values in $e_+^*\mc{O}_{\gamma_+}\tensor e_-^*\mc{O}_-$. This then determines an orientation of $\Phi^\J(\gamma_+,\gamma_-)$, with values in the same local system, by the ``$\R$ direction first'' convention.

We now have the following analogue of Proposition~\ref{prop:gbs}.

\begin{proposition}
\label{prop:cobsigns}
With the notation and hypotheses of Lemma~\ref{lem:morphism2}:
\begin{description}
	\item{(a)} Let
	\[
	([u_+],[u_0]) \in \M^{\J_+}_{d_+}(\gamma_+,\gamma_0)\times_{\overline{\gamma_0}} \Phi^\J_{d_0}(\gamma_0,\gamma_-).
	\]
	Then there is a neighborhood
	\[
	U\subset\overline{\Phi}^\J_{d_++d_0}(\gamma_+,\gamma_-)
	\]
	of $([u_+],[u_0])$, which is naturally a smooth manifold with boundary, and a neighborhood
	\begin{equation}
	\label{eqn:gbsa}
	V\subset \M^{\J_+}_{d_+}(\gamma_+,\gamma_0)\times_{\overline{\gamma_0}} \Phi^\J_{d_0}(\gamma_0,\gamma_-)
	\end{equation}
	of $([u_+],[u_0])$, with a canonical orientation preserving diffeomorphism $\partial U \simeq V$.
	\item{(b)}
	Let
	\[
	([u_0],[u_-])\in \Phi^\J_{d_0}(\gamma_+,\gamma_0)\times_{\overline{\gamma_0}} \M^{\J_-}_{d_-}(\gamma_0,\gamma_-).
	\]
	Then there is a neighborhood
	\[
	U\subset\overline{\Phi}^\J_{d_0+d_-}(\gamma_+,\gamma_-)
	\]
	of $([u_0],[u_-])$, which is naturally a smooth manifold with boundary, and a neighborhood
	\begin{equation}
	\label{eqn:gbsb}
	V \subset (-1)^{d_0} \Phi^\J_{d_0}(\gamma_+,\gamma_0)\times_{\overline{\gamma_0}} \M^{\J_-}_{d_-}(\gamma_0,\gamma_-)
	\end{equation}
	of $([u_0],[u_-])$, with a canonical orientation preserving diffeomorphism $\partial U \simeq V$.
\end{description}
\end{proposition}

\begin{proof}
This follows the proof of Proposition~\ref{prop:gbs}, with minor modifications. There are some sign changes at the very end due to the fact that the cobordism moduli spaces $\widetilde{\Phi}^\J$ do not have $\R$ actions by translation in the target. These sign changes work as follows.

(a) Up to \eqref{eqn:nwgsw}, the proof follows the proof of Proposition~\ref{prop:gbs}, with $\widetilde{\widetilde{\M}}^\J(\gamma_+,\gamma_0)$ replaced by $\widetilde{\widetilde{\M}}^{\J_+}(\gamma_+,\gamma_0)$, and with $\widetilde{\widetilde{\M}}^\J(\gamma_0,\gamma_-)$ replaced by $\widetilde{\Phi}^\J(\gamma_0,\gamma_-)$. The analogue of \eqref{eqn:nwgsw} here is
\begin{equation}
\label{eqn:nwgswa}
T_{[u]}\Phi^\J(\gamma_+,\gamma_-) \simeq (-1)^{d_++1}(\op{Im}\sigma^+_0-\op{Im}\sigma^0_0)^{-1}(0) \subset T_{[u_+]}\widetilde{\M}^{\J_+}(\gamma_+,\gamma_0)\oplus T_{u_0}\Phi^\J(\gamma_0,\gamma_-).
\end{equation}

In the isomorphism \eqref{eqn:nwgswa}, the direction pointing ``out of the boundary'' on the left hand side corrresponds to the derivative of the $\R$ action on the first summand on the right hand side. It follows that there is a neighborhood $U$ of $([u_+],[u_0])$ in $\overline{\Phi}^\J_{d_++d_0}(\gamma_+,\gamma_-)$ and a neighborhood
\begin{equation}
\label{eqn:locdiffa}
V \subset (-1)^{d_++1}(e^+_0 - e^0_0)^{-1}(0) \subset \M^{\J_+}_{d_+}(\gamma_+,\gamma_0) \times \Phi^\J_{d_0}(\gamma_0,\gamma_-)
\end{equation}
of $([u_+],[u_0])$ with a canonical orientation preserving diffeomorphism $\partial U\simeq V$.
Here
\[
\begin{split}
e^+_0 : \M^{\J_+}_{d_+}(\gamma_+,\gamma_0) &\longrightarrow \overline{\gamma_0},\\
e^0_0: \Phi^\J_{d_0}(\gamma_0,\gamma_-)  &\longrightarrow \overline{\gamma_0}
\end{split}
\]
denote the evaluation maps. According to the fiber product orientation convention from \cite[\S2.1]{td}, this means that the middle of \eqref{eqn:locdiffa} agrees with the right hand side of \eqref{eqn:gbsa}.

(b) In this case the analogue of \eqref{eqn:nwgsw} is
\begin{equation}
\label{eqn:nwgswb}
T_{[u]}\Phi^\J(\gamma_+,\gamma_-) \simeq (-1)^{d_0+1}(\op{Im}\sigma^0_0 - \op{Im}\sigma^-_0)^{-1}(0) \subset T_u\Phi^\J(\gamma_+,\gamma_0) \oplus T_{[u_-]}\widetilde{\M}^{\J_-}(\gamma_0,\gamma_-).
\end{equation}
In the isomorphism \eqref{eqn:nwgswb}, the direction pointing ``out of the boundary'' on the left hand side corresponds to minus the derivative of the $\R$ action on the second summand on the right hand side. Because of this minus sign, and because the first summand on the right hand side has dimension $d_0$, this results in an extra factor of $(-1)^{d_0}$ in \eqref{eqn:gbsb} as compared with \eqref{eqn:gbsa}.
\end{proof}

We can now prove the following lemma which is needed in \S\ref{sec:scaling}.

\begin{lemma}
\label{lem:scacob}
Under the identifications in Lemma~\ref{lem:scop}(a):
\begin{description}
\item{(a)} The diffeomorphism \eqref{eqn:scacob1} is orientation preserving.
\item{(b)} The diffeomorphism \eqref{eqn:scacob2} is orientation preserving with respect to the orientation of $\overline{\gamma}$ given by the Reeb vector field.
\end{description} 
\end{lemma}

\begin{proof}
(a)
The proof of Lemma~\ref{lem:scop}(b) shows that the diffeomorphism $\phi$ in \S\ref{sec:scaling} induces an orientation preserving isomorphism
\[
\widetilde{\Phi}^{\J^X}\left(^{{e^b}}\gamma_+,{^{e^a}}\gamma_-\right) \stackrel{\simeq}{\longrightarrow} \widetilde{\widetilde{\M}}^\J(\gamma_+,\gamma_-)
\]
sending $u\mapsto \phi\circ u$. By our convention at the end of \S\ref{sec:oms}, it follows that the diffeomorphism \eqref{eqn:scacob1} is orientation preserving.

(b)
For a single Reeb orbit $\gamma$ of $\lambda$, composition with $\phi$ likewise gives an orientation preserving diffeomorphism
\[
\widetilde{\Phi}^{\J^X}\left(^{{e^b}}\gamma,{^{e^a}}\gamma\right) \stackrel{\simeq}{\longrightarrow} \widetilde{\widetilde{\M}}^\J(\gamma,\gamma)
\]
By symplectic action considerations, the elements of the right hand side consist of holomorphic maps
\[
u:\R\times S^1 \to \R\times\overline{\gamma}
\]
of degree $d(\gamma)$. We then obtain an orientation preserving diffeomorphism
\[
\widetilde{\widetilde{\M}}^\J(\gamma,\gamma) \stackrel{\simeq}{\longrightarrow} \R\times\overline{\gamma}
\]
sending $u\mapsto u(0,0)$. By our convention at the end of \S\ref{sec:oms}, it follows that the diffeomorphism \eqref{eqn:scacob2} is orientation preserving.
\end{proof}

\subsection{Family orientations}
\label{sec:familysigns}

We now explain the orientations of the $S^1$-equivariant moduli spaces defined in \S\ref{sec:familymoduli}, and we use the notation from that section.

\begin{proof}[Proof of Proposition~\ref{prop:famor}.]
Let $x$ be a critical point of $f:BS^1\to\R$ and let $\gamma$ be a Reeb orbit.
We first define the local system $\mc{O}_{(x,\gamma)}$.

Recall from Proposition~\ref{prop:orientations} that there is a canonical local system $\mc{O}_\gamma$ over $\overline{\gamma}$. Let $\widetilde{\mc{O}}$ denote the pullback of this local system to $\pi^{-1}(x)\times\overline{\gamma}$ via the projection
\[
\pi^{-1}(x)\times\overline{\gamma} \longrightarrow \overline{\gamma}.
\]

We claim that the restriction of the local system $\widetilde{\mc{O}}$ to each fiber of the $S^1$ action \eqref{eqn:s1actm} on $\pi^{-1}(x)\times\overline{\gamma}$ is trivial. If the Reeb orbit $\gamma$ is good, then by Proposition~\ref{prop:orientations}(a), the local system $\widetilde{\mc{O}}$ is trivial over all of $\pi^{-1}(x)\times\overline{\gamma}$. If the Reeb orbit $\gamma$ is bad, then the local system $\widetilde{\mc{O}}$ is nontrivial over each circle $\{z\}\times\overline{\gamma}$, and trivial over each circle $\pi^{-1}(x)\times\{p\}$. But for a bad Reeb orbit $\gamma$, the covering multiplicity $d(\gamma)\in\Z$ is even, so each orbit of the $S^1$ action \eqref{eqn:s1actm} wraps an even number of times around the $\overline{\gamma}$ direction, and the restriction of $\overline{\mc{O}}$ to the orbit is still trivial.

It follows from the previous paragraph that $\overline{\mc{O}}$ descends to a well defined local system over $\overline{(x,\gamma)}$, which we denote by $\mc{O}_{(x,\gamma)}$.

Assertion (a) in Proposition~\ref{prop:famor} now follows from the above discussion.

To prove assertion (b), let $\widetilde{\widetilde{\M}}^{\mathfrak J}((x_+,\gamma_+),(x_-,\gamma_-))$ denote the quotient of the moduli space $\widehat{\M}^{\mathfrak J}((x_+,\gamma_+),(x_-,\gamma_-))$ by the $S^1$ action \eqref{eqn:s1actm}. Thus
\[
\widetilde{\M}^{\mathfrak J}((x_+,\gamma_+),(x_-,\gamma_-)) = \widetilde{\widetilde{\M}}^{\mathfrak J}((x_+,\gamma_+),(x_-,\gamma_-))/\R
\]
where $\R$ acts by translation of the parameter $s$.

Let $(\eta,u)\in\widehat{\M}^{\mathfrak J}$. We can then describe the tangent space $T_{[(\eta,u)]}\widetilde{\widetilde{\M}}^{\mathfrak J}$ as follows. Let $\widetilde{\M}^{\op{Morse}}(x_+,x_-)$ denote the moduli space of parametrized flow lines of $V$ asymptotic to $x_+$ and $x_-$. Let $\widetilde{\M}^{\op{Morse}}(\pi^{-1}(x_+),\pi^{-1}(x_-))$ denote the moduli space of parametrized flow lines of  $\widetilde{V}$ asymptotic to points in $\pi^{-1}(x_+)$ and $\pi^{-1}(x_-)$. Thus
\[
\widetilde{\M}^{\op{Morse}}(x_+,x_-) = \widetilde{\M}^{\op{Morse}}(\pi^{-1}(x_+),\pi^{-1}(x_-))/S^1.
\]
It follows from equation \eqref{eqn:efl} that $\widetilde{\M}^{\op{Morse}}(x_+,x_-)$ is identified with a complex linear subspace of a projective space (minus the points $x_+$ and $x_-$ when these are distinct) via the map sending $\eta\mapsto\eta(0)$. Thus $\widetilde{\M}^{\op{Morse}}(x_+,x_-)$ has a canonical complex orientation.

Let
\[
W\subset T_\eta\widetilde{\M}^{\op{Morse}}(\pi^{-1}(x_+),\pi^{-1}(x_-))
\]
be a lift of $T_{\pi\circ\eta}\widetilde{\M}^{\op{Morse}}(x_+,x_-)$. The derivative of the parametrized Cauchy-Riemann equation \eqref{eqn:familycr} at $(\eta,u)$ is described by an operator
\begin{equation}
\label{eqn:famreg}
W\oplus L^{p,\delta}_1(u^*T(\R\times Y)) \oplus{\mathbb C}\oplus{\mathbb C} \longrightarrow L^{p,\delta}(u^*T(\R\times Y)).
\end{equation}
This operator is defined analogously to \eqref{eqn:dersec}, with the domain extended by $W$ to allow for $\eta$ to move in its moduli space of Morse flow lines. As in Proposition~\ref{prop:tanx}, if ${\mathfrak J}$ is generic, then for each $(\eta,u)$ the operator \eqref{eqn:famreg} is surjective, which implies that the moduli space $\widetilde{\widetilde{\M}}^{\mathfrak J}((x_+,\gamma_+), (x_-,\gamma_-))$ is naturally a smooth manifold whose tangent space is the kernel of \eqref{eqn:famreg}. There is then a canonical exact sequence
\begin{equation}
\label{eqn:fces}
0 \longrightarrow \Ker(D_u) \longrightarrow T_{[(\eta,u)]}\widetilde{\widetilde{\mc{M}}}^{\mathfrak J}((x_+,\gamma_+),(x_-,\gamma_-)) {\longrightarrow} W \oplus \bbC\oplus\bbC {\longrightarrow} \Coker(D_u) \longrightarrow 0.
\end{equation}
Here $D_u$ as in \eqref{eqn:lincr} is now the derivative of the parametrized Cauchy-Riemann equation \eqref{eqn:familycr} with respect to deformations of $u$ alone.

Since $W$ has a canonical orientation as noted above, it follows that the exact sequence \eqref{eqn:fces} determines an orientation of $\widetilde{\widetilde{\M}}^{\mathfrak J}$ with values in $e_+^*\mc{O}_{(x_+,\gamma_+)}\tensor e_-^*\mc{O}_{(x_-,\gamma_-)}$, by the same argument as used in \S\ref{sec:oms} to prove Proposition~\ref{prop:orientations}(b). We then orient $\widetilde{\M}^{\mathfrak J}=\widetilde{\widetilde{\M}}^{\mathfrak J}/\R$ using the ``$\R$-direction first'' convention, and finally orient $\M^\frak{J}=\widetilde{\M}^\frak{J}$ using the ``$\R$-direction first'' convention again.
\end{proof}

Analogously to Definition~\ref{def:regular}, we formulate:

\begin{definition}
\label{def:famreg}
A pair $(\eta,u)\in \widetilde{\widetilde{\mc{M}}}^{\mathfrak J}((x_+,\gamma_+),(x_-,\gamma_-))$, or its equivalence class in $\mc{M}^{\mathfrak J}((x_+,\gamma_+),(x_-,\gamma_-))$, is {\em regular\/} if the operator \eqref{eqn:famreg} is surjective.
\end{definition}

We can now prove the following lemma which is needed in \S\ref{sec:chaut}. To prepare to state the lemma, let $x\in\op{Crit}(f)$ and let $z$ be a lift of $x$ to $ES^1$. Given a Reeb orbit $\gamma$, define a diffeomorphism
\[
\rho_z:\overline{\gamma} \stackrel{\simeq}{\longmapsto} \overline{(x,\gamma)}
\]
by sending $p\mapsto [(z,p)]$. By the definitions of the local systems $\mc{O}_\gamma$ and $\mc{O}_{(x,\gamma)}$, the diffeomorphism $\rho_z$ has a canonical lift\footnote{If $z'$ is obtained from $z$ by rotating $1/d(\gamma)$ around the circle, then the diffeomorphisms $\rho_{z'}$ and $\rho_{z}$ agree; but the isomorphisms $\widetilde{\rho_z}$ and $\widetilde{\rho_{z'}}$ agree if and only if $\gamma$ is good.} to an isomorphism
\begin{equation}
\label{eqn:clift}
\widetilde{\rho_z}: \mc{O}_\gamma \stackrel{\simeq}{\longrightarrow} \mc{O}_{(x,\gamma)}.
\end{equation}

\begin{lemma}
\label{lem:chautor}
 Let $\eta$ be the constant flow line of $\widetilde{V}$ from $z$ to itself. Given ${\mathfrak J}$, define $\J=\{J_t\}$ by $J_t={\mathfrak J}_{t,z}$. Suppose that $\J$ is generic in the sense of Proposition~\ref{prop:transversality}. Let $\gamma_+,\gamma_-$ be distinct Reeb orbits. Then the diffeomorphism
 \begin{equation}
 \label{eqn:eqor}
\M^\J_d(\gamma_+,\gamma_-) \stackrel{\simeq}{\longrightarrow} \M_d^{\mathfrak J}((x,\gamma_+),(x,\gamma_-))
 \end{equation}
sending $[u]\mapsto [(\eta,u)]$ is orientation preserving with respect to \eqref{eqn:clift}.
\end{lemma}

\begin{proof}
We have a diffeomorphism
\begin{equation}
 \label{eqn:eqor2}
\widetilde{\widetilde{\M}}^\J_d(\gamma_+,\gamma_-) \stackrel{\simeq}{\longrightarrow} \widetilde{\widetilde{\M}}_d^{\mathfrak J}((x,\gamma_+),(x,\gamma_-))
 \end{equation}
sending $u\mapsto [(\eta,u)]$. Under the diffeomorphism \eqref{eqn:eqor2}, the exact sequences \eqref{eqn:ces1} and \eqref{eqn:fces} used to orient its two sides agree, since $W$ here is a $0$-dimensional vector space oriented positively. Finally, the same convention is used to pass from the orientations of the two sides of \eqref{eqn:eqor2} to the orientations of the two sides of \eqref{eqn:eqor}.
\end{proof}




\noindent \textsc{Michael Hutchings \\  University of California at Berkeley}\\
{\em email: }\texttt{hutching@math.berkeley.edu}\\

\noindent \textsc{Jo Nelson \\  Rice University}\\
{\em email: }\texttt{jo.nelson@rice.edu}\\

\end{document}